\documentclass[leqno,11pt,twoside]{amsart}

\DeclareRobustCommand{\SkipTocEntry}[4]{}

\makeatletter
\def\@tocline#1#2#3#4#5#6#7{\relax
  \ifnum #1>\c@tocdepth 
  \else
    \par \addpenalty\@secpenalty\addvspace{#2}%
    \begingroup \hyphenpenalty\@M
    \@ifempty{#4}{%
      \@tempdima\csname r@tocindent\number#1\endcsname\relax
    }{%
      \@tempdima#4\relax
    }%
    \parindent\z@ \leftskip#3\relax \advance\leftskip\@tempdima\relax
    \rightskip\@pnumwidth plus4em \parfillskip-\@pnumwidth
    #5\leavevmode\hskip-\@tempdima
      \ifcase #1
       \or\or \hskip 1em \or \hskip 2em \else \hskip 3em \fi%
      #6\nobreak\relax
    \dotfill\hbox to\@pnumwidth{\@tocpagenum{#7}}\par
    \nobreak
    \endgroup
  \fi}
\makeatother

\usepackage{amsmath}
\usepackage{amsfonts}
\usepackage{amssymb}
\usepackage{amsthm}
\usepackage{color}
\usepackage[all]{xy}
\usepackage{url}
\usepackage{stmaryrd}

\usepackage{hyperref}

\theoremstyle{definition}
\newtheorem{defn}{Definition}[section]
\newtheorem{ex}[defn]{Example}
\newtheorem{exercise}[defn]{Exercise}
\newtheorem{rmk}[defn]{Remark}
\newtheorem{question}[defn]{Question}
\theoremstyle{plain}
\newtheorem{thm}[defn]{Theorem}
\newtheorem{lem}[defn]{Lemma}
\newtheorem{prop}[defn]{Proposition}
\newtheorem{cor}[defn]{Corollary}

\def\C{\ensuremath{\mathbb{C}}}
\def\G{\ensuremath{\mathbb{G}}}
\def\H{\ensuremath{\mathbb{H}}}
\def\N{\ensuremath{\mathbb{N}}}
\def\P{\ensuremath{\mathbb{P}}}
\def\Q{\ensuremath{\mathbb{Q}}}
\def\R{\ensuremath{\mathbb{R}}}
\def\Z{\ensuremath{\mathbb{Z}}}

\def\AA{\ensuremath{\mathcal A}}
\def\BB{\ensuremath{\mathcal B}}
\def\CC{\ensuremath{\mathcal C}}
\def\DD{\ensuremath{\mathcal D}}
\def\EE{\ensuremath{\mathcal E}}
\def\FF{\ensuremath{\mathcal F}}

\def\HH{\ensuremath{\mathcal H}}
\def\II{\ensuremath{\mathcal I}}

\def\LL{\ensuremath{\mathcal L}}

\def\OO{\ensuremath{\mathcal O}}
\def\PP{\ensuremath{\mathcal P}}

\def\TT{\ensuremath{\mathcal T}}
\def\UU{\ensuremath{\mathcal U}}

\def\WW{\ensuremath{\mathcal W}}

\def\ZZ{\ensuremath{\mathcal Z}}

\def\Amp{\mathop{\mathrm{Amp}}\nolimits}

\def\Aut{\mathop{\mathrm{Aut}}\nolimits}

\def\ch{\mathop{\mathrm{ch}}\nolimits}

\def\Coh{\mathop{\mathrm{Coh}}\nolimits}

\def\Db{\mathop{\mathrm{D}^{\mathrm{b}}}\nolimits}

\def\dim{\mathop{\mathrm{dim}}\nolimits}

\def\ev{\mathop{\mathrm{ev}}\nolimits}
\def\inf{\mathop{\mathrm{inf}}\nolimits}

\def\End{\mathop{\mathrm{End}}\nolimits}

\def\ext{\mathop{\mathrm{ext}}\nolimits}
\def\Ext{\mathop{\mathrm{Ext}}\nolimits}

\def\gcd{\mathop{{\mathrm{gcd}}}\nolimits}
\def\GL{\mathop{\mathrm{GL}}\nolimits}

\def\Hom{\mathop{\mathrm{Hom}}\nolimits}

\def\id{\mathop{\mathrm{id}}\nolimits}

\def\im{\mathop{\mathrm{im}}\nolimits}

\def\Ker{\mathop{\mathrm{Ker}}\nolimits}
\def\Kom{\mathop{\mathrm{Kom}}\nolimits}
\def\Cok{\mathop{\mathrm{Cok}}\nolimits}

\def\mod{\mathop{\mathrm{mod}}\nolimits}
\def\min{\mathop{\mathrm{min}}\nolimits}

\def\num{\mathop{\mathrm{num}}\nolimits}

\def\NS{\mathop{\mathrm{NS}}\nolimits}

\def\perf{\mathop{\mathrm{perf}}\nolimits}
\def\PGL{\mathop{\mathrm{PGL}}}
\def\PSL{\mathop{\mathrm{PSL}}}
\def\Pic{\mathop{\mathrm{Pic}}\nolimits}

\def\Qcoh{\mathop{\mathrm{Qcoh}}\nolimits}

\def\rk{\mathop{\mathrm{rk}}}
\def\Sch{\mathop{\underline{\mathrm{Sch}}}}
\def\Set{\mathop{\underline{\mathrm{Set}}}}

\def\Sym{\mathop{\mathrm{Sym}}\nolimits}

\def\td{\mathop{\mathrm{td}}\nolimits}

\def\Stab{\mathop{\mathrm{Stab}}\nolimits}

\def\into{\ensuremath{\hookrightarrow}}
\def\onto{\ensuremath{\twoheadrightarrow}}

\begin{document}

\title{Lectures on Bridgeland Stability}

\author{Emanuele Macr\`i}
\address{Northeastern University, Department of Mathematics, 360 Huntington Avenue, Boston, MA 02115-5000, USA}
\curraddr{Universit\'e Paris-Sud, D\'epartement de Math\'ematiques, Rue Michel Magat, B\^at. 307, 91405 Orsay, France}
\email{emanuele.macri@u-psud.fr}
\urladdr{https://www.math.u-psud.fr/~macri/}

\author{Benjamin Schmidt}
\address{The Ohio State University, Department of Mathematics, 231 W 18th Avenue, Columbus, OH 43210-1174, USA}
\curraddr{Gottfried Wilhelm Leibniz Universit\"at Hannover, Institut f\"ur Algebraische Geometrie, Welfengarten 1, 30167 Hannover, Germany}
\email{bschmidt@math.uni-hannover.de}
\urladdr{https://sites.google.com/site/benjaminschmidtmath/}



\begin{abstract}
In these lecture notes we give an introduction to Bridgeland stability conditions on smooth complex projective varieties with a particular focus on the case of surfaces. This includes basic definitions of stability conditions on derived categories, basics on moduli spaces of stable objects and variation of stability. These notes originated from lecture series by the first author at the summer school \emph{Recent advances in algebraic and arithmetic geometry}, Siena, Italy, August 24-28, 2015 and at the school \emph{Moduli of Curves}, CIMAT, Guanajuato, Mexico, February 22 - March 4, 2016.
\end{abstract}

\maketitle
\tableofcontents

\section{Introduction}

Stability conditions in derived categories provide a framework for the study of moduli spaces of complexes of sheaves. They have been introduced in \cite{Bri07:stability_conditions}, with inspiration from work in string theory \cite{Dou02:mirror_symmetry}. It turns out that the theory has far further reach. A non exhaustive list of influenced areas is given by counting invariants, geometry of moduli spaces of sheaves, representation theory, homological mirror symmetry, and classical algebraic geometry.
This article will focus on the basic theory of Bridgeland stability on smooth projective varieties and give some applications to the geometry of moduli spaces sheaves. We pay particular attention to the case of complex surfaces.


\medskip

\noindent
{\bf Stability on curves.}
The theory starts with vector bundles on curves.
We give an overview of the classical theory in Section \ref{sec:curves}.
Let $C$ be a smooth projective curve.
In order to obtain a well behaved moduli space one has to restrict oneself to so called \emph{semistable} vector bundles. Any vector bundle $E$ has a \emph{slope} defined as $\mu(E) = \tfrac{d(E)}{r(E)}$, where $d(E)$ is the \emph{degree} and $r(E)$ is the \emph{rank}. It is called semistable if for all sub-bundles $F \subset E$ the inequality $\mu(F) \leq \mu(E)$ holds.

The key properties of this notion that one wants to generalize to higher dimensions are the following.
Let $\AA = \Coh(C)$ denote the category of coherent sheaves.
One can recast the information of rank and degree as an additive homomorphism
\begin{equation*}
Z: K_0(C) \to \C, \, \, v \mapsto -d(v) + \sqrt{-1}\, r(v),
\end{equation*}
where $K_0(C)$ denotes the Grothendieck group of $C$, generated by classes of vector bundles.
Then:
\begin{enumerate}
\item For any $E \in \AA$, we have $\Im Z(E) \geq 0$.
\item\label{enum:Intro1} If $\Im Z(E) = 0$ for some non trivial $E \in \AA$, then $\Re Z(E) < 0$.
\item For any $E \in \AA$ there is a filtration \[
0 = E_0 \subset E_1 \subset \ldots \subset E_{n-1} \subset E_n = E
\]
of objects $E_i \in \AA$ such that $A_i = E_i/E_{i-1}$ is semistable for all $i = 1, \ldots, n$ and $\mu(A_1) > \ldots > \mu(A_n)$.
\end{enumerate}


\medskip

\noindent
{\bf Higher dimensions.}
The first issue one has to deal with is that if one asks for the same properties to be true for coherent sheaves on a higher dimensional smooth projective variety $X$, it is not so hard to see that property \eqref{enum:Intro1} cannot be achieved (by any possible group homomorphism $Z$).
The key idea is then to change the category in which to define stability.
The bounded derived category of coherent sheaves $\Db(X)$ contains many full abelian subcategories with similar properties as $\Coh(X)$ known as \emph{hearts of bounded t-structures}.
A \emph{Bridgeland stability condition} on $\Db(X)$ is such a heart $\AA \subset \Db(X)$ together with an additive homomorphism $Z: K_0(X) \to \C$ satisfying the three properties above (together with a technical condition, called the \emph{support property}, which will be fundamental for the deformation properties below).
The precise definition is in Section \ref{sec:bridgeland_stability}.

Other classical generalizations of stability on curves such as slope stability or Gieseker stability (see Section \ref{sec:generalizations}) do not directly fit into the framework of Bridgeland stability. However, for most known constructions, their moduli spaces still appear as special cases of Bridgeland stable objects. We will explain this in the case of surfaces.


\medskip

\noindent
{\bf Bridgeland's deformation theorem.}
The main theorem in \cite{Bri07:stability_conditions} (see Theorem \ref{thm:BridgelandMain}) is that the set of stability conditions $\Stab(X)$ can be given the structure of a complex manifold in such a way that the map $(\AA, Z) \mapsto Z$ which forgets the heart is a local homeomorphism.
For general $X$ is not even known whether $\Stab(X) \neq \emptyset$. However, if $\dim X = 2$, the situation is much more well understood. In Section \ref{sec:surfaces} we construct stability conditions $\sigma_{\omega, B}$ for each choice of ample $\R$-divisor class $\omega$ and arbitrary $\R$-divisor class $B$. This construction originated in \cite{Bri08:stability_k3} in the case of K3 surfaces. Arcara and Bertram realized that the construction can be generalized to any surface by using the Bogomolov Inequality in \cite{AB13:k_trivial}.
The proof of the support property is in \cite{BM11:local_p2, BMT14:stability_threefolds, BMS16:abelian_threefolds}.
As a consequence, these stability conditions vary continuously in $\omega$ and $B$.


\medskip

\noindent
{\bf Moduli spaces.}
If we fix a numerical class $v$, it turns out that semistable objects with class $v$ vary nicely within $\Stab(X)$ due to \cite{Bri08:stability_k3}. More precisely, there is a locally finite wall and chamber structure such that the set of semistable objects with class $v$ is constant within each chamber.
In the case of surfaces, as mentioned before, there will be a chamber where Bridgeland semistable objects of class $v$ are exactly (twisted) Gieseker semistable sheaves.
The precise statement is Corollary \ref{cor:LargestWallExists}.
The first easy applications of the wall and chamber structure are in Section \ref{sec:applications}.

The next question is about moduli spaces.
Stability of sheaves on curves (or more generally, Gieseker stability on higher dimensional varieties) is associated to a GIT problem.
This guarantees that moduli spaces exist as projective schemes.
For Bridgeland stability, there is no natural GIT problem associated to it.
Hence, the question of existence and the fact that moduli spaces are indeed well-behaved is not clear in general.
Again, in the surface case, for the stability conditions $\sigma_{\omega,B}$, it is now known \cite{Tod08:K3Moduli} that moduli spaces exist as Artin stacks of finite-type over $\C$. In some particular surfaces, it is also known coarse moduli spaces parameterizing S-equivalence classes of semistable objects exist, and they are projective varieties.
We review this in Section \ref{subsec:ModuliSurfaces}.


\medskip

\noindent
{\bf Birational geometry of moduli spaces of sheaves on surfaces.}
The birational geometry of moduli spaces of sheaves on surfaces has been heavily studied by using wall crossing techniques in Bridgeland's theory. Typical question are what their nef and effective cones are and what the stable base locus decomposition of the effective cone is. The case of $\P^2$ was studied in many articles such as \cite{ABCH13:hilbert_schemes_p2, CHW17:effective_cones_p2, LZ18:stability_p2, LZ19:NewStabilityP2, Woo13:torsion_sheaves_p2}. The study of the abelian surfaces case started in \cite{MM13:stability_abelian_surfaces} and was completed in \cite{MYY14:stability_k_trivial_surfaces,YY14:stability_abelian_surfaces}. The case of K3 surfaces was handled in \cite{MYY14:stability_k_trivial_surfaces, BM14:projectivity, BM14:stability_k3}. Enriques surfaces were studied in \cite{Nue14:stability_enriques}. We will showcase some of the techniques in Section \ref{sec:nef} by explaining how to compute the nef cone of the Hilbert scheme of $n$ points for surfaces of Picard rank one if $n \gg 0$. In this generality it was done in \cite{BHLRSW15:nef_cones} and then later generalized to moduli of vector bundles with large discriminant in \cite{CH18:nef_cones}.

The above proofs are based on the so called \emph{Positivity Lemma} from \cite{BM14:projectivity} (see Theorem \ref{thm:positivity_lemma}).
Roughly, the idea is that to any Bridgeland stability condition $\sigma$ and to any family $\EE$ of $\sigma$-semistable objects parameterized by a proper scheme $S$, there is a nef divisor class $D_{\sigma,\EE}$ on $S$.
Moreover, if there exists a curve $C$ in $S$ such that $D_{\sigma,\EE}.C=0$, then all objects $\EE_{|X \times \{c\}}$ are S-equivalent, for all $c\in C$.
In examples, the divisor class will induce an ample divisor class on the moduli space of stable objects. Hence, we can use the Positivity Lemma in two ways: until we move in a chamber in the space of stability conditions, this gives us a subset of the ample cone of the moduli space.
Once we hit a wall, we have control when we hit also a boundary of the nef cone if we can find a curve of $\sigma$-stable objects in a chamber which becomes properly semistable on the wall.



\medskip

\noindent
{\bf Bridgeland stability for threefolds.}
As mentioned before, the original motivation for Bridgeland stability comes from string theory. In particular, it requires the construction of stability conditions on Calabi-Yau threefolds. It is an open problem to even find a single example of a stability condition on a simply connected projective Calabi-Yau threefold where skyscraper sheaves are stable (examples where they are semistable are in \cite{BMS16:abelian_threefolds}). Most successful attempts on threefolds trace back to a conjecture in \cite{BMT14:stability_threefolds}. In the surface case the construction is based on the classical Bogomolov inequality for Chern characters of semistable vector bundles. By analogy a generalized Bogomolov inequality for threefolds involving the third Chern character was conjectured in \cite{BMT14:stability_threefolds} that allows to construct Bridgeland stability. In \cite{Sch17:counterexample} it was shown that this conjectural inequality needs to be modified, since it does not hold for the blow up of $\P^3$ in a point.
There are though many cases in which the original inequality is true. The first case was $\P^3$ in \cite{BMT14:stability_threefolds, Mac14:conjecture_p3}. A similar argument was then successfully applied to the smooth quadric hypersurface in $\P^4$ in \cite{Sch14:conjecture_quadric}. The case of abelian threefolds was independently proved in \cite{MP15:conjecture_abelian_threefoldsI,MP16:conjecture_abelian_threefoldsII} and \cite{BMS16:abelian_threefolds}. Moreover, as pointed out in \cite{BMS16:abelian_threefolds}, this also implies the case of \'etale quotients of abelian threefolds and gives the existence of Bridgeland stability condition on orbifold quotients of abelian threefolds (this includes examples of Calabi-Yau threefolds which are simply-connected). The latest progress is the proof of the conjecture for all Fano threefolds of Picard rank one in \cite{Li19:conjecture_fano_threefold} and a proof of a modified version for all Fano threefolds independently in \cite{BMSZ17:stability_fano} and \cite{Piy17:Fano}.

Once stability conditions exist on threefolds, it is interesting to study moduli spaces therein and which geometric information one can get by varying stability.
For projective space this approach has led to first results in \cite{Sch15:stability_threefolds, Xia18:twisted_cubics, GHS16:elliptic_quartics}.


\medskip

\noindent
{\bf Structure of the notes.}
In Section \ref{sec:curves} we give a very light introduction to stability of vector bundles on curves. The chapter serves mainly as motivation and is logically independent of the remaining notes. Therefore, it can safely be skipped if the reader wishes to do so. In Section \ref{sec:generalizations}, we discuss classical generalizations of stability from curves to higher dimensional varieties. Moduli spaces appearing out of those are often times of classical interest and connecting these to Bridgeland stability is usually key. In Section \ref{sec:stability_abelian} and Section \ref{sec:bridgeland_stability} we give a full definition of what a Bridgeland stability condition is and prove or point out important basic properties. In Section \ref{sec:surfaces}, we demonstrate the construction of stability on smooth projective surfaces. Section \ref{sec:applications} and Section \ref{sec:nef} are about concrete examples. We show how to compute the nef cone for Hilbert schemes of points in some cases, how one can use Bridgeland stability to prove Kodaira vanishing on surfaces and discuss some further questions on possible applications to projective normality for surfaces. The last Section \ref{sec:P3} is about threefolds. We explain the construction of stability conditions on $\Db(\P^3)$. As an application we point out how Castelnuovo's classical genus bound for non degenerate curves turns out to be a simple consequence of the theory. In Appendix \ref{sec:derived_categories} we give some background on derived categories.


These notes contain plenty of exercises and we encourage the reader to do as many of them as possible. A lot of aspects of the theory that seem obscure and abstract at first turn out to be fairly simple in practice.


\medskip

\noindent
{\bf What is not in these notes.}
One of the main topics of current interest we do not cover in these notes is the ``local case'' (i.e., stability conditions on CY3 triangulated categories defined using quivers with potential; see, e.g., \cite{Bri06:non_compact, KS08:wall_crossing, BS15:quadratic_differentials}) or more generally stability conditions on Fukaya categories (see, e.g., \cite{DHKK14:dynamical_systems, HKK17:flat_surfaces, Joy15:conjectures_fukaya}).
Another fundamental topic is the connection with counting invariants;
for a survey we refer to \cite{Tod12:introduction_dt_theory, Tod12:stability_curve_counting, Tod14:icm}.
Connections to string theory are described for example in \cite{Asp05:d-branes, Bri06:icm, GMN13:hitchin}.
Connections to Representation Theory are in \cite{ABM15:stability}.

There is also a recent survey \cite{Hui17:stability_survey} focusing more on both the classical theory of semistable sheaves and concrete examples still involving Bridgeland stability. The note \cite{Bay18:BrillNoether} focuses instead on deep geometric applications of the theory (the classical Brill-Noether theorem for curves). The survey \cite{Huy11:intro_stability} focuses more on K3 surfaces and applications therein.
Finally, Bridgeland's deformation theorem is the topic of the excellent survey  \cite{Bay11:lectures_notes_stability}, with a short proof recently appearing in \cite{Bay16:short_proof}.

%
%



\medskip

\noindent
{\bf Notation.}
\begin{center}
  \begin{tabular}{ r l }
    $G_k$ & the $k$-vector space $G\otimes k$ for a field $k$ and abelian group $G$ \\
    $X$ & a smooth projective variety over $\C$ \\
    $\II_Z$ & the ideal sheaf of a closed subscheme $Z \subset X$ \\
    $\Db(X)$ & the bounded derived category of coherent sheaves on $X$ \\
    $\ch(E)$ & the Chern character of an object $E \in D^b(X)$ \\
    $K_0(X)$ & the Grothendieck group of $X$ \\
    $K_{\num}(X)$ & the numerical Grothendieck group of $X$ \\
    $\NS(X)$ & the N\'eron-Severi group of $X$ \\
    $N^1(X)$ & $\NS(X)_{\R}$ \\
    $\Amp(X)$ & the ample cone inside $N^1(X)$ \\
    $\Pic^d(C)$ & the Picard variety of lines bundles of degree $d$ on a smooth curve $C$
  \end{tabular}
\end{center}


\medskip

\noindent
{\bf Acknowledgments.}
We would very much like to thank Benjamin Bakker, Arend Bayer, Aaron Bertram, Izzet Coskun, Jack Huizenga, Daniel Huybrechts, Mart\'i Lahoz, Max Lieblich, Ciaran Meachan, Angel Rios Ortiz, Paolo Stellari, Yukinobu Toda, and Xiaolei Zhao for very useful discussions, many explanations on the topics of these notes, and for pointing out imprecisions in the first versions.
We are also grateful to Jack Huizenga for sharing a preliminary version of his survey article \cite{Hui17:stability_survey} with us and to the referee for very useful suggestions which improved the readability of these notes.
The first author would also like to thank very much the organizers of the two schools for the kind invitation and the excellent atmosphere, and the audience for many comments, critiques, and suggestions for improvement.
The second author would like to thank Northeastern University for the hospitality during the writing of this article. This work was partially supported by NSF grant DMS-1523496 and a Presidential Fellowship of the Ohio State University.


\section{Stability on Curves}\label{sec:curves}

The theory of Bridgeland stability conditions builds on the usual notions of stability for sheaves.
In this section we review the basic definitions and properties of stability for vector bundles on curves. Basic references for this section are \cite{New78:introduction_moduli,Ses80:vector_bundles_curves_book, Pot97:lectures_vector_bundles, HL10:moduli_sheaves}. Throughout this section $C$ will denote a smooth projective complex curve of genus $g \geq 0$.

\subsection{The projective line}
\label{subsec:P1}
The starting point is the projective line $\P^1$. In this case, vector bundles can be fully classified, by using the following decomposition theorem, which is generally attributed to Grothendieck, who proved it in modern language. It was known way before in the work of Dedekind-Weber, Birkhoff, Hilbert, among others (see \cite[Theorem 1.3.1]{HL10:moduli_sheaves}).

\begin{thm}
\label{thm:P1}
Let $E$ be a vector bundle on $\P^1$. Then there exist unique integers $a_1, \ldots, a_n \in \Z$ satisfying $a_1> \ldots >a_n$ and unique non-zero vector spaces $V_1, \ldots, V_n$ such that $E$ is isomorphic to
\[
E \cong \left(\OO_{\P^1}(a_1)\otimes_\C V_1\right) \oplus \ldots \oplus \left(\OO_{\P^1}(a_n) \otimes_C V_n\right).
\]
\end{thm}

\begin{proof}
For the existence of the decomposition we proceed by induction on the rank $r(E)$. If $r(E)=1$ there is nothing to prove. Hence, we assume $r(E)>1$.

By Serre duality and Serre vanishing, for all $a \gg 0$ we have
\[
\Hom(\OO_{\P^1}(a), E) = H^1(\P^1, E^\vee(a-2))^\vee = 0.
\]
We pick the largest integer $a \in \Z$ such that $\Hom(\OO_{\P^1}(a), E) \neq 0$ and a non-zero morphism $\phi \in \Hom(\OO_{\P^1}(a), E)$.
Since $\OO_{\P^1}(a)$ is torsion-free of rank $1$, $\phi$ is injective.
Consider the exact sequence
\[
0 \to \OO_{\P^1}(a) \xrightarrow{\phi} E \to F=\Cok(\phi) \to 0.
\]

We claim that $F$ is a vector bundle of rank $r(E)-1$. Indeed, by Exercise \ref{ex:TorsionExactSequence} below, if this is not the case, there is a subsheaf $T_F \into F$ supported in dimension $0$. In particular, we have a morphism from the skyscraper sheaf $\C(x) \into F$. But this gives a non-zero map $\OO_{\P^1}(a+1) \to E$, contradicting the maximality of $a$.

By induction $F$ splits as a direct sum 
\[
F \cong \bigoplus_j \OO_{\P^1}(b_j)^{\oplus r_j}.
\]
The second claim is that $b_j \leq a$, for all $j$.
If not, there is a non-zero morphism $\OO_{\P^1}(a+1) \into F$.
Since $\Ext^1(\OO_{\P^1}(a+1),\OO_{\P^1}(a)) = H^1(\P^1,\OO_{\P^1}(-1))=0$, this morphism lifts to a non-zero map $\OO_{\P^1}(a+1)\to E$, contradicting again the maximality of $a$.

Finally, we obtain a decomposition of $E$ as direct sum of line bundles since
\[
\Ext^1(\oplus_j \OO_{\P^1}(b_j)^{\oplus r_j},\OO_{\P^1}(a)) = \oplus_j H^0(\P^1,\OO_{\P^1}(b_j-a-2))^{\oplus r_j}=0.
\]

The uniqueness now follows from
\[
\Hom(\OO_{\P^1}(a),\OO_{\P^1}(b))=0
\]
for all $a > b$.
\end{proof}

In higher genus, Theorem \ref{thm:P1} fails, but some of its features still hold. The direct sum decomposition will be replaced by a filtration (the \emph{Harder-Narasimhan filtration}) and the blocks $\OO_{\P^1}(a)\otimes_\C V$ by \emph{semistable} sheaves. The ordered sequence $a_1>\ldots>a_n$ will be replaced by an analogously ordered sequence of \emph{slopes}. Semistable sheaves (of fixed rank and degree) can then be classified as points of a projective variety, the \emph{moduli space}. Finally, the uniqueness part of the proof will follow by an analogous vanishing for morphisms of semistable sheaves.

\begin{exercise}
\label{ex:TorsionExactSequence}
Show that any coherent sheaf $E$ on a smooth projective curve $C$ fits into a unique canonical exact sequence
\[ 
0 \to T_E \to E \to F_E \to 0,
\]
where $T_E$ is a torsion sheaf supported and $F_E$ is a vector bundle. Show additionally that this sequence splits (non-canonically).
\end{exercise}

\subsection{Stability}\label{subsec:StabilityCurves}
We start with the basic definition of slope stability.

\begin{defn}\label{def:SlopeCurves}
\begin{enumerate}
\item The \emph{degree} $d(E)$ of a vector bundle $E$ on $C$ is defined to be the degree of the line bundle $\bigwedge^{r(E)} E$.
\item The \emph{degree} $d(T)$ for a torsion sheaf $T$ on $C$ is defined to be the length of its scheme theoretic support.
\item The \emph{rank} of an arbitrary coherent sheaf $E$ on $C$ is defined as $r(E)=r(F_E)$, while its
\emph{degree} is defined as $d(E) = d(T_E) + d(F_E)$.
\item The \emph{slope} of a coherent sheaf $E$ on $C$ is defined as
\[
\mu(E) = \frac{d(E)}{r(E)},
\]
where dividing by $0$ is interpreted as $+\infty$.
\item A coherent sheaf $E$ on $C$ is called \emph{(semi)stable} if for any proper non-trivial subsheaf $F \subset E$ the inequality $\mu(F) < (\leq) \mu(E)$ holds.
\end{enumerate}
\end{defn}

The terms in the previous definition are often times only defined for vector bundles. The reason for this is the following exercise.

\begin{exercise}
\begin{enumerate}
\item Show that the degree is additive in short exact sequences, i.e., for any short exact sequence
\[
0 \to F \to E \to G \to 0
\]
in $\Coh(X)$ the equality $d(E) = d(F) + d(G)$ holds.
\item If $E \in \Coh(C)$ is semistable, then it is either a vector bundle or a torsion sheaf.
\item A vector bundle $E$ on $C$ is (semi)stable if and only if for all non trivial subbundles $F \subset E$ with $r(F) < r(E)$ the inequality $\mu(F) < (\leq) \mu(E)$ holds.
\end{enumerate}
\end{exercise}

\begin{exercise}\label{ex:EllipticCurve}
Let $g=1$ and let $p\in C$ be a point.
Then
\[
\Ext^1(\OO_C,\OO_C)=\C \text{ and } \Ext^1(\OO_C(p),\OO_C)\cong \C.
\]
Hence, we have two non-trivial extensions:
\begin{align*}
& 0 \to \OO_C \to V_0 \to \OO_C \to 0\\
& 0 \to \OO_C \to V_1 \to \OO_C(p) \to 0
\end{align*}
Show that $V_0$ is semistable but not stable and that $V_1$ is stable.
\end{exercise}

A useful fact to notice is that whenever a non-zero sheaf $E \in \Coh(C)$ has rank $0$, its degree is strictly positive. This is one of the key properties we want to generalize to higher dimensions. It also turns out to be useful for proving the following result, which generalizes Theorem \ref{thm:P1} to any curve.

\begin{thm}[\cite{HN74:hn_filtration}]
\label{thm:HNcurves}
Let $E$ be a non-zero coherent sheaf on $C$.
Then there is a unique filtration (called \emph{Harder-Narasimhan filtration})
\[
0 = E_0 \subset E_1 \subset \ldots \subset E_{n-1} \subset E_n = E
\]
of coherent sheaves such that $A_i = E_i/E_{i-1}$ is semistable for all $i = 1, \ldots, n$ and $\mu(A_1) > \ldots > \mu(A_n)$.
\end{thm}

We will give a proof of this statement in a more general setting in Proposition \ref{prop:hn_exists}. In the case of $\P^1$ semistable vector bundles are simply direct sums of line bundles, and stable vector bundles are simply line bundles. Vector bundles on elliptic curves can also be classified (see \cite{Ati57:vector_bundles_elliptic_curves} for the original source and \cite{Pol03:abelian_varieties} for a modern proof):

\begin{ex}
Let $g = 1$, fix a pair $(r,d) \in \Z_{\geq 1} \times \Z$, let $E$ be a vector bundle with $r(E)=r$ and $d(E) = d$. Then
\begin{itemize}
\item $\gcd(r,d)=1$: $E$ is semistable if and only if $E$ is stable if and only if $E$ is indecomposable (i.e., it cannot be decomposed as a non-trivial direct sum of vector bundles) if and only if $E$ is simple (i.e., $\End(E)=\C$). Moreover, all such $E$ are of the form $L \otimes E_0$, where $L\in\Pic^0(C)$ and $E_0$ can be constructed as an iterated extension similarly as in Exercise \ref{ex:EllipticCurve}.
\item $\gcd(r,d)\neq1$: There exist no stable vector bundles. Semistable vector bundles can be classified via torsion sheaves. More precisely, there is a one-to-one correspondence between semistable vector bundles of rank $r$ and degree $d$ and torsion sheaves of length $\gcd(r,d)$. Under this correspondence, indecomposable vector bundles are mapped onto torsion sheaves topologically supported on a point.
\end{itemize}
\end{ex}

\begin{rmk}\label{rmk:JHfiltrations}
Let $E$ be a semistable coherent sheaf on $C$.
Then there is a (non-unique!) filtration (called the \emph{Jordan-H\"older filtration})
\[
0 = E_0 \subset E_1 \subset \ldots \subset E_{n-1} \subset E_n = E
\]
of coherent sheaves such that $A_i = E_i/E_{i-1}$ is stable for all $i = 1, \ldots, n$ and their slopes are all equal. The factors $A_i$ are unique up to reordering. We say that two semistable sheaves are \emph{S-equivalent} if they have the same stable factors up to reordering. More abstractly, this follows from the fact that the category of semistable sheaves on $C$ with fixed slope is an abelian category of finite length (we will return to this in Exercise \ref{exercise:StableObjectsBridgelandStability}). Simple objects in that category agree with stable sheaves.
\end{rmk}

\subsection{Moduli spaces}
\label{subsec:ModuliSpacesCurves}
In this section, let $g \geq 2$ and denote by $C$ a curve of genus $g$. We fix integers $r \in \Z_{\geq 1}$, $d \in \Z$, and a line bundle $L \in \Pic^d(C)$ of degree $d$.

\begin{defn}
\label{def:functors}
We define two functors $\Sch_{\C} \to \Set$ as follows:
\begin{align*}
\underline{M}_C(r,d) (B) & := \left\{\EE \text{ vector bundle on } C\times B\,:\, \forall b\in B,\, \begin{array}{l} \EE|_{C\times\{b\}} \text{ is semistable}\\ r(\EE|_{C\times\{b\}})=r\\
d(\EE|_{C\times\{b\}})=d
\end{array} \right\} \, /\, \cong, \\
\underline{M}_C(r,L) (B) & := \left\{\EE \in \underline{M}_C(r,d) (B) \,:\, \det(\EE|_{C\times\{b\}})\cong L \right\},
\end{align*}
where $B$ is a scheme (locally) of finite-type over $\C$.
\end{defn}

We will denote the (open) subfunctors consisting of stable vector bundles by $\underline{M}^s_C(r,d)$ and $\underline{M}^s_C(r,d)$. The following result summarizes the work of many people, including Dr\'ezet, Mumford, Narasimhan, Ramanan, Seshadri, among others (see \cite{DN89:picard_group_semistable_curves, New78:introduction_moduli, Ses80:vector_bundles_curves_book, Pot97:lectures_vector_bundles, HL10:moduli_sheaves}).
It gives a further motivation to study stable vector bundles, besides classification reasons: the moduli spaces are interesting algebraic varieties.

\begin{thm}
\label{thm:MainSemistableBundlesCurves}
(i) There exists a coarse moduli space $M_C(r,d)$ for the functor $\underline{M}_C(r,d)$ parameterizing S-equivalence classes of semistable vector bundles on $C$. It has the following properties:
\begin{itemize}
\item It is non-empty.
\item It is an integral, normal, factorial, projective variety over $\C$ of dimension $r^2(g-1)+1$.
\item It has a non-empty smooth open subset $M^s_C(r,d)$ parameterizing stable vector bundles.
\item Except when $g=2$, $r=2$, and $d$ is even \footnote{See Example \ref{ex:g2d0}.}, $M_C(r,d)\setminus M^s_C(r,d)$ consists of all singular points of $M_C(r,d)$, and has codimension at least $2$.
\item $\Pic(M_C(r, d)) \cong \Pic(\Pic^d(C)) \times \Z$.
\item If $\gcd(r,d)=1$, then $M_C(r,d)=M^s_C(r,d)$ is a fine moduli space\footnote{To be precise, we need to modify the equivalence relation for $\underline{M}^s_C(r,d)(B)$: $\EE\sim\EE'$ if and only if $\EE \cong \EE' \otimes p_B^*\LL$, where $\LL\in\Pic(B)$.}.
\end{itemize}

(ii) A similar statement holds for $\underline{M}_C(r,L)$\footnote{By using the morphism $\det\colon M_C(r,d)\to \Pic^d(C)$, and observing that $M_C(r,L)=\det^{-1}(L)$.}, with the following additional properties:
\begin{itemize}
\item Its dimension is $(r^2-1)(g-1)$.
\item $\Pic(M_C(r,L)=\Z \cdot \theta$, where $\theta$ is ample.
\item If $u:=\gcd(r,d)$, then the canonical line bundle is $K_{M_C(r,L)}=-2u\theta$.
\end{itemize}
\end{thm}

\begin{proof}[Ideas of the proof]
$\bullet$ Boundedness: Fix a very ample line bundle $\OO_C(1)$ on $C$.
Then, for any semistable vector bundle $E$ on $C$, Serre duality implies
\[
H^1(C, E\otimes \OO_C(m-1)) \cong \Hom(E, \OO_C(-m+1)\otimes K_C)=0,
\]
for all $m\geq m_0$, where $m_0$ only depends on $\mu(E)$. Hence, by replacing $E$ with $E\otimes\OO_C(m_0)$, we obtain a surjective map
\[
\OO_C^{\oplus \chi(C,E)} \onto E.
\]

$\bullet$ Quot scheme: Let $\gamma = \chi(C,E)$, where $E$ is as in the previous step.
We consider the subscheme $R$ of the Quot scheme on $C$ parameterizing quotients $\phi: \OO_C^{\oplus \gamma} \onto U$, where $U$ is locally-free and $\phi$ induces an isomorphism  $H^0(C,\OO_C^{\oplus \gamma})\cong H^0(C,U)$. Then it can be proved that $R$ is smooth and integral. Moreover, the group $\PGL(\gamma)$ acts on $R$.

$\bullet$ Geometric Invariant Theory: We consider the set of semistable (resp., stable) points $R^{ss}$ (resp., $R^s$) with respect to the group action given by $\PGL(\gamma)$ (and a certain multiple of the natural polarization on $R$). Then $M_C(r,d) = R^{ss} \sslash \PGL(\gamma)$ and $M_C^s(r,d) = R^{s} / \PGL(\gamma)$.
\end{proof}
 
\begin{rmk}
While the construction of the line bundle $\theta$ in Theorem \ref{thm:MainSemistableBundlesCurves} and its ampleness can be directly obtained from GIT, we can give a more explicit description by using purely homological algebra techniques; see Exercise \ref{exer:ThetaCurves} below.

In fact, the ``positivity'' of $\theta$ can be split into two statements: the first is the property of $\theta$ to be strictly nef (i.e., strictly positive when intersecting all curves in $M_C(r,L)$); the second is semi-ampleness (or, in this case, to show that $\Pic(M_C(r,L)\cong\Z$). The second part is more subtle and it requires a deeper understanding of stable vector bundles or of the geometry of $M_C(r,L)$. But the first one turns out to be a fairly general concept, as we will see in Section \ref{sec:nef} in the Positivity Lemma.
\end{rmk}

The following exercise foreshadows how we will construct nef divisors in the case of Bridgeland stability in Section \ref{sec:nef}.

\begin{exercise}\label{exer:ThetaCurves}
Let $Z: \Z^{\oplus 2} \to \C$ be the group homomorphism given by
\[
Z(r,d) = - d + \sqrt{-1} r.
\]
Let $r_0,d_0\in\Z$ be such that $r_0\geq1$ and $\gcd(r_0,d_0)=1$, and let $L_0 \in \Pic^{d_0}(C)$.
We consider the moduli space $M:=M_C(r_0,L_0)$ and we let $\EE_0$ be a universal family on $C\times M$.

For an integral curve $\gamma \subset M$, we define
\[
\ell_Z . \gamma := - \Im \frac{Z(r(\Phi_{\EE_0}(\OO_\gamma)),d(\Phi_{\EE_0}(\OO_\gamma)))}{Z(r_0,d_0)} \in\R,
\]
where $\OO_\gamma\in\Coh(M)$ is the structure sheaf of $\gamma$ in $M$, $\Phi_{\EE_0}:\Db(M)\to \Db(C)$, $\Phi_{\EE_0}(-):= (p_C)_*(\EE_0\otimes p_M^*(-))$ is the Fourier-Mukai transform with kernel $\EE_0$, and $\Im$ denotes the imaginary part of a complex number.
\begin{enumerate}
\item Show that by extending $\ell_Z$ by linearity we get the numerical class of a real divisor $\ell_Z\in N^1(M) = N_1(M)^\vee$.
\item Show that
\[
\Im \frac{Z(r(\Phi_{\EE_0}(\OO_\gamma)),d(\Phi_{\EE_0}(\OO_\gamma)))}{Z(r_0,d_0)} = \Im \frac{Z(r(\Phi_{\EE_0}(\OO_\gamma(m))),d(\Phi_{\EE_0}(\OO_\gamma(m))))}{Z(r_0,d_0)},
\]
for any $m\in\Z$ and for any ample line bundle $\OO_\gamma(1)$ on $\gamma$.
\item Show that $\Phi_{\EE_0}(\OO_\gamma(m))$ is a sheaf, for $m\gg0$.
\item By using the existence of relative Harder-Narasimhan filtration (see, e.g., \cite[Theorem 2.3.2]{HL10:moduli_sheaves}), show that, for $m \gg 0$,
\[
\mu(\Phi_{\EE_0}(\OO_\gamma(m))) \leq \frac{d_0}{r_0}.
\]
\item Deduce that $\ell_Z$ is nef, namely $\ell_Z.\gamma\geq0$ for all $\gamma\subset M$ integral curve.
\end{enumerate}
\end{exercise}

\subsection{Equivalent definitions for curves}\label{subsec:EquivalentDefCurves}

In this section we briefly mention other equivalent notions of stability for curves.

Let $C$ be a curve.
We have four equivalent definition of stability for vector bundles:
\begin{enumerate}
\item Slope stability (Definition \ref{def:SlopeCurves}); this is the easiest to handle.
\item GIT stability (sketched in the proof of Theorem \ref{thm:MainSemistableBundlesCurves}); this is useful in the construction of moduli spaces as algebraic varieties.
\item Faltings-Seshadri definition (see below); this is useful to find sections of line bundles on moduli spaces, and more generally to construct moduli spaces without using GIT.
\item Differential geometry definition (see below); this is the easiest for proving results and deep properties of stable objects.
\end{enumerate}

\subsubsection*{Faltings' approach}
The main result is the following.

\begin{thm}[Faltings-Seshadri]
\label{thm:Faltings}
Let $E$ be a vector bundle on a curve $C$. Then E is semistable if and only if there exists a vector bundle $F$ such that $\Hom(F,E)=\Ext^1(F,E)=0$.
\end{thm}

\begin{proof}
``$\Rightarrow$'': This is very hard. We refer to \cite{Ses93:vector_bundles_curves}.

``$\Leftarrow$'': Suppose that $E$ is not semistable.
Then, by definition, there exists a subbundle $G\subset E$ such that $\mu(G) > \mu(E)$. By assumption, $\chi(F,E) = 0$. Due to the Riemann-Roch Theorem, we have
\[
1-g+\mu(E) - \mu(F) =0.
\]
Hence,
\[
1-g+\mu(G) - \mu(F) > 0.
\]
By applying the Riemann-Roch Theorem again, we have $\chi(F,G)>0$, and so $\Hom(F,G)\neq0$.
But this implies $\Hom(F,E)\neq0$, which is a contradiction.
\end{proof}

An example for Theorem \ref{thm:Faltings} is in Section \ref{subsec:ExamplesCurves}.

\subsubsection*{Differential geometry approach}

To simplify, since this is not the main topic of this survey, we will treat only the degree $0$ case.
We refer to \cite[Theorem 2]{NS65:stable_vector_bundles} for the general result and for all details.

\begin{thm}[Narasimhan-Seshadri]
\label{thm:NarasimhanSeshadri}
There is a one-to-one correspondence
\[
\left\{\mathrm{Stable\, vector\, bundles\, on\, }C\,\mathrm{ of\, degree\, }0 \right\} \stackrel{1:1}{\longleftrightarrow} \left\{\mathrm{Irreducible\, unitary\, representations\, of\, }\pi_1(C)\right\}
\]
\end{thm}
\begin{proof}[Idea of the proof]
We only give a very brief idea on how to associate a stable vector bundle to an irreducible unitary representation.
Let $C$ be a curve of genus $g \geq 2$. Consider its universal cover $p: \H \to C$. The group of deck transformation is $\pi_1(C)$. Consider the trivial bundle $\H \times \C^r$ on $\H$ and a representation $\rho \colon \pi_1(C) \to \GL(r,\C)$.
Then we get a $\pi_1(C)$-bundle via the action
\begin{align*}
a_\tau \colon \H \times \C^r & \to \H \times \C^r, \\
(y,v) &\mapsto (\tau \cdot y, \rho(\tau) \cdot v),
\end{align*}
for $\tau \in \pi_1(C)$.
This induces a vector bundle $E_\rho$ of rank $r$ on $C$ with degree $0$.

If $\rho$ is a unitary representation, then the semistability of $E$ can be shown as follows. Assume there is a sub vector bundle $W \subset E$ with $d(W) > 0$. By taking $\bigwedge^{r(W)}$ we can reduce to the case $r(W)=1$, namely $W$ is a line bundle of positive degree. 

By tensoring with a line bundle of degree $0$, we can assume $W$ has a section.
But it can be shown (see \cite[Proposition 4.1]{NS65:stable_vector_bundles}) that a vector bundle corresponding to a non-trivial unitary irreducible representation does not have any section. By splitting $\rho$ as direct sum of irreducible representation, this implies that $W$ has to be a sub line bundle of the trivial bundle, which is impossible.

If the representation is irreducible, the bundle is stable.
\end{proof}

As mentioned before, an important remark about the interpretation of stability of vector bundles in terms of irreducible representations is that it is easier to prove results on stable sheaves in this language.
The first example is the preservation of stability by the tensor product (\cite[Theorem 3.1.4]{HL10:moduli_sheaves}): the tensor product of two semistable vector bundles is semistable.
Example \ref{ex:Mumford} is an application of Theorem \ref{thm:NarasimhanSeshadri} and preservation of stability by taking the symmetric product.

\subsubsection*{Gieseker's ample bundle characterization}

We also mention that for geometric applications (e.g., in the algebraic proof of preservation of stability by tensor product mentioned above), the following result by Gieseker (\cite{Gie79:bogomolov} and \cite[Theorem 3.2.7]{HL10:moduli_sheaves}) is of fundamental importance.

\begin{thm}
\label{thm:Gieseker}
Let $E$ be a semistable vector bundle on $C$. Consider the projective bundle $\pi:\P_C(E) \to C$ and denote the tautological line bundle by $\OO_\pi(1)$.
Then:
\begin{enumerate}
\item\label{enum:nef} $d(E)\geq0$ if and only if $\OO_\pi(1)$ is nef.
\item $d(E)>0$ if and only if $\OO_\pi(1)$ is ample.
\end{enumerate}
\end{thm}

\begin{proof}[Idea of the proof]
We just sketch how to prove the non-trivial implication of \eqref{enum:nef}.
Assume that $d(E)\geq0$.
To prove that $\OO_\pi(1)$ is nef, since it is relatively ample, we only need to prove that $\OO_\pi(1) \cdot \gamma \geq0$ for a curve $\gamma \subset \P_C(E)$ which maps surjectively onto $C$ via $\pi$.

By taking the normalization, we can assume that $\gamma$ is smooth.
Denote by $f:\gamma \to C$ the induced finite morphism. Then we have a surjection $f^*E \to \OO_\pi(1)|_\gamma$. But $f^*E$ is semistable, by Exercise \ref{ex:StabilityPullback} below. Hence, $d(\OO_\pi(1)|_\gamma)\geq0$, which is what we wanted.
\end{proof}

\begin{exercise}\label{ex:StabilityPullback}
Let $f:C'\to C$ be a finite morphism between curves. Let $E$ be a vector bundle on $C$. Show that $E$ is semistable if and only if $f^*E$ is semistable.
\emph{Hint: One implication (``$\Leftarrow$'') is easy. For the other (``$\Rightarrow$''), by using the first implication, we can reduce to a Galois cover. Then use uniqueness of Harder-Narasimhan filtrations to show that if $f^*E$ is not semistable. Then all its semistable factors must be invariant. Then use that all factors are vector bundles to show that they descend to achieve a contradiction.}
\end{exercise}

\subsection{Applications and examples}
\label{subsec:ExamplesCurves}

In this section we present a few examples and applications of stability of vector bundles on curves. We will mostly skip proofs, but give references to the literature.

\begin{ex}[Mumford's example; {\cite[Theorem 10.5]{Har70:ample_subvarieties}}]
\label{ex:Mumford}
Let $C$ be a curve of genus $g\geq2$.
Then there exists a stable rank $2$ vector bundle $U$ on $C$ such that
\begin{equation}\label{eq:MumfordEx}
H^0(C,\Sym^m U) = 0,
\end{equation}
for all $m>0$\footnote{If we let $\pi\colon X:=\P_C(U)\to C$ be the corresponding ruled surface, then \eqref{eq:MumfordEx} implies that the nef divisor $\OO_\pi(1)$ is not ample, although it has the property that $\OO_\pi(1) \cdot \gamma>0$, for all curves $\gamma\subset X$.}.

Indeed, by using Theorem \ref{thm:NarasimhanSeshadri}, since
\[
\pi_1(C) = \langle a_1,b_1,\ldots,a_g,b_g | a_1b_1a_1^{-1}b_1^{-1} \cdot \ldots \cdot a_g b_g a_g^{-1} b_g^{-1} \rangle,
\]
we only need to find two matrices $A,B\in\mathrm{U}(2)$ such that $\Sym^m A$ and $\Sym^m B$ have no common fixed subspace, for all $m>0$.
Indeed, by letting them be the matrices corresponding to $a_1$ and $b_1$, the matrices corresponding to the other $a_i$'s and $b_i$'s can be chosen to satisfy the relation.

We can choose
\[
A = \begin{pmatrix}
\lambda_1 & 0\\ 0 & \lambda_2
\end{pmatrix}
\]
where $|\lambda_i|=1$ and $\lambda_2/\lambda_1$ is not a root of unity. Then
\[
\Sym^m A = \begin{pmatrix}
\lambda_1^m & 0 & \dots & 0\\
0 & \lambda_1^{m-1} \lambda_2 & \dots & 0\\
\vdots & \vdots & \ddots & \vdots\\
0 & 0 & \dots & \lambda_2^m
\end{pmatrix}
\]
has all eigenvalues distinct, and so the only fixed subspaces are subspaces generated by the standard basis vectors.

If we let
\[
B = \begin{pmatrix}
\mu_{11} & \mu_{12}\\ \mu_{21} & \mu_{22}
\end{pmatrix},
\]
then all entries of $\Sym^m B$ are certain polynomials in the $\mu_{ij}$ not identically zero.
Since $\mathrm{U}(2) \subset \GL(2,\C)$ is not contained in any analytic hypersurface, we can find $\mu_{ij}$ such that all entries of $\Sym^m B$ are non-zero, for all $m>0$.
This concludes the proof.
\end{ex}

The two classical examples of moduli spaces are the following (see \cite{NR69:moduli_vector_bundles, DR76:vector_bundles_rank2}).

\begin{ex}\label{ex:g2d0}
Let $C$ be a curve of genus $g=2$. Then $M_C(2,\OO_C) \cong \P^3$.

Very roughly, the basic idea behind this example is the following.
We can identify $\P^3$ with $\P(H^0(\Pic^1(C), 2 \Theta))$, where $\Theta$ denotes the theta-divisor on $\Pic^1(C)$.
The isomorphism in the statement is then given by associating to any rank $2$ vector bundle $W$ the subset
\[
\left\{ \xi \in \Pic^1(C)\,:\, H^0(C,W\otimes\xi)\neq 0\right\}.
\]
It is interesting to notice that the locus of strictly semistable vector bundles corresponds to the S-equivalence class of vector bundles of the form $M\oplus M^*$, where $M\in\Pic^0(C)$.
Geometrically, this is the Kummer surface associated to $\Pic^0(C)$ embedded in $\P^3$.
\end{ex}

\begin{ex}
\label{ex:IntersectionQuadrics}
Let $C$ be a curve of genus $g=2$, and let $L$ be a line bundle of degree $1$.
Then $M_C(2,L) \cong Q_1 \cap Q_2 \subset \P^5$, where $Q_1$ and $Q_2$ are two quadric hypersurfaces.

The space $M_2(2,L)$ is one of the first instances of derived category theory, semi-orthogonal decompositions, and their connection with birational geometry.
Indeed, the result above can be enhanced by saying that there exists a \emph{semi-orthogonal decomposition} of the derived category of $M_C(2,L)$ in which the non-trivial component is equivalent to the derived category of the curve $C$.
We will not add more details here, since it is outside the scope of these lecture notes; we refer instead to \cite{BO95:semiorthogonal_decomposition, Kuz14:semiorthogonal_decompositions} and references therein.
There is also an interpretation in terms of Bridgeland stability conditions (see \cite{BLMS16:BridgelandStabilitySemiOrth}).
\end{ex}

\begin{ex}
Let $C$ be a curve of genus $2$.
The moduli space $M_C(3,\OO_C)$ also has a very nice geometric interpretation, and it is very closely related to the Coble cubic hypersurface in $\P^8$.
We refer to \cite{Ort03:ortega_thesis} for the precise statement and for all the details.
\end{ex}

Finally, we mention another classical application of vector bundles techniques to birational geometry. In principle, by using Bridgeland stability, this could be generalized to higher dimensions. We will talk briefly about this in Section \ref{sec:applications}.

\begin{ex}
\label{ex:LazarsfledMukaiCurves}
Let $C$ be a curve of genus $g\geq1$, and let $L$ be an ample divisor on $C$ of degree $d(L) \geq 2g+1$. Then the vector bundle
\[
M_L := \Ker\left(\ev\colon \OO_C\otimes_\C H^0(C,L) \onto L\right)
\]
is semistable. As an application this implies that the embedding of $C$ in $\P(H^0(C,L)^\vee)$ induced by $L$ is projectively normal, a result by Castelnuovo, Mattuck, and Mumford. We refer to \cite{Laz16:linear_series_survey} for a proof of this result, and for more applications.

We will content ourselves to prove a somehow related statement for the canonical line bundle, by following \cite[Theorem 1.6]{GP00:vanishing_theorems}, to illustrate Faltings' definition of stability: the vector bundle $M_{K_C}$ is semistable.

To prove this last statement we let $L$ be a general line bundle of degree $g+1$ on $C$.
Then $L$ is base-point free and $H^1(C,L)=0$.

Start with the exact sequence
\[
0 \to M_{K_C} \to \OO_C\otimes_\C H^0(C,K_C) \to K_C \to 0,
\]
tensor by $L$ and take cohomology. We have a long exact sequence
\begin{equation}\label{eq:FaltingsApplication}
\begin{split}
0 \to H^0(C,M_{K_C}\otimes L) \to H^0(C,K_C) \otimes H^0(C,L) &\stackrel{\alpha}{\to} H^0(C,K_C \otimes L) \\ &\to H^1(C,M_{K_C}\otimes L) \to 0.
\end{split}
\end{equation}

By the Riemann-Roch Theorem, $h^0(C,L)=2$.
Hence, we get an exact sequence
\[
0\to L^\vee \to \OO_C\otimes H^0(C,L) \to L \to 0.
\]
By tensoring by $K_C$ and taking cohomology, we get
\[
0 \to H^0(C,K_C\otimes L^\vee) \to H^0(C,K_C)\otimes H^0(C,L) \stackrel{\alpha}{\to} H^0(C,K_C \otimes L),
\]
and so $\Ker(\alpha)\cong H^0(C,K_C\otimes L^\vee) = H^1(C,L)^\vee = 0$.

By using \eqref{eq:FaltingsApplication}, this gives $H^0(C,M_{K_C}\otimes L) = H^1(C,M_{K_C}\otimes L) = 0$, and so $M_{K_C}$ is semistable, by Theorem \ref{thm:Faltings}.
\end{ex}


\section{Generalizations to Higher Dimensions}
\label{sec:generalizations}

The definitions of stability from Section \ref{subsec:EquivalentDefCurves} generalize in different ways from curves to higher dimensional varieties. In this section we give a quick overview on this. We will not talk about how to generalize the Narasimhan-Seshadri approach. There is a beautiful theory (for which we refer for example to \cite{Kob87:differential_geometry}), but it is outside the scope of this survey.

Let $X$ be a smooth projective variety over $\C$ of dimension $n\geq2$.
We fix an ample divisor class $\omega \in N^1(X)$ and another divisor class $B \in N^1(X)$.

\begin{defn}\label{def:TwistedChern}
We define the twisted Chern character as
\[
\ch^B := \ch \cdot e^{-B}.
\]
\end{defn}

If the reader is unfamiliar with Chern characters we recommend either a quick look at Appendix A of \cite{Har77:algebraic_geometry} or a long look at \cite{Ful98:intersection_theory}.
By expanding Definition \ref{def:TwistedChern}, we have for example
\begin{align*}
\ch^{B}_0 &= \ch_0 = \rk, \\
\ch^{B}_1 &= \ch_1 - B \cdot \ch_0 ,\\
\ch^{B}_2 &= \ch_2 - B \cdot \ch_1 + \frac{B^2}{2} \cdot \ch_0.
\end{align*}

\subsubsection*{Gieseker-Maruyama-Simpson stability}

GIT stability does generalize to higher dimensions. The corresponding ``numerical notion'' is the one of (twisted) Gieseker-Maruyama-Simpson stability (see \cite{Gie77:vector_bundles, Mar77:stable_sheavesI, Mar78:stable_sheavesII, Sim94:moduli_representations, MW97:thaddeus_principle}).

\begin{defn}
\label{def:GiesekerStability}
Let $E\in\Coh(X)$ be a pure sheaf of dimension $d$.
\begin{enumerate}
\item The $B$-twisted Hilbert polynomial is
\[
P(E,B,t) := \int_X \ch^B(E) \cdot e^{t\omega} \cdot \td_X = a_d(E,B) t^d + a_{d-1}(E,B) t^{d-1} + \ldots + a_0(E,B)
\]
\item We say that $E$ is $B$-twisted Gieseker (semi)stable if, for any proper non-trivial subsheaf $F\subset E$, the inequality
\[
\frac{P(F,B,t)}{a_d(F,B)} < (\leq) \frac{P(E,B,t)}{a_d(E,B)}
\]
holds, for $t\gg0$.
\end{enumerate}
\end{defn}

If $E$ is torsion-free, then $d = n$ and $a_d(E,B) = \ch_0(E)$. Theorem \ref{thm:HNcurves} and the existence part in Theorem \ref{thm:MainSemistableBundlesCurves} generalize for (twisted) Gieseker stability (the second when $\omega$ and $B$ are rational classes), but singularities of moduli spaces become very hard to understand.

\subsubsection*{Slope stability}

We can also define a twisted version of the standard slope stability function for $E \in \Coh(X)$ by
\[
\mu_{\omega, B}(E) := \frac{\omega^{n-1} \cdot \ch^{B}_1(E)}{\omega^n \cdot \ch^{B}_0(E)} = \frac{\omega^{n-1} \cdot \ch_1(E)}{\omega^n \cdot \ch_0(E)} -  \frac{\omega^{n-1} \cdot B}{\omega^n},
\]
where dividing by $0$ is interpreted as $+\infty$.

\begin{defn}\label{def:SlopeStability}
A sheaf $E \in \Coh(X)$ is called slope (semi)stable if for all subsheaves $F \subset E$ the inequality $\mu_{\omega, B}(F) < (\leq) \mu_{\omega, B}(E/F)$ holds.
\end{defn}

Notice that this definition is independent of $B$ and classically one sets $B = 0$. However, this more general notation will turn out useful in later parts of these notes. Also, torsion sheaves are all slope semistable of slope $+\infty$, and if a sheaf with non-zero rank is semistable then it is torsion-free.

This stability behaves well with respect to certain operations, like restriction to (general and high degree) hypersurfaces, pull-backs by finite morphisms, tensor products, and we will see that Theorem \ref{thm:HNcurves} holds for slope stability as well.
The main problem is that moduli spaces (satisfying a universal property) are harder to construct, even in the case of surfaces.

\subsubsection*{Bridgeland stability}

We finally come to the main topic of this survey, Bridgeland stability.
This is a direct generalization of slope stability. The main difference is that we will need to change the category we are working with. Coherent sheaves will never work in dimension $\geq 2$.
Instead, we will look for other abelian categories inside the bounded derived category of coherent sheaves on $X$.

To this end, we first have to treat the general notion of slope stability for abelian categories. Then we will introduce the notion of a bounded t-structure and the key operation of tilting. Finally, we will be able to define Bridgeland stability on surfaces and sketch how to conjecturally construct Bridgeland stability conditions in higher dimensions.

The key advantage of Bridgeland stability is the very nice behavior with respect to change of stability. We will see that moduli spaces of Gieseker semistable and slope semistable sheaves are both particular cases of Bridgeland semistable objects, and one can pass from one to the other by varying stability. This will give us a technique to study those moduli spaces as well.
Finally, we remark that even Faltings' approach to stability on curves, which despite many efforts still lacks a generalization to higher dimensions as strong as for the case of curves, seems to require as well to look at stability for complexes of sheaves (see, e.g., \cite{ACK07:functorial_construction, HP12:postnikov, HP14:postnikov2}).

As drawbacks, moduli spaces in Bridgeland stability are not associated naturally to a GIT problem and therefore harder to construct. Also, the categories to work with are not ``local'', in the sense that many geometric constructions for coherent sheaves will not work for these more general categories, since we do not have a good notion of ``restricting a morphism to an open subset''.


\section{Stability in Abelian Categories}
\label{sec:stability_abelian}

The goal of the theory of stability conditions is to generalize the situation from curves to higher dimensional varieties with similar nice properties. In order to do so it turns out that the category of coherent sheaves is not good enough. In this chapter, we will lay out some general foundation for stability in a more general abelian category. Throughout this section we let $\AA$ be an abelian category with Grothendieck group $K_0(\AA)$.

\begin{defn}\label{def:StabilityFunction}
Let $Z: K_0(\AA) \to \C$ be an additive homomorphism.
We say that $Z$ is a \emph{stability function} if, for all non-zero $E\in\AA$, we have
\[
\Im Z(E) \geq 0, \, \, \text{ and } \, \,  \Im Z(E) = 0 \, \Rightarrow \, \Re Z(E) <0.
\]
\end{defn}

Given a stability function, we will often denote $R := \Im Z$,  the \emph{generalized rank}, $D := - \Re Z$, the \emph{generalized degree}, and $M := \frac{D}{R}$, the \emph{generalized slope}, where as before dividing by zero is interpreted as $+\infty$.

\begin{defn}\label{def:SlopeStabilityAbelian}
Let $Z: K_0(\AA) \to \C$ be a stability function.
A non-zero object $E\in\AA$ is called \emph{(semi)stable} if, for all proper non trivial subobjects $F \subset E$, the inequality $M(F) < (\leq) M(E)$ holds.
\end{defn}

When we want to specify the function $Z$, we will say that objects are $Z$-semistable, and add the suffix $Z$ to all the notation.

\begin{ex}\label{ex:StabilityCurvesAgain}
(1) Let $C$ be a curve, and let $\AA:=\Coh(C)$.
Then the group homomorphism
\[
Z: K_0(C) \to \C, \, \, Z(E) = -d(E) + \sqrt{-1} \, r(E)
\]
is a stability function on $\AA$.
Semistable objects coincide with semistable sheaves in the sense of Definition \ref{def:SlopeCurves}.

(2) Let $X$ be a smooth projective variety of dimension $n\geq2$, $\omega\in N^1(X)$ be an ample divisor class and $B\in N^1(X)$.
Let $\AA:=\Coh(X)$.
Then the group homomorphism
\[
\overline{Z}_{\omega,B}: K_0(X) \to \C, \, \, \overline{Z}_{\omega,B}(E) = -\omega^{n-1} \cdot \ch_1^B(E) + \sqrt{-1} \, \omega^n \cdot \ch_0^B(E)
\]
is {\bf not} a stability function on $\AA$.
Indeed, for a torsion sheaf $T$ supported in codimension $\geq 2$, we have $\overline{Z}_{\omega,B}(T)=0$.
More generally, by \cite[Lemma 2.7]{Tod09:limit_stable_objects} there is no stability function whose central charge factors through the Chern character for $\AA = \Coh(X)$.

(3) Let $X$, $\omega$, and $B$ be as before.
Consider the quotient category
\[
\AA:= \Coh_{n,n-1}(X) := \Coh(X) / \Coh(X)_{\leq n-2}
\]
of coherent sheaves modulo the (Serre) subcategory $\Coh(X)_{\leq n-2}$ of coherent sheaves supported in dimension $\leq n-2$ (see \cite[Definition 1.6.2]{HL10:moduli_sheaves}). Then $\overline{Z}_{\omega,B}$ is a stability function on $\AA$. Semistable objects coincide with slope semistable sheaves in the sense of Definition \ref{def:SlopeStability}.
\end{ex}

\begin{exercise}\label{ex:EquivDefStabilityAbelian}
Show that an object $E \in \AA$ is (semi)stable if and only for all quotient $E \onto G$ the inequality $M(E) < (\leq) M(G)$ holds.
\end{exercise}

Stability for abelian categories still has Schur's property:

\begin{lem}\label{lem:Schur}
Let $Z: K_0(\AA) \to \C$ be a stability function.
Let $A,B\in\AA$ be non-zero objects which are semistable with $M(A) > M(B)$.
Then $\Hom(A,B)=0$.
\end{lem}

\begin{proof}
Let $f:A\to B$ be a morphism, and let $Q := \im(f) \subset B$.
If $f$ is non-zero, then $Q \neq 0$.

Since $B$ is semistable, we have $M(Q)\leq M(B)$.
Since $A$ is semistable, by Exercise \ref{ex:EquivDefStabilityAbelian}, we have $M(A) \leq M(Q)$. This is a contradiction to $M(A) > M(B)$.
\end{proof}

\begin{defn}
\label{defn:stability_abelian}
Let $Z: K_0(\AA) \to \C$ be an additive homomorphism.
We call the pair $(\AA, Z)$ a \emph{stability condition} if
\begin{itemize}
\item $Z$ is a stability function, and
\item Any non-zero $E \in \AA$ has a filtration, called the \emph{Harder-Narasimhan filtration},
\[
0 = E_0 \subset E_1 \subset \ldots \subset E_{n-1} \subset E_n = E
\]
of objects $E_i \in \AA$ such that $A_i = E_i/E_{i-1}$ is semistable for all $i = 1, \ldots, n$ and $M(A_1) > \ldots > M(A_n)$.
\end{itemize}
\end{defn}

\begin{exercise}\label{exercise:hn_filtration}
Show that for any stability condition $(\AA, Z)$ the Harder-Narasimhan filtration for any object $E \in \AA$ is unique up to isomorphism.
\end{exercise}

In concrete examples we will need a criterion for the existence of Harder-Narasimhan filtrations. This is the content of Proposition \ref{prop:hn_exists}. It requires the following notion.

\begin{defn}
An abelian category $\AA$ is called \emph{noetherian}, if for any $A \in \AA$ and for any ascending chain of subobjects of $A$
\[
A_0 \subset A_1 \subset \ldots \subset A_i \subset \ldots \subset A
\]
the equality $A_i = A_j$ holds for $i,j \gg 0$.
\end{defn}

Standard examples of noetherian abelian categories are the category of modules over a noetherian ring or the category of coherent sheaves on a variety. Before dealing with Harder-Narasimhan filtrations, we start with a preliminary result.

\begin{lem}
\label{lem:bounded_degree}
Let $Z: K_0(\AA) \to \C$ be a stability function.
Assume that
\begin{itemize}
\item $\AA$ is noetherian, and
\item the image of $R$ is discrete in $\R$.
\end{itemize}
Then, for any object $E \in \AA$, there is a number $D_E \in \R$ such that for any $F \subset E$ the inequality $D(F) \leq D_E$ holds.
\end{lem}

\begin{proof}
Since the image of $R$ is discrete in $\R$, we can do induction on $R(E)$. If $R(E) = 0$ holds, then $R(F) = 0$. In particular, $0 < D(F) \leq D(E)$.

Let $R(E) > 0$. Assume there is a sequence $F_n \subset E$ for $n \in \N$ such that
\[
\lim_{n \to \infty} D(F_n) = \infty.
\]
If the equality $R(F_n) = R(E)$ holds for any $n \in \N$, then the quotient satisfies $R(E/F_n) = 0$. In particular, $D(E/F_n) \geq 0$ implies $D(F_n) \leq D(E)$. Therefore, we can assume $R(F_n) < R(E)$ for all $n$.

We will construct an increasing sequence of positive integers $n_k$ such that $D(F_{n_1} + \ldots + F_{n_k}) > k$ and $R(F_{n_1} + \ldots + F_{n_k}) < R(E)$. We are done after that, because this contradicts $\AA$ being noetherian. By assumption, we can choose $n_1$ such that $D(F_{n_1}) \geq 1$. Assume we have constructed $n_1$, \ldots, $n_{k-1}$. There is an exact sequence 
\[
0 \to (F_{n_1} + \ldots + F_{n_{k-1}}) \cap F_n \to (F_{n_1} + \ldots + F_{n_{k-1}}) \oplus F_n \to F_{n_1} + \ldots + F_{n_{k-1}} + F_n \to 0
\]
for any $n > n_{k-1}$, where intersection and sum are taken inside of $E$. Therefore,
\[
D(F_{n_1} + \ldots + F_{n_{k-1}} + F_n) = D(F_{n_1} + \ldots + F_{n_{k-1}}) + D(F_n) - D((F_{n_1} + \ldots + F_{n_{k-1}}) \cap F_n).
\]
By induction $D((F_{n_1} + \ldots + F_{n_{k-1}}) \cap F_n)$ is bounded from above and we obtain
\[
\lim_{n \to \infty} D(F_{n_1} + \ldots + F_{n_{k-1}} + F_n) = \infty.
\]
As before this is only possible if $R(F_{n_1} + \ldots + F_{n_{k-1}} + F_n) < R(E)$ for $n \gg 0$. Therefore, we can choose $n_k > n_{k-1}$ as claimed.
\end{proof}

We will follow a proof from \cite{Gra84:polygon_proof} (the main idea goes back at least to \cite{Sha76:degeneration_vector_bundles}) for the existence of Harder-Narasimhan filtrations, as reviewed in \cite{Bay11:lectures_notes_stability}.

\begin{prop}
\label{prop:hn_exists}
Let $Z: K_0(\AA) \to \C$ be a stability function.
Assume that
\begin{itemize}
\item $\AA$ is noetherian, and
\item the image of $R$ is discrete in $\R$.
\end{itemize}
Then Harder-Narasimhan filtrations exist in $\AA$ with respect to $Z$.
\end{prop}

\begin{proof}
We will give a proof when the image of $D$ is discrete in $\R$ and leave the full proof as an exercise.

\begin{figure}[ht]
  \centerline{
  \xygraph{
!{<0cm,0cm>;<1cm,0cm>:<0cm,1cm>::}
!{(0,0) }*{\bullet}="v0"
!{(-2,1) }*{\bullet}="v1"
!{(-3,2) }*{\bullet}="v2"
!{(-3,4) }*{\bullet}="vn-2"
!{(-2,5) }*{\bullet}="vn-1"
!{(1,6) }*{\bullet}="vn"
!{(0.3,0) }*{v_0}="l0"
!{(-1.7,1.1) }*{v_1}="l1"
!{(-2.7,2) }*{v_2}="l2"
!{(-2.5,3.9) }*{v_{n-2}}="ln-2"
!{(-1.5,4.8) }*{v_{n-1}}="ln-1"
!{(1.3,6.1) }*{v_n}="ln"
"v0"-"v1"
"v0"-"vn"
"v1"-"v2"
"v2"-@{.}"vn-2"
"vn-2"-"vn-1"
"vn-1"-"vn"
}
  }
  \caption{The polygon $\PP(E)$.}
  \label{fig:polygon_hn}
\end{figure}

Let $E \in \AA$ be non-zero. The goal is to construct a Harder-Narasimhan filtration for $E$. By $\HH(E)$ we denote the convex hull of the set of all $Z(F)$ for $F \subset E$. By Lemma \ref{lem:bounded_degree} this set is bounded from the left. By $\HH_l$ we denote the half plane to the left of the line between $Z(E)$ and $0$. If $E$ is semistable, we are done. Otherwise the set $\PP(E) = \HH(E) \cap \HH_L$ is a convex polygon (see Figure \ref{fig:polygon_hn}).

Let $v_0 = 0, v_1, \ldots, v_n = Z(E)$ be the extremal vertices of $\PP(E)$ in order of ascending imaginary part. 
Choose $F_i \subset E$ with $Z(F_i) = v_i$ for $i = 1, \ldots, n-1$. We will now prove the following three claims to conclude the proof.
\begin{enumerate}
\item The inclusion $F_{i-1} \subset F_i$ holds for all $i = 1, \ldots, n$.
\item The object $G_i = F_i / F_{i-1}$ is semistable for all $i = 1, \ldots, n$.
\item The inequalities $M(G_1) > M(G_2) > \ldots > M(G_n)$ hold.
\end{enumerate}

By definition of $\HH(E)$ both $Z(F_{i-1} \cap F_i)$ and $Z(F_{i-1} + F_i)$ have to lie in $\HH(E)$. Moreover, we have $R(F_{i-1} \cap F_i) \leq R(F_{i-1}) < R(F_i) \leq R(F_{i-1} + F_i))$. To see part (1) observe that
\[
Z(F_{i-1} + F_i) - Z(F_{i-1} \cap F_i) = (v_{i-1} - Z(F_{i-1} \cap F_i)) + (v_i - Z(F_{i-1} \cap F_i)). 
\]
This is only possible if $Z(F_{i-1} \cap F_i) = Z(F_{i-1})$ and $Z(F_{i-1} + F_i)) = Z(F_i)$ (see Figure \ref{fig:inclusion_hn}).

\begin{figure}[ht]
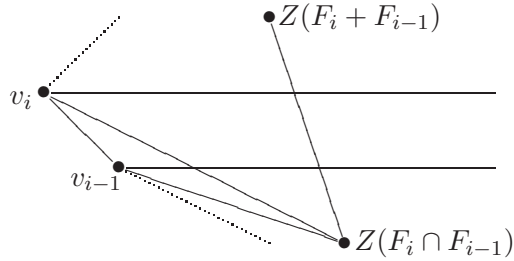

  \centerline{
  \xygraph{
!{<0cm,0cm>;<1cm,0cm>:<0cm,1cm>::}
!{(0,0) }*{\bullet}="intersection"
!{(-3,1) }*{\bullet}="vi-1"
!{(-4,2) }*{\bullet}="vi"
!{(-1,3) }*{\bullet}="sum"
!{(-3,3) }*{}="upper_limit"
!{(-1,0) }*{}="lower_limit"
!{(2,2) }*{}="upper_ray_right"
!{(2,1) }*{}="lower_ray_right"
!{(1.2,0) }*{Z(F_i \cap F_{i-1})}="lintersection"
!{(-3.3,0.8) }*{v_{i-1}}="lvi-1"
!{(-4.3,1.9) }*{v_i}="lvi"
!{(0.2,3) }*{Z(F_i + F_{i-1})}="lsum"
"vi-1"-"vi"
"vi-1"-@{.}"lower_limit"
"vi"-@{.}"upper_limit"
"vi"-"upper_ray_right"
"vi-1"-"lower_ray_right"
"intersection"-"vi-1"
"intersection"-"vi"
"intersection"-"sum"
}
  }
  \caption{Vectors adding up within a convex set.}
  \label{fig:inclusion_hn}
\end{figure}

This implies $Z(F_{i-1}/(F_{i-1} \cap F_i)) = 0$ and the fact that $Z$ is a stability function shows $F_{i-1} = F_{i-1} \cap F_i$. This directly leads to (1).

Let $\overline{A} \subset G_i$ be a non zero subobject with preimage $A \subset F_i$. Then $Z(A)$ has to be in $\HH(E)$ and $R(F_{i-1}) \leq R(A) \leq R(F_i)$. That implies $Z(A) - Z(F_{i-1})$ has smaller or equal slope than $Z(F_i) - Z(F_{i-1})$. But this is the same as $M(\overline{A}) \leq M(G_i)$, proving (2).

\begin{figure}[ht]
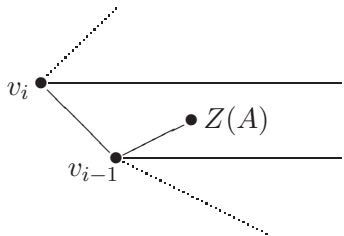

  \centerline{
  \xygraph{
!{<0cm,0cm>;<1cm,0cm>:<0cm,1cm>::}
!{(-2,1) }*{\bullet}="vi-1"
!{(-1,1.5) }*{\bullet}="ZA"
!{(-3,2) }*{\bullet}="vi"
!{(-2,3) }*{}="upper_limit"
!{(0,0) }*{}="lower_limit"
!{(1,2) }*{}="upper_ray_right"
!{(1,1) }*{}="lower_ray_right"
!{(-0.4,1.5) }*{Z(A)}="lZA"
!{(-2.3,0.8) }*{v_{i-1}}="lvi-1"
!{(-3.3,1.9) }*{v_i}="lvi"
"vi-1"-"vi"
"vi-1"-@{.}"lower_limit"
"vi"-@{.}"upper_limit"
"vi"-"upper_ray_right"
"vi-1"-"lower_ray_right"
"vi-1"-"ZA"
}
  }
  \caption{Subobjects of $G_i$ have smaller or equal slope.}
  \label{fig:semistability_hn}
\end{figure}

The slope $M(G_i)$ is the slope of the line through $v_i$ and $v_{i-1}$. Convexity of $\PP(E)$ implies (3).
\end{proof}

\begin{exercise}
The goal of this exercise is to remove the condition that the image of $D$ is discrete in the proof of Proposition \ref{prop:hn_exists}. Show that there are objects $F_i \subset E$ with $Z(F_i) = v_i$ for $i = 1, \ldots, n-1$. \emph{Hint: By definition of $\PP(E)$ there is a sequence of objects $F_{j,i} \subset E$ such that
\[
\lim_{j \to \infty} Z(F_{j,i}) = Z(F_i).
\]
Since $R$ has discrete image, we can assume $R(F_{j,i}) = \Im v_i$ for all $j$. Replace the $F_{j,i}$ by an ascending chain of objects similarly to the proof of Lemma \ref{lem:bounded_degree}.}
\end{exercise}

\begin{ex}\label{ex:ExamplesStability}
For the stability functions in Example \ref{ex:StabilityCurvesAgain} (1) and (3) Harder-Narasimhan filtrations exist.
\end{ex}

We conclude the section with an example coming from representation theory.

\begin{exercise}
\label{ex:Quiver}
Let $A$ be a finite-dimensional algebra over $\C$ and $\AA=\mod\text{-}A$ be the category of finitely dimensional (right) $A$-modules. We denote the simple modules by $S_1, \ldots, S_m$. Pick $Z_i \in \C$ for $i = 1, \ldots, m$, where $\Im(Z_i) \geq 0$ holds and if $\Im Z_i  = 0$ holds, then we have $\Re Z_i < 0$.
\begin{enumerate}
\item Show that there is a unique homomorphism $Z: K_0(\AA) \to \C$ with $Z(S_i) = Z_i$.
\item Show that $(\AA, Z)$ is a stability condition.
\end{enumerate}
The corresponding projective moduli spaces of semistable $A$-modules were constructed in \cite{Kin94:moduli_quiver_reps}.
\end{exercise}

\begin{rmk}
\label{rmk:WeakStability}
There is also a weaker notion of stability on abelian categories, with similar properties.
We say that a homomorphism $Z$ is a \emph{weak stability function} if, for all non-zero $E\in\AA$, we have
\[
\Im Z(E) \geq 0, \, \, \text{ and } \, \,  \Im Z(E) = 0 \, \Rightarrow \, \Re Z(E) \leq 0.
\]
Semistable objects can be defined similarly, and Proposition \ref{prop:hn_exists} still holds. The proof can be extended to this more general case by demanding that the $F_i$ are maximal among those objects satisfying $Z(F_i) = v_i$. A \emph{weak stability condition} is then defined accordingly.

For example, slope stability of sheaves is coming from a weak stability condition on $\Coh(X)$. This is going to be useful for Bridgeland stability on higher dimensional varieties (see Section \ref{sec:P3}, or \cite{BMT14:stability_threefolds, BMS16:abelian_threefolds, PT19:bridgeland_moduli_properties}).
\end{rmk}


\section{Bridgeland Stability}
\label{sec:bridgeland_stability}

In this section we review the basic definitions and results in the general theory of Bridgeland stability conditions. The key notions are the heart of a bounded t-structure and slicings recalled in Section \ref{subsec:heart}. We then introduce the two equivalent definitions in Section \ref{subsec:BridgelandDefinition}. Bridgeland's Deformation Theorem, together with a sketch of its proof, is in Section \ref{subsec:BridgelandDefoThm}. Finally, we discuss moduli spaces in Section \ref{subsec:moduli} and the fundamental wall and chamber structure of the space of stability conditions in Section \ref{subsec:WallChamberGeneral}. Some of the proofs we sketch here will be fully proved in the case of surfaces in the next section.

Throughout this section, we let $X$ be a smooth projective variety over $\C$. We also denote the bounded derived category of coherent sheaves on $X$ by $\Db(X):=\Db(\Coh(X))$. A quick introduction to bounded derived categories can be found in Appendix \ref{sec:derived_categories}. The results in this section are still true in the more general setting of triangulated categories.

\subsection{Heart of a bounded t-structure and slicing}\label{subsec:heart}

For the general theory of t-structures we refer to \cite{BBD82:pervers_sheaves}.
In this section we content ourselves to speak about the heart of a bounded t-structure, which in this case is equivalent.

\begin{defn}
\label{def:heart}
The \emph{heart of a bounded t-structure} on $\Db(X)$ is a
full additive subcategory $\AA \subset \Db(X)$ such that
\begin{itemize}
  \item for integers $i > j$ and $A,B \in \AA$, the vanishing
  $\Hom(A[i],B[j]) = \Hom(A, B[j-i]) = 0$ holds, and
  \item for all $E \in \Db(X)$ there are integers $k_1 > \ldots > k_m$, objects $E_i \in \Db(X)$, $A_i \in \AA$ for $i = 1, \ldots, m$ and a collection of triangles
  \[
  \xymatrix{
  0=E_0 \ar[r] & E_1 \ar[r] \ar[d] & E_2 \ar[r] \ar[d] & \dots \ar[r] & E_{m-1}
  \ar[r] \ar[d] & E_m = E \ar[d] \\
  & A_1[k_1] \ar@{-->}[lu] & A_2[k_2] \ar@{-->}[lu] & & A_{m-1}[k_{m-1}]
  \ar@{-->}[lu] & A_m[k_m]. \ar@{-->}[lu] }
  \]
\end{itemize}
\end{defn}

\begin{lem}
The heart $\AA$ of a bounded t-structure in $\Db(X)$ is abelian.
\end{lem}

\begin{proof}[Sketch of the proof]
Exact sequences in $\AA$ are nothing but exact triangles in $\Db(X)$ all of whose objects are in $\AA$. To define kernel and cokernel of a morphism $f:A \to B$, for $A,B\in\AA$ we can proceed as follows. Complete the morphism $f$ to an exact triangle
\[
A \stackrel{f}{\to} B \to C \to A[1].
\]
By the definition of a heart, we have a triangle
\[
C_{>0} \to C \to C_{\leq 0} \to C_{>0}[1],
\]
where $C_{>0}$ belongs to the category generated by extensions of $\AA[i]$, with $i>0$, and $C_{\leq0}$ to the category generated by extensions of $\AA[j]$, with $j\leq 0$.

Consider the composite map $B\to C\to C_{\leq 0}$.
Then an easy diagram chase shows that $C_{\leq 0}\in\AA$.
Similarly, the composite map $C_{>0}\to C \to A[1]$ gives $C_{>0}\in\AA[1]$.
Then $\Ker(f)=C_{>0}[-1]$ and $\Cok(f)=C_{\leq 0}$.
\end{proof}

The standard example of a heart of a bounded t-structure on $\Db(X)$ is given by $\Coh(X)$. If $\Db(\AA) \cong \Db(X)$, then $\AA$ is the heart of a bounded t-structure. The converse does not hold (see Exercise \ref{exer:Boston1222016} below), but this is still one of the most important examples for having an intuition about this notion.
In particular, given $E\in\Db(X)$, the objects $A_1,\ldots,A_n\in\AA$ in the definition of a heart are its \emph{cohomology objects} with respect to $\AA$. They are denoted $A_1=:\HH^{-k_1}_\AA(E),\ldots,A_n=:\HH^{-k_n}_\AA(E)$. In the case $\AA=\Coh(X)$, these are nothing but the cohomology sheaves of a complex. These cohomology objects have long exact sequences, just as cohomology sheaves of a complexes.

\begin{exercise}\label{exer:Boston1222016}
Let $X=\mathbb{P}^1$.
Let $\AA\subset\Db(\P^1)$ be the additive subcategory generated by extensions by $\OO_{\P^1}[2]$ and $\OO_{\P^1}(1)$.
\begin{enumerate}
\item Let $A\in\AA$. Show that there exists $a_0,a_1\geq0$ such that
\[
A \cong \OO_{\P^1}^{\oplus a_0}[2] \oplus \OO_{\P^1}(1)^{\oplus a_1}.
\]
\item Show that $\AA$ is the heart of a bounded t-structure on $\Db(\P^1)$.
\emph{Hint: Use Theorem \ref{thm:P1}.}
\item Show that, as an abelian category, $\AA$ is equivalent to the direct sum of two copies of the category of finite-dimensional vector spaces. 
\item Show that $\Db(\AA)$ is not equivalent to $\Db(\P^1)$.
\end{enumerate}
\end{exercise}

\begin{exercise}
Let $\AA\subset\Db(X)$ be the heart of a bounded t-structure.
Show that there is a natural identification between Grothendieck groups
\[
K_0(\AA) = K_0(X) = K_0(\Db(X)).
\]
\end{exercise}

A slicing further refines the heart of a bounded t-structure.

\begin{defn}[\cite{Bri07:stability_conditions}]
A \emph{slicing} $\PP$ of $\Db(X)$ is a collection of subcategories $\PP(\phi) \subset
\Db(X)$ for all $\phi \in \mathbb{R}$ such that
\begin{itemize}
  \item $\PP(\phi)[1] = \PP(\phi + 1)$,
  \item if $\phi_1 > \phi_2$ and $A \in \PP(\phi_1)$, $B \in \PP(\phi_2)$ then
  $\Hom(A,B) = 0$,
  \item for all $E \in \Db(X)$ there are real numbers $\phi_1 > \ldots > \phi_m$, objects $E_i \in \Db(X)$ for $i = 1, \ldots, m$, and a collection of triangles
  \[
  \xymatrix{
  0=E_0 \ar[r] & E_1 \ar[r] \ar[d] & E_2 \ar[r] \ar[d] & \ldots \ar[r] & E_{m-1}
  \ar[r] \ar[d] & E_m = E \ar[d] \\
  & A_1 \ar@{-->}[lu] & A_2 \ar@{-->}[lu] & & A_{m-1} \ar@{-->}[lu] & A_m
  \ar@{-->}[lu] }
  \]
  where $A_i \in P(\phi_i)$.
\end{itemize}
For this filtration of an element $E \in \Db(X)$ we write $\phi^-(E) := \phi_m$
and $\phi^+(E) := \phi_1$. Moreover, for $E \in P(\phi)$ we call $\phi(E):=\phi$ the \emph{phase} of $E$.
\end{defn}

The last property is called the \emph{Harder-Narasimhan filtration}. By setting
$\AA := \PP((0,1])$ to be the extension closure of the subcategories $\{\PP(\phi):\phi \in (0,1]\}$ one gets the heart of a bounded t-structure from a slicing\footnote{More generally, by fixing $\phi_0\in\R$, the category $\PP((\phi_0,\phi_0+1])$ is also a heart of a bounded t-structure. A slicing is a family of hearts, parameterized by $\R$.}. In both cases of a slicing and the heart of a bounded t-structure the Harder-Narasimhan filtration is unique similarly to Exercise \ref{exercise:hn_filtration}.

\begin{exercise}\label{exercise:tStructuresContained}
Let $\AA,\BB\subset \Db(X)$ be hearts of bounded t-structures. Show that if $\AA \subset \BB$, then $\AA=\BB$. Similarly, if $\PP$ and $\PP'$ are two slicings and $\PP'(\phi) \subset \PP(\phi)$ for all $\phi \in \R$, then $\PP$ and $\PP'$ are identical.
\end{exercise}

\subsection{Bridgeland stability conditions: definition}\label{subsec:BridgelandDefinition}

The definition of a Bridgeland stability condition will depend on some additional data.
More precisely, we fix a finite rank lattice $\Lambda$ and a surjective group homomorphism
\[
v\colon K_0(X) \onto \Lambda.
\]
We also fix a norm $\| \cdot \|$ on $\Lambda_\R$. Recall that all choices of norms are equivalent here and subsequent definitions will not depend on it.

\begin{ex}
\label{ex:NumericalGrothendieck}
Let $T$ be the set of $v \in K_0(X)$ such that $\chi(v,w) = 0$ for all $w \in K_0(X)$. Then the \emph{numerical Grothendieck group} $K_{\num}(X)$ is defined as $K_0(X)/T$.
We have that $K_{\num}(X)$ is a finitely generated $\Z$-lattice.
The choice of $\Lambda=K_{\num}(X)$ together with the natural projection is an example of a pair $(\Lambda,v)$ as before.

If $X$ is a surface, then $K_{\num}(X)$ is nothing but the image of the Chern character map
\[
\ch\colon K_0(X) \to H^*(X,\Q),
\]
and the map $v$ is the Chern character.
For K3 or abelian surfaces, $K_{\num}(X) = H^*_{\mathrm{alg}}(X,\Z) = H^0(X,\Z) \oplus \NS(X) \oplus H^4(X,\Z)$.
\end{ex}

\begin{defn}[\cite{Bri07:stability_conditions}]\label{def:Bridgeland1}
A \emph{Bridgeland stability condition} on $\Db(X)$ is a pair $\sigma
= (\PP,Z)$ where:
\begin{itemize}
\item $\PP$ is a slicing of $\Db(X)$, and
\item $Z\colon \Lambda \to \C$ is an additive
homomorphism, called the \emph{central charge},
\end{itemize}
satisfying the following properties:
\begin{enumerate}
\item For any non-zero $E\in\PP(\phi)$, we have
\[
Z(v(E)) \in \R_{>0} \cdot e^{\sqrt{-1} \pi \phi}.
\]
\item (\emph{support property})
\[
C_{\sigma}:= \inf \left\{ \frac{|Z(v(E))|}{\| v(E) \|}\, : \, 0\neq E \in \PP(\phi), \phi \in \R \right\} > 0.
\]
\end{enumerate}
\end{defn}

As before, the heart of a bounded t-structure can be defined by $\AA := P((0,1])$.
Objects in $\PP(\phi)$ are called $\sigma$-semistable of phase $\phi$.
The \emph{mass} of an object $E \in \Db(X)$ is defined as $m_\sigma(E)=\sum |Z(A_j)|$, where $A_1,\ldots,A_m$ are the Harder-Narasimhan factors of $E$.

\begin{exercise}\label{exercise:StableObjectsBridgelandStability}
Let $\sigma=(\PP,Z)$ be a stability condition.
Show that the category $\PP(\phi)$ is abelian and of finite length (i.e., it is noetherian and artinian). \emph{Hint: Use the support property to show there are only finitely many classes of subobjects in $\PP(\phi)$ of a given object in $\PP(\phi)$.}
\end{exercise}

The simple object in $\PP(\phi)$ are called \emph{$\sigma$-stable}.
As in the case of stability of sheaves on curves (see Remark \ref{rmk:JHfiltrations}), since the category $\PP(\phi)$ is of finite length, $\sigma$-semistable objects admit finite  \emph{Jordan-H\"older filtrations} in $\sigma$-stable ones of the same phase.
The notion of \emph{S-equivalent} $\sigma$-semistable objects is defined analogously as well.

The support property was introduced in \cite{KS08:wall_crossing}. It is equivalent to Bridgeland's notion of a full locally-finite stability condition in \cite[Definition 4.2]{Bri08:stability_k3} (see \cite[Proposition B.4]{BM11:local_p2}).
There is an equivalent formulation: There is a symmetric bilinear form $Q$ on $\Lambda_\R$ such that
\begin{enumerate}
  \item all semistable objects $E\in \PP$ satisfy the inequality $Q(v(E),v(E))
  \geq 0$ and
  \item all non zero vectors $v \in \Lambda_\R$ with $Z(v) = 0$ satisfy $Q(v, v) < 0$.
\end{enumerate}
The inequality $Q(v(E),v(E)) \geq 0$ can be viewed as some generalization of the classical Bogomolov inequality for vector bundles; we will see the precise relation in Section \ref{sec:surfaces}.
By abuse of notation we will often forget $v$ from the notation.
For example, we will write $Q(E,F)$ instead of $Q(v(E), v(F))$.
We will also use the notation $Q(E) = Q(E,E)$.

\begin{exercise}
Show that the previous two definitions of the support property are equivalent. \emph{Hint: Use $Q(w) = \tfrac{1}{C^2} |Z(w)|^2 - ||w||^2$.}
\end{exercise}

Definition \ref{def:Bridgeland1} is short and good for abstract argumentation, but it is not very practical for finding concrete examples. The following lemma shows that this definition of a stability condition on $\Db(X)$ and the one given in the previous section for an arbitrary abelian category $\AA$ are closely related.

\begin{lem}[{\cite{Bri07:stability_conditions}[Proposition 5.3]}]
Giving a stability condition $(\PP,Z)$ on $\Db(X)$ is equivalent to giving a stability condition $(\AA, Z)$ in the sense of Definition \ref{defn:stability_abelian}, where $\AA$ is the heart of a bounded t-structure on $\Db(X)$ together with the support property
\[
\inf \left\{ \frac{|Z(v(E))|}{\| v(E) \|}\, : \, 0\neq E \in \AA\ \mathrm{ semistable}\right\} > 0.
\]
\end{lem}

\begin{proof}
Assume we have a stability conditions $(\PP,Z)$ on $\Db(X)$. Then we can define a heart $\AA = \PP((0,1])$. Then $(\AA, Z)$ is a stability conditions in the sense of Definition \ref{defn:stability_abelian} satisfying the support property. The other way around, we can define $\PP(\phi)$ to be the category of semistable objects of phase $\phi$ in $\AA$ whenever $\phi \in (0,1]$. This definition can be extended to any $\phi \in \R$ via the property $\PP(\phi)[1] = \PP(\phi + 1)$.
\end{proof}

From here on we will interchangeably use $(\PP,Z)$ and $(\AA, Z)$ to denote a stability condition.

\begin{ex}\label{ex:BridgelandStabilityCurves}
Let $C$ be a curve.
Then the stability condition in Example \ref{ex:StabilityCurvesAgain} (1) gives a Bridgeland stability condition on $\Db(C)$.
The lattice $\Lambda$ is nothing but $H^0(C,\Z)\oplus H^2(C,\Z)$, the map $v$ is nothing but $(r,d)$, and we can choose $Q = 0$.
\end{ex}

\begin{exercise}
Show that the stability condition in Exercise \ref{ex:Quiver} gives a Bridgeland stability condition on $\Db(A)=\Db(\mod\text{-}A)$. What is $Q$?
\end{exercise}

\subsection{Bridgeland's deformation theorem}\label{subsec:BridgelandDefoThm}

The main theorem in \cite{Bri07:stability_conditions} is the fact that the set of stability conditions can be given the structure of a complex manifold.

Let $\Stab(X)$ be the set of stability conditions on $\Db(X)$ (with respect to $\Lambda$ and $v$). This set can be given a topology as the coarsest topology such that for any $E \in \Db(X)$ the maps $(\AA,Z) \mapsto Z$, $(\AA, Z) \mapsto \phi^+(E)$ and $(\AA, Z) \mapsto \phi^-(E)$ are continuous. Equivalently, the topology is induced by the generalized (i.e., with values in $[0,+\infty]$) metric
\[
d(\sigma_1,\sigma_2) = \underset{0\neq E \in \Db(X)}{\sup}\left\{|\phi^+_{\sigma_1}(E)-\phi^+_{\sigma_2}(E)|,\, | \phi^-_{\sigma_1}(E) - \phi^-_{\sigma_2}(E)|,\, \|Z_1-Z_2\| \right\},
\]
for $\sigma_i=(\PP_i,Z_i)\in\Stab(X)$, $i=1,2$; here, by abuse of notation, $\| \cdot \|$ also denotes the induced operator norm on $\Hom(\Lambda,\C)$.

\begin{rmk}\label{rmk:GroupActions}
There are two group actions on the space of stability conditions.

(1) The universal cover $\widetilde{\GL}^+(2,\R)$ of $\GL^+(2,\R)$, the $2\times 2$ matrices with real entries and positive determinant, acts on the right of $\Stab(X)$ as follows.
We first recall the presentation
\[
\widetilde{\GL}^+(2,\R) = \left\{ (T,f)\,:\, \begin{array}{l}
f\colon \R\to\R\ \mathrm{increasing}, \ f(\phi+1)=f(\phi)+1\\
T\in\GL^+(2,\R)\\
f|_{\R/2\Z} = T|_{(\R^2 \backslash \{ 0 \})/\R_{>0}}
\end{array} \right\}
\]
Then $(T,f)$ acts on $(\PP,Z)$ by $(T,f)\cdot (\PP,Z) = (\PP',Z')$, where $Z'=T^{-1}\circ Z$ and $\PP'(\phi) = \PP(f(\phi))$.

(2) The group of exact autoequivalences $\Aut_\Lambda(\Db(X))$ of $\Db(X)$, whose action $\Phi_*$ on $K_0(X)$ is compatible with the map $v\colon K_0(X) \to \Lambda$, acts on the left of $\Stab(X)$ by $\Phi\cdot (\PP,Z) = (\Phi(\PP), Z \circ \Phi_*)$.
In \cite{BB13:autoequivalences_k3}, Bayer and Bridgeland use this action and a description of the geometry of the space of stability conditions on K3 surfaces of Picard rank $1$ to describe the full group of derived autoequivalences. This idea should work for all K3 surfaces, as envisioned by Bridgeland in \cite[Conjecture 1.2]{Bri08:stability_k3}.
\end{rmk}

\begin{thm}[\cite{Bri07:stability_conditions}]\label{thm:BridgelandMain}
The map $\ZZ\colon \Stab(X) \to \Hom(\Lambda, \C)$ given by $(\AA,Z) \mapsto Z$ is a local homeomorphism. In particular, $\Stab(X)$ is a complex manifold of dimension $\rk(\Lambda)$.
\end{thm}

\begin{proof}[Ideas from the proof]
We follow the presentation in \cite[Section 5.5]{Bay11:lectures_notes_stability} and we refer to it for the complete argument.
Let $\sigma=(\AA,Z)\in\Stab(X)$.
We need to prove that any group homomorphism $W$ near $Z$ extends to a stability condition near $\sigma$.

The first key property is the following (\cite[Lemma 5.5.4]{Bay11:lectures_notes_stability}): The function $C:\Stab(X)\to\R_{>0}$, $\sigma\mapsto C_\sigma$ is continuous, where $C_\sigma$ is the constant appearing in the support property in Definition \ref{def:Bridgeland1}.

By using this, and the $\widetilde{\GL}^+(2,\R)$-action described in Remark \ref{rmk:GroupActions} (by just rotating of $\pi/2$), it is possible to reduce to the following case:
We let $W\colon \Lambda\to\Z$ satisfy the assumptions
\begin{itemize}
\item $\Im W = \Im Z$;
\item $\| W- Z \| < \epsilon C_\sigma$, with $\epsilon<\frac{1}{8}$.
\end{itemize}

The claim is that $(\AA,W)\in\Stab(X)$.
To prove this, under our assumptions, the only thing needed is to show that HN filtrations exist with respect to $W$.
We proceed as in the proof of Proposition \ref{prop:hn_exists}.
We need an analogue of Lemma \ref{lem:bounded_degree} and to show that there are only finitely many vertices in the Harder-Narasimhan polygon.
The idea for these is the following.
Let $E\in\AA$.
\begin{itemize}
\item The existence of HN filtrations for $Z$ implies there exists a constant $\Gamma_E$ such that
\[
\Re Z(F) \geq \Gamma_E + m_\sigma(F),
\]
for all $F\subset E$.
\item By taking the HN filtration of $F$ with respect to $Z$, and by using the support property for $\sigma$, we get
\[
|\Re W(F_i) - \Re Z(F_i)| \leq \epsilon | Z(F_i) |,
\]
where the $F_i$'s are the HN factors of $F$ with respect to $Z$.
\item By summing up, we get
\[
\Re W(F) \geq \Re Z(F) - \epsilon m_\sigma(F) \geq \Gamma_E + (\underbrace{1-\epsilon}_{>0}) m_\sigma(F) > \Gamma_E,
\]
namely Lemma \ref{lem:bounded_degree} holds.
\item If $F\subset E$ is an extremal point of the polygon for $W$, then
\[
\max \{ 0,\Re W(E)\} > \Re W(F),
\]
and so
\[
m_\sigma(F) \leq \frac{\max \{ 0,\Re W(E)\} - \Gamma_E}{1-\epsilon}=: \Gamma_E'.
\]
\item Again, by taking the HN factors of $F$ with respect to $Z$, we get $|Z(F_i)|<\Gamma_E'$, and so by using the support property,
\[
\| F_i \| < \frac{\Gamma_E'}{C_\sigma}.
\]
\item From this we deduce there are only finitely many classes, and so finitely many vertices in the polygon. The proof of Proposition \ref{prop:hn_exists} applies now and this gives the existence of Harder-Narasimhan filtrations with respect to $W$.
\item The support property follows now easily.
\end{itemize}
\end{proof}

We can think about Bridgeland's main theorem more explicitly in terms of the quadratic form $Q$ appearing in the support property. This approach actually gives an ``$\epsilon$-free proof'' of Theorem \ref{thm:Gieseker} and appears in \cite{Bay16:short_proof}.
The main idea is contained in the following proposition, which is \cite[Proposition A.5]{BMS16:abelian_threefolds}. We will not prove this, but we will see it explicitly in the case of surfaces.

\begin{prop}\label{prop:ExplicitBridgelandMain}
Assume that $\sigma=(\PP,Z)\in\Stab(X)$ satisfies the support property with respect to a quadratic form $Q$ on $\Lambda_\R$.
Consider the open subset of $\Hom(\Lambda,\C)$ consisting of central charges on whose kernel $Q$ is negative definite, and let $U$ be the connected component containing $Z$. Let $\UU\subset\Stab(X)$ be the connected component of the preimage $\ZZ^{-1}(U)$ containing $\sigma$. Then:
\begin{enumerate}
\item The restriction $\ZZ|_{\UU}\colon \UU \to U$ is a covering map.
\item Any stability condition $\sigma'\in\UU$ satisfies the support property with respect to the same quadratic form $Q$.
\end{enumerate}
\end{prop}

\begin{ex}
Spaces of stability conditions are harder to study in general. For curves everything is known.
As before, we keep $\Lambda$ and $v$ as in Example \ref{ex:BridgelandStabilityCurves}.
\begin{enumerate}
\item $\Stab(\P^1) \cong \C^2$ (see \cite{Oka06:P1}).
\item Let $C$ be a curve of genus $g\geq1$. Then $\Stab(C)\cong \H \times \C\, (= \sigma_0 \cdot \widetilde{\GL}^+(2,\R))$ (see \cite{Bri07:stability_conditions, Mac07:curves}), where $\sigma_0$ is the stability $(\Coh(C),-d+\sqrt{-1}\, r)$ in Example \ref{ex:BridgelandStabilityCurves}.
\end{enumerate}

Let $\sigma=(\PP,Z)\in\Stab(C)$. We will examine only the case $g \geq 1$.
The key point of the proof is to show that the skyscraper sheaves $\C(x)$, $x\in C$, and all line bundles $L$ are all in the category $\PP((\phi_0,\phi_0+1])$, for some $\phi_0\in\R$, and they are all $\sigma$-stable.

Assume first that $\C(x)$ is not $\sigma$-semistable, for some $x\in C$.
By taking its HN filtration, we obtain an exact triangle
\[
A \to \C(x) \to B \to A[1]
\]
with $\Hom^{\leq 0}(A,B)=0$\footnote{More precisely, $A\in\PP(<\phi_0)$ and $B\in\PP(\geq\phi_0)$, for some $\phi_0\in\R$.}.
By taking cohomology sheaves, we obtain a long exact sequence of coherent sheaves on $C$
\[
0\to \HH^{-1}(B) \to \HH^0(A) \to \C(x) \stackrel{f}{\to} \HH^0(B) \to \HH^1(A) \to 0,
\]
and $\HH^{i-1}(B)\cong \HH^i(A)$, for all $i\neq 0,1$.

But, since $C$ is smooth of dimension $1$, an object $F$ in $\Db(C)$ is isomorphic to the direct sum of its cohomology sheaves: $F\cong \oplus_i \HH^i(F)[-i]$.
Since $\Hom^{\leq0}(A,B)=0$, this gives $\HH^{i-1}(B)\cong \HH^i(A)=0$, for all $i\neq 0,1$.

We look at the case $f=0$. The case in which $f\neq 0$ (and so, $f$ is injective) can be dealt with similarly.
In this case, $\HH^{0}(B)\cong \HH^1(A)$ and so they are both trivial as well, since otherwise $\Hom^{-1}(A,B)\neq0$.
Therefore, $A\cong \HH^0(A)$ and $B\cong\HH^{-1}(B)[1]$.
But, by using Serre duality, we have
\begin{align*}
0\neq \Hom(\HH^{-1}(B),\HH^0(A))\hookrightarrow \Hom(\HH^{-1}(B), \HH^0(A)\otimes K_C)&\cong\Hom(\HH^0(A),\HH^{-1}(B)[1])\\ &\cong\Hom(A,B),
\end{align*}
a contradiction.

We now claim that $\C(x)$ is actually stable.
Indeed, a similar argument as before shows that if $\C(x)$ is not $\sigma$-stable, then the only possibility is that all its stable factors are isomorphic to a single object $K$. 
But then, by looking at its cohomology sheaves, $K$ must be a sheaf as well, which is impossible.
Set $\phi_x$ as the phase of $\C(x)$ in $\sigma$.

In the same way, all line bundles are $\sigma$-stable as well.
For a line bundle $A$ on $C$, we set $\phi_A$ as the phase of $A$ in $\sigma$.

The existence of maps $A\to\C(x)$ and $\C(x)\to A[1]$ gives the inequalities
\[
\phi_x-1\leq\phi_A\leq\phi_x.
\]
Since $A$ and $\C(x)$ are both $\sigma$-stable, the strict inequalities hold.
Hence, $Z$ is an isomorphism as a map $\R^2\to\R^2$ and orientation-preserving.
By acting with an element of $\widetilde{\GL}^+(2,\R)$, we can therefore assume that $Z = Z_0 = -d + \sqrt{-1}\, r$ and that, for some $x\in C$, the skyscraper sheaf $\C(x)$ has phase $1$.
But then all line bundles on $C$ have phases in the interval $(0,1)$, and so all skyscraper sheaves have phase $1$ as well.
But this implies that $\PP((0,1])=\Coh(C)$, and so $\sigma=\sigma_0$.
\end{ex}

\begin{ex}
One of the motivations for the introduction of Bridgeland stability conditions is coming from mirror symmetry (see \cite{Dou02:mirror_symmetry}).
In case of elliptic curves, the mirror symmetry picture is particularly easy to explain (see \cite[Section 9]{Bri07:stability_conditions} and \cite{Bri09:seattle_survey}).
Indeed, let $C$ be an elliptic curve.
We can look at the action of the subgroup $\C \subset \widetilde{\GL}^+(2,\R)$ and of $\Aut(\Db(C))$ on $\Stab(C)$, and one deduces
\[
\Aut(\Db(C)) \backslash \Stab(C) / \C \cong  \H / \PSL(2,\Z),
\]
coherently with the idea that spaces of stability conditions should be related to variations of complex structures on the mirror variety (in this case, a torus itself).
\end{ex}

\subsection{Moduli spaces}\label{subsec:moduli}

The main motivation for stability conditions is the study of moduli spaces of semistable objects.
In this section we recall the general theory of moduli spaces of complexes, and then define in general moduli spaces of Bridgeland semistable objects.

To define the notion of a \emph{family of complexes}, we recall first the notion of a \emph{perfect complex}.
We follow \cite{Ina02:moduli_complexes,Lie06:moduli_complexes}.
Let $B$ be a scheme locally of finite type over $\C$. We denote the unbounded derived category of quasi-coherent sheaves by $\mathrm{D}(\Qcoh(X\times B))$.
A complex $E\in\mathrm{D}(\Qcoh(X\times B))$ is called \emph{$B$-perfect} if it is, locally over $B$, isomorphic to a bounded complex of flat sheaves over $B$ of finite presentation.
The triangulated subcategory of $B$-perfect complexes is denoted by $\mathrm{D}_{B\text{-perf}}(X\times B)$.

\begin{defn}
We let $\mathfrak{M}\colon \underline{\mathrm{Sch}}\to \underline{\mathrm{Grp}}$ be the $2$-functor which maps a scheme $B$ locally of finite type over $\C$ to the groupoid $\mathfrak{M}(B)$ given by
\[
\left\{\EE\in\mathrm{D}_{B\text{-perf}}(X\times B)\,:\, \Ext^i(\EE_b,\EE_b)=0, \text{ for all }i<0\text{ and all geometric pts }b\in B  \right\},
\]
where $\EE_b$ is the restriction $\EE|_{X\times\{b\}}$
\end{defn}

\begin{thm}[{\cite[Theorem 4.2.1]{Lie06:moduli_complexes}}]
\label{thm:Lieblich}
$\mathfrak{M}$ is an Artin stack, locally of finite type, locally quasi separated, and with separated diagonal.
\end{thm}

We will not discuss Artin stacks in detail any further in this survey; we refer to the introductory book \cite{ols16:algebraic_spaces_stacks} for the basic terminology (or \cite{stacks-project}).
We will mostly use the following special case.
Let $\mathfrak{M}_{\mathrm{Spl}}$ be the open substack of $\mathfrak{M}$ parameterizing simple objects (as in the sheaf case, these are complexes with only scalar endomorphisms). This is also an Artin stack with the same properties as $\mathfrak{M}$.
Define also a functor $\underline{M}_{\mathrm{Spl}}\colon \underline{\mathrm{Sch}}\to \underline{\mathrm{Set}}$ by forgetting the groupoid structure on $\mathfrak{M}_{\mathrm{Spl}}$ and by sheafifying it (namely, as in the stable sheaves case, through quotienting by the equivalence relation given by tensoring with line bundles from the base $B$).

\begin{thm}[{\cite[Theorem 0.2]{Ina02:moduli_complexes}}]
\label{thm:Inaba}
The functor $\underline{M}_{\mathrm{Spl}}$ is representable by an algebraic space locally of finite type over $\C$.
Moreover, the natural morphism $\mathfrak{M}_{\mathrm{Spl}}\to \underline{M}_{\mathrm{Spl}}$ is a $\mathbb{G}_m$-gerbe.
\end{thm}

We can now define the moduli functor for Bridgeland semistable objects.

\begin{defn}
Let $\sigma=(\PP,Z)\in\Stab(X)$.
Fix a class $v_0\in\Lambda$ and a phase $\phi\in\R$ such that $Z(v_0)\in\R_{>0}\cdot e^{\sqrt{-1}\pi\phi}$.
We define  $\widehat{M}_{\sigma}(v,\phi)$ to be the set of $\sigma$-semistable objects in $\PP(\phi)$ of numerical class $v_0$ and $\mathfrak{M}_{\sigma}(v,\phi)\subset\mathfrak{M}$ the substack of objects in $M_{\sigma}(v,\phi)$.
Similarly, we define $\mathfrak{M}^s_{\sigma}(v,\phi)$ as the substack parameterizing stable objects.
\end{defn}

\begin{question}\label{question:OpenAndBounded}
Are the inclusions $\mathfrak{M}^s_{\sigma}(v,\phi)\subset\mathfrak{M}_{\sigma}(v,\phi)\subset\mathfrak{M}$ open?
Is the set $\widehat{M}_{\sigma}(v,\phi)$ bounded?
\end{question}

If the answers to Question \ref{question:OpenAndBounded} are both affirmative, then by Theorem \ref{thm:Lieblich}  and \cite[Lemma 3.4]{Tod08:K3Moduli}, $\mathfrak{M}_{\sigma}(v,\phi)$ and $\mathfrak{M}^s_{\sigma}(v,\phi)$ are both Artin stack of finite type over $\C$.
Moreover, by Theorem \ref{thm:Inaba}, the associated functor $\underline{M}^s_{\sigma}(v,\phi)$ would be represented by an algebraic space of finite type over $\C$.

We will see in Section \ref{subsec:ModuliSurfaces} that Question \ref{question:OpenAndBounded} has indeed an affirmative answer in the case of certain stability conditions on surfaces \cite{Tod08:K3Moduli} (or more generally, when Bridgeland stability conditions are constructed via an iterated tilting procedure \cite{PT19:bridgeland_moduli_properties}).

A second fundamental question is about the existence of a moduli space:

\begin{question}\label{question:ProjectivityModuli}
Is there a coarse moduli space $M_{\sigma}(v,\phi)$ parameterizing S-equivalence classes of $\sigma$-semistable objects? Is $M_{\sigma}(v,\phi)$ a projective scheme?
\end{question}

Question \ref{question:ProjectivityModuli} has an affirmative complete answer only for the projective plane $\P^2$ (see \cite{ABCH13:hilbert_schemes_p2}), $\P^1 \times \P^1$ and the blow up of a point in $\P^2$ (see \cite{AM17:projectivity}), and partial answers for other surfaces, including abelian surfaces (see \cite{MYY14:stability_k_trivial_surfaces,YY14:stability_abelian_surfaces}), K3 surfaces (see \cite{BM14:projectivity}), and Enriques surfaces (see \cite{Nue14:stability_enriques, {Yos16:projectivity_enriques}}). In Section \ref{subsec:ModuliSurfaces} we show how the projectivity is shown in case of $\P^2$ via quiver representations. The other two del Pezzo cases were proved with the same method, but turn out more technically involved.
As remarked in Section \ref{sec:generalizations}, since Bridgeland stability is not a priori associated to a GIT problem, it is harder to answer Question \ref{question:ProjectivityModuli} in general.
The recent works \cite{AS12:good_moduli, AHR15:luna_stacks} on good quotients for Artin stacks may lead to a general answer to this question, though.
Once a coarse moduli space exists, then separatedness and properness of the moduli space is a general result by Abramovich and Polishchuk \cite{AP06:constant_t_structures}, which we will review in Section \ref{sec:nef}.

It is interesting to note that the technique in \cite{AP06:constant_t_structures} is also very useful in the study of the geometry of the moduli space itself. This will also be reviewed in Section \ref{sec:nef}.

\begin{exercise}
Show that Question \ref{question:OpenAndBounded} and Question \ref{question:ProjectivityModuli} have both affirmative answers for curves (you can assume, for simplicity, that the genus is $\geq1$). Show also that moduli spaces in this case are exactly moduli spaces of semistable sheaves as reviewed in Section \ref{sec:curves}.
\end{exercise}

\subsection{Wall and chamber structure}\label{subsec:WallChamberGeneral}

We conclude the section by explaining how stable object change when the stability condition is varied within $\Stab(X)$ itself. It turns out that the topology on $\Stab(X)$ is defined in such as way that this happens in a controlled way. This will be one of the key properties in studying the geometry of moduli spaces of semistable objects.

\begin{defn}\label{def:WallGeneral}
Let $v_0,w\in\Lambda\setminus\{0\}$ be two non-parallel vectors.
A \emph{numerical wall} $W_w(v_0)$ for $v_0$ with respect to $w$ is a non empty subset of $\Stab(X)$ given by
\[
W_w(v_0):=\left\{\sigma=(\PP,Z)\in\Stab(X)\,:\, \Re Z(v_0)\cdot \Im Z(w)=\Re Z(w)\cdot\Im Z(v_0)\right\}.
\]
\end{defn}

We denote the set of numerical walls for $v_0$ by $\WW(v_0)$.
By Theorem \ref{thm:BridgelandMain}, a numerical wall is a real submanifold of $\Stab(X)$ of codimension $1$.

We will use walls only for special subsets of the space of stability conditions (the $(\alpha,\beta)$-plane in Section \ref{subsec:WallChamber}). We will give a complete proof in this case in Section \ref{sec:surfaces}.
The general behavior is explained in \cite[Section 9]{Bri08:stability_k3}.
We direct the reader to \cite[Proposition 3.3]{BM11:local_p2} for the complete proof of the following result.

\begin{prop}\label{prop:locally_finite}
Let $v_0\in\Lambda$ be a primitive class, and let $S\subset\Db(X)$ be an arbitrary set of objects of class $v_0$.
Then there exists a collection of walls $W^S_w(v_0)$, $w\in\Lambda$, with the following properties:
\begin{enumerate}
\item Every wall $W^S_w(v_0)$ is a closed submanifold with boundary of real codimension
one.
\item The collection $W^S_w(v_0)$ is locally finite (i.e., every compact subset $K\subset\Stab(X)$ intersects only a finite number of walls).
\item\label{enum:ActualWall3} For every stability condition $(\PP,Z)$ on a wall in $W^S_w(v_0)$, there exists a phase
$\phi\in\R$ and an inclusion $F_w \into E_{v_0}$ in $\PP(\phi)$ with $v(F_w)=w$ and $E_{v_0}\in S$.
\item\label{enum:ActualWall4} If $C\subset\Stab(X)$ is a connected component of the complement of $\cup_{w\in\Lambda} W^S_w(v_0)$ and $\sigma_1,\sigma_2\in C$, then an object $E_{v_0}\in S$ is $\sigma_1$-stable if and only if it is $\sigma_2$-stable.
\end{enumerate}
In particular, the property for an object to be stable is open in $\Stab(X)$.
\end{prop}

The walls in Proposition \ref{prop:locally_finite} will be called \emph{actual walls}.
A \emph{chamber} is defined to be a connected component of the complement of the set of actual walls, as in \eqref{enum:ActualWall4} above.

\begin{proof}[Idea of the proof]
For a class $w\in\Lambda$, we let $V^S_w$ be the set of stability conditions for which there exists an inclusion as in \eqref{enum:ActualWall3}.
Clearly $V^S_w$ is contained in the numerical wall $W_w(v_0)$.

The first step is to show that there only finitely many $w$ for which $V^S_w$ intersects a small neighborhood around a stability condition.
This is easy to see by using the support property: see also Lemma \ref{lem:convex_cone} below.

We then let $W^S_w(v_0)$ be the codimension one component of $V^S_w$.
It remains to show \eqref{enum:ActualWall4}.
The idea is the following: higher codimension components
of $V^S_w$ always come from objects $E_{v_0}$ that are semistable on this
component and unstable at any nearby point.
\end{proof}

\begin{rmk}\label{rmk:WallChamberGeneralSemistability}
We notice that if $v_0$ is not primitive and we ask only for semistability on chambers, then part \eqref{enum:ActualWall3} cannot be true: namely actual walls may be of higher codimension. On the other hand, the following holds (see \cite[Section 9]{Bri08:stability_k3}): Let $v_0\in\Lambda$. Then there exists a locally finite collection of numerical walls $W_{w}(v_0)$ such that on each chamber the set of semistable objects $\widehat{M}_\sigma(v_0,\phi)$ is constant.
\end{rmk}

The following lemma clarifies how walls behave with respect to the quadratic form appearing in the support property.
It was first observed in \cite[Appendix A]{BMS16:abelian_threefolds}.

\begin{lem}
\label{lem:convex_cone}
Let $Q$ be a quadratic form on a real vector space $V$ and $Z: V \to \C$ a linear map such that the kernel of $Z$ is negative semi-definite with respect to $Q$.
If $\rho$ is a ray in $\C$ starting at the origin, we define
\[ \CC^+_{\rho} = Z^{-1}(\rho) \cap \{ Q \geq 0 \}.\]
\begin{enumerate}
  \item If $w_1, w_2 \in \CC^+_{\rho}$, then $Q(w_1, w_2) \geq 0$.
  \item The set $\CC^+_{\rho}$ is a convex cone.
  \item Let $w, w_1, w_2 \in \CC^+_{\rho}$ with $w = w_1 + w_2$. Then $0 \leq Q(w_1) + Q(w_2) \leq Q(w)$. Moreover, $Q(w_1) = Q(w)$ implies $Q(w) = Q(w_1) = Q(w_2) = Q(w_1, w_2) = 0$.
  \item If the kernel of $Z$ is negative definite with respect to $Q$, then any vector $w \in \CC^+_{\rho}$ with $Q(w) = 0$ generates an extremal ray of $\CC^+$.
\end{enumerate}
\begin{proof}
For any non trivial $w_1, w_2 \in \CC^+_{\rho}$ there is $\lambda > 0$ such that $Z(w_1 - \lambda w_2) = 0$. Therefore, we get
\[
0 \geq Q(w_1 - \lambda w_2) = Q(w_1) + \lambda^2 Q(w_2) - 2\lambda Q(w_1, w_2).
\]
The inequalities  $Q(w_1) \geq 0$ and $Q(w_2) \geq 0$ lead to $Q(w_1, w_2) \geq 0$. Part (2) follows directly from (1). The first claim in (3) also follows immediately from (1) by using
\[
Q(w) = Q(w_1) + Q(w_2) + 2 Q(w_1, w_2) \geq 0.
\]
Observe that $Q(w_1) = Q(w)$ implies $Q(w_2) = Q(w_1, w_2) = 0$ and therefore,
\[
0 = 2\lambda Q(w_1, w_2) \geq Q(w_1) + \lambda^2 Q(w_2) = Q(w_1) \geq 0.
\]

Let $w \in \CC^+_{\rho}$ with $Q(w) = 0$. Assume that $w$ is not extremal, i.e., there are linearly independent $w_1, w_2 \in \CC^+_{\rho}$ such that $w = w_1 + w_2$. By (3) we get $Q(w_1) = Q(w_2) = Q(w_1, w_2) = 0$. As before $Z(w_1 - \lambda w_2) = 0$, but this time $w_1 - \lambda w_2 \neq 0$. That means
\[
0 > Q(w_1 - \lambda w_2) = Q(w_1) + \lambda^2 Q(w_2) - 2\lambda Q(w_1, w_2) = 0. \qedhere
\]
\end{proof}
\end{lem}

\begin{cor}{\cite[Proposition A.8]{BMS16:abelian_threefolds}}
\label{cor:Qzero}
Assume that $U$ is a path-connected open subset of $\Stab(X)$ such that all $\sigma \in U$ satisfy the support property with respect to the quadratic form $Q$. If $E \in \Db(X)$ with $Q(E) = 0$ is $\sigma$-stable for some $\sigma \in U$ then it is $\sigma'$-stable for all $\sigma' \in U$.
\begin{proof}
If there is a point in $U$ at which $E$ is unstable, then there is a point $\sigma = (\AA, Z)$ at which $E$ is strictly semistable. If $\rho$ is the ray that $Z(E)$ lies on, then the previous lemma implies that $v(E)$ is extremal in $\CC^+_{\rho}$. That is contradiction to $E$ being strictly semistable.
\end{proof}
\end{cor}


\section{Stability Conditions on Surfaces}
\label{sec:surfaces}

This is one of the key sections of these notes. We introduce the fundamental operation of tilting for hearts of bounded t-structures and use it to give examples of Bridgeland stability conditions on surfaces.
Another key ingredient is the Bogomolov inequality for slope semistable sheaves, which we will recall as well, and show its close relation with the support property for Bridgeland stability conditions introduced in the previous section.
We then conclude with a few examples of stable objects, explicit description of walls, and the behavior at the large volume limit point.

\subsection{Tilting of t-structures}

Given the heart of a bounded t-structure, the process of tilting is used to obtain a new heart. For a detailed account of the general theory of tilting we refer to \cite{HRS96:tilting} and \cite[Section 5]{BvdB03:functors}.
The idea of using this operation to construct Bridgeland stability conditions is due to Bridgeland in \cite{Bri08:stability_k3}, in the case of K3 surfaces. The extension to general smooth projective surfaces is in \cite{AB13:k_trivial}.

\begin{defn}\label{def:TorsionPair}
Let $\AA$ be an abelian category.
A \emph{torsion pair} on $\AA$ consists of a pair of full additive subcategories $(\FF,\TT)$ of $\AA$ such that the following two properties hold:
\begin{itemize}
\item For any $F \in \FF$ and $T \in \TT$, $\Hom(T,F)=0$.
\item For any $E \in \AA$ there are $F \in \FF$, $T \in \TT$ together with an exact sequence
\[
0 \to T \to E \to F \to 0.
\]
\end{itemize}
\end{defn}

By the vanishing property, the exact sequence in the definition of a torsion pair is unique.

\begin{exercise}
Let $X$ be a smooth projective variety.
Show that the pair of subcategories
\begin{align*}
&\TT := \left\{\text{Torsion sheaves on }X \right\}\\ 
&\FF := \left\{\text{Torsion-free sheaves on }X \right\}
\end{align*}
gives a torsion pair in $\Coh(X)$.
\end{exercise}

\begin{lem}[Happel-Reiten-Smal\o]
Let $X$ be a smooth projective variety.
Let $\AA\subset\Db(X)$ be the heart of a bounded t-structure, and let $(\FF,\TT)$ be a torsion pair in $\AA$.
Then the category
\[
\AA^\sharp := \left\{ E\in\Db(X)\,:\, \begin{array}{l}
\HH_\AA^i(E)=0,\, \mathrm{for\, all\, }i\neq 0,-1\\
\HH^0_\AA(E)\in\TT\\
\HH^{-1}_\AA(E)\in\FF
\end{array} \right\}
\]
is the heart of a bounded t-structure on $\Db(X)$, where $\HH^\bullet_\AA$ denotes the cohomology object with respect to the t-structure $\AA$.
\end{lem}

\begin{proof}
We only need to check the conditions in Definition \ref{def:heart}.
We check the first one and leave the second as an exercise.

Given $E,E'\in\AA^\sharp$, we need to show that $\Hom^{<0}(E,F)=0$.
By definition of $\AA^\sharp$, we have exact triangles
\[
F_E[1] \to E \to T_E \, \, \text{ and } \, \, F_{E'}[1]\to E'\to T_{E'},
\]
where $F_E,F_{E'}\in\FF$ and $T_E,T_{E'}\in\TT$.
By taking the functor $\Hom$, looking at the induced long exact sequences, and using the fact that negative $\Hom$'s vanish for objects in $\AA$, the required vanishing amounts to showing that $\Hom(T_E,F_{E'})=0$. This is exactly the first condition in the definition of a torsion pair.
\end{proof}

\begin{exercise}
Show that $\AA^\sharp$ can also be defined as the smallest extension closed full subcategory of $\Db(X)$ containing both $\TT$ and $\FF[1]$. We will use the notation $\AA^\sharp = \langle \FF[1], \TT \rangle$.
\end{exercise}

\begin{exercise}
\label{exercise:heartIsATilt}
Let $\AA$ and $\BB$ be two hearts of bounded t-structures on $\Db(X)$ such that
$\AA \subset \langle \BB, \BB[1] \rangle$. Show that $\AA$ is a tilt of $\BB$. \emph{Hint: Set $\TT = \BB \cap \AA$ and $\FF = \BB \cap \AA[-1]$ and show that $(\FF, \TT)$ is a torsion pair on $\BB$.}
\end{exercise}

\subsection{Construction of Bridgeland stability conditions on surfaces}
\label{subsec:SurfaceHeart}

Let $X$ be a smooth projective surface over $\C$.
As in Section \ref{sec:generalizations}, we fix an ample divisor class $\omega \in N^1(X)$ and a divisor class $B\in N^1(X)$. We will construct a family of Bridgeland stability conditions  that depends on these two parameters.

As remarked in Example \ref{ex:StabilityCurvesAgain} (2), $\Coh(X)$ will never be the heart of a Bridgeland stability condition. The idea is to use tilting, by starting with coherent sheaves, and use slope stability to define a torsion pair.

Let $\AA=\Coh(X)$.
We define a pair of subcategories
\begin{align*}
\TT_{\omega, B} &= \left\{E \in \Coh(X) : \text{any semistable factor $F$ of $E$ satisfies $\mu_{\omega, B}(F) > 0$} \right\}, \\
\FF_{\omega, B} &=  \left\{E \in \Coh(X) : \text{any semistable factor $F$ of $E$ satisfies $\mu_{\omega, B}(F) \leq 0$} \right\}.
\end{align*}
The existence of Harder-Narasimhan filtrations for slope stability (see Example \ref{ex:ExamplesStability}) and Lemma \ref{lem:Schur} show that this is indeed a torsion pair on $\Coh(X)$.

\begin{defn}
We define the tilted heart
\[
\Coh^{\omega,B}(X) := \AA^\sharp = \langle\FF_{\omega, B}[1], \TT_{\omega, B} \rangle.
\]
\end{defn}

\begin{exercise}
Show that the categories $\TT_{\omega,B}$ and $\FF_{\omega,B}$ can also be defined as follows:
\begin{align*}
\TT_{\omega, B} &=  \left\{E \in \Coh(X) : \forall\, E\onto Q\neq0,\ \mu_{\omega, B}(Q) \leq 0 \right\},\\
\FF_{\omega, B} &= \left\{E \in \Coh(X) : \forall\, 0\neq F\subset E,\ \mu_{\omega, B}(F) > 0 \right\}.
\end{align*}
\end{exercise}

\begin{exercise}
Show that the category $\Coh^{\omega,B}(X)$ depends only on $\frac{\omega}{\omega^2}$ and $\frac{\omega \cdot B}{\omega^2}$.
\end{exercise}

\begin{exercise}\label{exercise:MinimalObjects}
We say that an object $S$ in an abelian category is \emph{minimal}\footnote{This is commonly called a \emph{simple} object. Unfortunately, in the theory of semistable sheaves, the word ``simple'' is used to indicate $\Hom(S,S)=\C$; this is why we use this slightly non-standard notation.} if $S$ does not have any non-trivial subobjects (or quotients).
Show that skyscraper sheaves are minimal objects in the category $\Coh^{\omega,B}(X)$.
Let $E$ be a $\mu_{\omega,B}$-stable vector bundle on $X$ with $\mu_{\omega, B}(E) = 0$. Show that $E[1]$ is a minimal object in $\Coh^{\omega,B}(X)$.
\end{exercise}

We now need to define a stability function on the tilted heart.
We set
\[
Z_{\omega,B} := - \int_X\, e^{\sqrt{-1}\,\omega} \cdot \ch^B.
\]
Explicitly, for $E\in\Db(X)$,
\[
Z_{\omega, B}(E) = \left(-\ch^{B}_2(E) + \frac{\omega^2}{2} \cdot \ch^{B}_0(E)\right) + \sqrt{-1} \ \omega \cdot \ch^{B}_1(E).
\]
The corresponding slope function is
\[
\nu_{\omega, B}(E) = \frac{\ch^{B}_2(E) - \frac{\omega^2}{2} \cdot \ch^{B}_0(E)}{\omega \cdot \ch^{B}_1(E)}.
\]

The main result is the following (see \cite{Bri08:stability_k3, AB13:k_trivial, BM11:local_p2}).
We will choose $\Lambda=K_{\num}(X)$ and $v$ as the Chern character map as in Example \ref{ex:NumericalGrothendieck}.

\begin{thm}\label{thm:BridgelandSurface}
Let $X$ be a smooth projective surface.
The pair $\sigma_{\omega, B} = (\Coh^{\omega,B}(X),Z_{\omega,B})$ gives a Bridgeland stability condition on $X$. Moreover, the map
\[
\Amp(X)\times N^1(X) \to \Stab(X), \, \, (\omega,B)\mapsto\sigma_{\omega,B}
\]
is a continuous embedding.
\end{thm}

Unfortunately the proof of this result is not so direct, even in the case of K3 surfaces.
We will give a sketch in Section \ref{subsec:ProofThmSurface} below.
The idea is to first prove the case in which $\omega$ and $B$ are rational classes.
The non trivial ingredient in this part of the proof is the classical Bogomolov inequality for slope semistable torsion-free sheaves. Then we show we can deform by keeping $B$ fixed and letting $\omega$ vary.
This, together with the behavior for $\omega$ ``large'', will give a Bogomolov inequality for Bridgeland stable objects. This will allow to deform $B$ as well and to show a general Bogomolov inequality for Bridgeland semistable objects, which will finally imply the support property and conclude the proof of the theorem.

The key result in the proof of Theorem \ref{thm:BridgelandSurface} can be summarized as follows.
It is one of the main theorems from \cite{BMT14:stability_threefolds} (see also \cite[Theorem 3.5]{BMS16:abelian_threefolds}).
We first need to introduce three notions of discriminant for an object in $\Db(X)$.

\begin{exercise}\label{exer:ConstantCH}
Let $\omega\in N^1(X)$ be an ample real divisor class.
Then there exists a constant $C_\omega\geq0$ such that, for every effective divisor $D\subset X$, we have
\[
C_\omega (\omega \cdot D)^2 + D^2 \geq 0.
\]
\end{exercise}

\begin{defn}\label{def:discriminant}
Let $\omega,B\in N^1(X)$ with $\omega$ ample.
We define the \emph{discriminant} function as
\[
\Delta := \left(\ch_1^B\right)^2 - 2 \ch_0^B \cdot \ch_2^B = \left(\ch_1\right)^2 - 2 \ch_0 \cdot \ch_2.
\]
We define the \emph{$\omega$-discriminant} as
\[
\overline{\Delta}^B_\omega := \left(\omega \cdot \ch_1^B \right)^2 - 2 \omega^2 \cdot \ch_0^B \cdot \ch_2^B.
\]
Choose a rational non-negative constant $C_\omega$ as in Exercise \ref{exer:ConstantCH} above.
Then we define the \emph{$(\omega,B,C_\omega)$-discriminant} as
\[
\Delta^C_{\omega,B} := \Delta + C_\omega (\omega \cdot \ch_1^B)^2.
\]
\end{defn}

\begin{thm}\label{thm:BogomolovBridgelandStability}
Let $X$ be a smooth projective surface over $\C$.
Let $\omega,B\in N^1(X)$ with $\omega$ ample.
Assume that $E$ is $\sigma_{\omega,B}$-semistable.
Then
\[
\Delta^C_{\omega,B}(E)\geq 0 \, \, \text{ and } \, \, \overline{\Delta}^B_\omega(E) \geq0.
\]
\end{thm}

The quadratic form $\Delta^C_{\omega,B}$ will give the support property for $\sigma_{\omega,B}$.
The quadratic form $\overline{\Delta}^B_\omega$ will give the support property on the $(\alpha,\beta)$-plane (see Section \ref{subsec:WallChamber}).

\subsection{Sketch of the proof of Theorem \ref{thm:BridgelandSurface} and Theorem \ref{thm:BogomolovBridgelandStability}}\label{subsec:ProofThmSurface}

We keep the notation as in the previous section.

\subsubsection*{Bogomolov inequality}

Our first goal is to show that $Z_{\omega,B}$ is a stability function on $\Coh^{\omega,B}(X)$.
The key result we need is the following (\cite{Bog78:inequality, Gie79:bogomolov, Rei78:bogomolov}, \cite[Theorem 12.1.1]{Pot97:lectures_vector_bundles}, \cite[Theorem 3.4.1]{HL10:moduli_sheaves}).

\begin{thm}[Bogomolov Inequality]
Let $X$ be a smooth projective surface.
Let $\omega,B\in N^1(X)$ with $\omega$ ample, and let $E$ be a $\mu_{\omega,B}$-semistable torsion-free sheaf.
Then
\[
\Delta(E) = \ch_1^B(E)^2 - 2 \ch_0^B(E) \cdot \ch_2^B(E) \geq 0.
\]
\end{thm}

\begin{proof}[Sketch of the proof]
Since neither slope stability nor $\Delta$ depend on $B$, we can assume $B=0$.
Also, by \cite[Lemma 4.C.5]{HL10:moduli_sheaves}, slope stability with respect to an ample divisor changes only at integral classes. Hence, we can assume $\omega=H\in\NS(X)$ is the divisor class of a very ample integral divisor (by abuse of notation, we will still denote it by $H$).

Since $\Delta$ is invariant both by tensor products of line bundles and by pull-backs via finite surjective morphisms, and since these two operations preserve slope stability of torsion-free sheaves (see e.g., Exercise \ref{ex:StabilityPullback} for the case of curves, and \cite[Lemma 3.2.2]{HL10:moduli_sheaves} in general), by taking a $\rk(E)$-cyclic cover (see, e.g., \cite[Theorem 3.2.9]{HL10:moduli_sheaves}) and by tensoring by a $\rk(E)$-root of $\ch_1(E)$, we can assume $\ch_1(E)=0$.
Finally, by taking the double-dual, $\Delta$ can only become worse.
Hence, we can assume $E$ is actually a vector bundle on $X$.

Hence, we are reduced to show $\ch_2(E)\leq0$.
We use the restriction theorem for slope stability (see \cite[Section 7.1 or 7.2]{HL10:moduli_sheaves} or \cite{Lan04:positive_char}): up to replacing $H$ with a multiple $uH$, for $u\gg0$ fixed, there exists a smooth projective curve $C\in |H|$ such that $E|_C$ is again slope semistable.
Also, as remarked in Section \ref{subsec:EquivalentDefCurves}, the tensor product $E|_C \otimes \ldots \otimes E|_C$ is still semistable on $C$.

Now the theorem follows by estimating Euler characteristics.
Indeed, on the one hand, by the Riemann-Roch Theorem we have
\[
\chi(X,\underbrace{E\otimes\ldots\otimes E}_{m\text{-times}}) = \rk(E)^m \chi(X,\OO_X) + m \rk(E)^{m-1}\ch_2(E),
\]
for all $m>0$.

On the other hand, we can use the exact sequence
\[
0 \to \left( E\otimes\ldots\otimes E\right) \otimes \OO_X(-H) \to E\otimes\ldots\otimes E \to E|_C\otimes\ldots\otimes E|_C \to 0,
\]
and, by using the remark above, deduce that
\[
h^0(X,E\otimes\ldots\otimes E) \leq h^0(C,E|_C\otimes\ldots\otimes E|_C) \leq \gamma_E \rk(E)^m,
\]
for a constant $\gamma_E>0$ which is independent on $m$.
Similarly, by Serre duality, and by using an analogous argument, we have
\[
h^2(X,E\otimes\ldots\otimes E) \leq \delta_E \rk(E)^m,
\]
for another constant $\delta_E$ which depends only on $E$.

By putting all together, since
\[
\chi(X,E\otimes\ldots\otimes E)\leq h^0(X,E\otimes\ldots\otimes E) + h^2(X,E\otimes\ldots\otimes E) \leq (\gamma_E + \delta_E) \rk(E)^m,
\]
we deduce
\[
\chi(X,\OO_X) + m \frac{\ch_2(E)}{\rk(E)} \leq \gamma_E + \delta_E,
\]
for any $m>0$, namely $\ch_2(E)\leq0$, as we wanted.
\end{proof}

\begin{rmk}\label{rmk:TorsionSheavesBogomolov}
Let $E$ be a torsion sheaf.
Then
\[
\Delta^C_{\omega,B}(E) \geq 0.
\]
Indeed, $\ch_1^B(E)=\ch^1(E)$ is an effective divisor.
Then the inequality follows directly by the definition of $C_\omega$ in Exercise \ref{exer:ConstantCH}.
\end{rmk}

We can now prove our result:

\begin{prop}\label{prop:StabilityFunctionSurface}
The group homomorphism $Z_{\omega,B}$ is a stability function on $\Coh^{\omega,B}(X)$.
\end{prop}

\begin{proof}
By definition, it is immediate to see that $\Im Z_{\omega,B}\geq0$ on $\Coh^{\omega,B}(X)$.
Moreover, if $E\in\Coh^{\omega,B}(X)$ is such that $\Im Z_{\omega,B}(E)=0$, then $E$ fits in an exact triangle
\[
\HH^{-1}(E)[1] \to E \to \HH^0(E),
\]
where $\HH^{-1}(E)$ is a $\mu_{\omega,B}$-semistable torsion-free sheaf with $\mu_{\omega,B}(\HH^{-1}(E))=0$ and $\HH^0(E)$ is a torsion sheaf with zero-dimensional support. Here, as usual, $\HH^\bullet$ denotes the cohomology sheaves of a complex.

We need to show $\Re Z_{\omega,B}(E) <0$.
Since $\Re Z_{\omega,B}$ is additive, we only need to show this for its cohomology sheaves.
But, on the one hand, since $\HH^0(E)$ is a torsion sheaf with support a zero-dimensional subscheme, we have
\[
\Re Z_{\omega,B}(\HH^0(E)) = - \ch_2(\HH^0(E)) < 0.
\]
On the other hand, by the Hodge Index Theorem, since $\omega \cdot \ch_1^B(\HH^{-1}(E))=0$, we have $\ch_1^B(\HH^{-1}(E))^2\leq0$. Therefore, by the Bogomolov inequality, we have $\ch_2^B(\HH^{-1}(E))\leq0$.
Hence,
\[
\Re Z_{\omega,B}(\HH^{-1}(E)[1]) = - \Re Z_{\omega,B}(\HH^{-1}(E)) = \underbrace{\ch_2^B(\HH^{-1}(E))}_{\leq 0} - \underbrace{\frac{\omega^2}{2} \ch_0^B(\HH^{-1}(E))}_{>0} < 0,
\]
thus concluding the proof.
\end{proof}

\subsubsection*{The rational case}

We now let $B \in \NS_\Q$ be a rational divisor class, and we let $\omega=\alpha H$, where $\alpha\in\R$ and $H\in \NS(X)$ is an integral ample divisor class.

We want to show that the group homomorphism $Z_{\omega,B}$ gives a stability condition in the abelian category $\Coh^{\omega,B}(X)$, in the sense of Definition \ref{defn:stability_abelian}. By using Proposition \ref{prop:StabilityFunctionSurface} above, we only need to show that HN filtrations exist.
To this end, we use Proposition \ref{prop:hn_exists}, and since $\Im Z_{\omega,B}$ is discrete under our assumptions, we only need to show the following.

\begin{lem}\label{lem:TiltNoetherian}
Under the previous assumption, the tilted category $\Coh^{\omega,B}(X)$ is noetherian.
\end{lem}

\begin{proof}
This is a general fact about tilted categories. It was first observed in the case of K3 surfaces in \cite[Proposition 7.1]{Bri08:stability_k3}. We refer to \cite[Lemma 2.17]{PT19:bridgeland_moduli_properties} for all details and just give a sketch of the proof.
As observed above, under our assumptions, $\Im Z_{\omega,B}$ is discrete and non-negative on $\Coh^{\omega,B}(X)$.
This implies that it is enough to show that, for any $M\in\Coh^{\omega,B}(X)$, there exists no infinite filtration
\[
0 = A_0 \subsetneq A_1 \subsetneq \ldots \subsetneq A_l \subsetneq \ldots \subset M
\]
with $\Im Z_{\omega,B}(A_l)=0$.

Write $Q_l:=M/A_l$.
By definition, the exact sequence
\[
0 \to A_l \to M \to Q_l \to 0
\]
in $\Coh^{\omega,B}(X)$ induces a long exact sequence of sheaves
\[
0 \to \HH^{-1}(A_l) \to \HH^{-1}(M) \to \HH^{-1}(Q_l) \to \HH^{0}(A_l) \to \HH^{0}(M) \to \HH^{0}(Q_l)\to 0.
\]
Since $\Coh(X)$ is noetherian, we have that both sequences
\begin{align*}
&\HH^0(M) = \HH^0(Q_0) \onto \HH^0(Q_1) \onto \HH^0(Q_2) \onto \ldots \\
&0=\HH^{-1}(A_0)\into\HH^{-1}(A_1) \into \HH^{-1}(A_2) \into \ldots \into \HH^{-1}(M)
\end{align*}
stabilize.
Hence, from now on we will assume $\HH^0(Q_l)=\HH^0(Q_{l+1})$ and $\HH^{-1}(A_l)=\HH^{-1}(A_{l+1})$ for all $l$.
Both $U:=\HH^{-1}(M)/\HH^{-1}(A_l)$ and $V:=\Ker(\HH^0(M)\onto\HH^0(Q_l))$ are constant and we have an exact sequence
\begin{equation}\label{eq:noetherianity1}
0 \to U \to \HH^{-1}(Q_l) \to \HH^0(A_l) \to V \to 0,
\end{equation}
again by definition of $\Coh^{\omega,B}(X)$, since $\Im Z_{\omega,B}(A_l)=0$, $\HH^0(A_l)$ is a torsion sheaf supported in dimension $0$. Write $B_l:=A_l/A_{l-1}$.
Again, by looking at the induced long exact sequence of sheaves and the previous observation, we deduce that
\[
0 \to \HH^{-1}(A_{l-1}) \xrightarrow{\cong} \HH^{-1}(A_l) \xrightarrow{0} \underbrace{\HH^{-1}(B_l)}_{\text{torsion-free}\Rightarrow =0} \to \underbrace{\HH^{0}(A_{l-1})}_{\text{torsion}} \to \HH^{0}(A_l) \to \HH^{0}(B_l)\to 0.
\]
Hence, the only thing we need is to show that $\HH^0(B_l)=0$, for all $l\gg 0$, or equivalently to bound the length of $\HH^0(A_l)$. But, by letting $K_l:=\Ker (\HH^0(A_l)\onto V)$, \eqref{eq:noetherianity1} gives an exact sequence
of sheaves
\[
0 \to U \to \HH^{-1}(Q_l) \to K_l \to 0.
\]
where $K_l$ is a torsion sheaf supported in dimension $0$ as well and $\HH^{-1}(Q_l)$ is torsion-free.
This gives a bound on the length of $K_l$ and therefore a bound on the length of $\HH^0(A_l)$, as we wanted.
\end{proof}

To finish the proof that, under our rationality assumptions, $\sigma_{\omega,B}$ is a Bridgeland stability condition, we still need to prove the support property, namely Theorem \ref{thm:BogomolovBridgelandStability} in our case. The idea is to let $\alpha \to \infty$ and use the Bogomolov theorem for sheaves (together with Remark \ref{rmk:TorsionSheavesBogomolov}).
First we use the following lemma (we will give a more precise statement in Section \ref{subsec:WallChamber}).
It first appeared in the K3 surface case as a particular case of \cite[Proposition 14.2]{Bri08:stability_k3}.

\begin{lem}\label{lem:large_volume_limit_tilt}
Let $\omega,B \in N^1(X)$ with $\omega$ ample.
If $E \in \Coh^{\omega, B}(X)$ is $\sigma_{\alpha\cdot \omega, B}$-semistable
for all $\alpha \gg 0$, then it satisfies one of the following conditions:
\begin{enumerate}
\item $\HH^{-1}(E)=0$ and $\HH^0(E)$ is a $\mu_{\omega,B}$-semistable torsion-free sheaf.
\item $\HH^{-1}(E)=0$ and $\HH^0(E)$ is a torsion sheaf.
\item $\HH^{-1}(E)$ is a $\mu_{\omega,B}$-semistable torsion-free sheaf and $\HH^0(E)$ is either $0$ or a torsion sheaf supported in dimension zero.
\end{enumerate}
\end{lem}

\begin{proof}
One can compute $\sigma_{\alpha \omega, B}$-stability with slope $\tfrac{2\nu_{\alpha \omega, B}}{\alpha}$ instead of $\nu_{\alpha \omega, B}$. This is convenient in the present argument because
\[
\lim_{\alpha \to \infty} \tfrac{2\nu{\alpha \omega, B}}{\alpha}(E) = - \mu_{\omega, B}^{-1}(E).
\]
By definition of $\Coh^{\omega, B}(X)$ the object $E$ is an extension $0 \to F[1] \to E \to T \to 0$,
where $F \in \FF_{\omega, B}$ and $T \in \TT_{\omega, B}$.

Assume that $\omega \cdot \ch_1^B(E) = 0$. Then both $\omega \cdot \ch_1^B(F) = 0$ and $\omega \cdot \ch_1^B(T) = 0$. By definition of $\FF_{\omega, B}$ and $\TT_{\omega, B}$ this means $T$ is $0$ or has be supported in dimension $0$ and $F$ is $0$ or a $\mu_{\omega, B}$-semistable torsion free sheaf. Therefore, for the rest of the proof we can assume $\omega \cdot \ch_1^B(E) > 0$.

Assume that $\ch_0^B(E) \geq 0$. By definition $-\mu_{\omega, B}^{-1}(F[1]) \geq 0$. The inequality $\omega \cdot \ch_1^B(E) > 0$ implies $-\mu_{\omega, B}^{-1}(E) < 0$. Since $E$ is $\sigma_{\alpha \omega, B}$-semistable for $\alpha \gg 0$, we get $F = 0$ and $E \in \TT_{\omega, B}$ is a sheaf. If $E$ is torsion, we are in case (2). Assume $E$ is neither torsion nor slope semistable. Then by definition of $\TT_{\omega, B}$ there is an exact sequence
\[
0 \to A \to E \to B \to 0
\]
in $\TT_{\omega, B} = \Coh^{\omega, B}(X) \cap \Coh(X)$ such that $\mu_{\omega, B}(A) > \mu_{\omega, B}(E) > 0$. But then $-\mu_{\omega, B}^{-1} (A) > -\mu_{\omega, B}^{-1} (E)$ contradicts the fact that $E$ is $\sigma_{\alpha \omega, B}$-semistable for $\alpha \gg 0$.

Assume that $\ch_0^B(E) < 0$. If $\omega \cdot \ch_1^B(T) = 0$, then $T \in \TT_{\omega, B}$ implies $\ch_0^B(T) = 0$ and up to semistability of $F$ we are in case (3). Let $\omega \cdot \ch_1^B(T) > 0$. By definition $-\mu_{\omega, B}^{-1}(T) < 0$. As before we have $-\mu_{\omega, B}^{-1}(T) > 0$ and the fact that $E$ is $\sigma_{\alpha \omega, B}$-semistable for $\alpha \gg 0$ implies $T = 0$. In either case we are left to show that $F$ is $\mu_{\omega, B}$-semistable. If not, there is an exact sequence
\[
0 \to A \to F \to B \to 0
\]
in $\FF_{\omega, B} = \Coh^{\omega, B}(X)[-1] \cap \Coh(X)$ such that $\mu_{\omega, B}(A) > \mu_{\omega, B}(F)$. Therefore, there is an injective map $A[1] \into E$ in $\Coh^{\omega, B}(X)$ such that $-\mu_{\omega, B}^{-1}(A[1]) > -\mu_{\omega, B}^{-1}(E)$ in contradiction to the fact that $E$ is $\sigma_{\alpha \omega, B}$-semistable for $\alpha \gg 0$.
\end{proof}

\begin{exercise}\label{exer:SmallerImaginaryPart}
Let $B \in \NS(X)_\Q$ be a rational divisor class and let $\omega = \alpha H$, where $\alpha\in\R_{>0}$ and $H \in \NS(X)_\Q$ is an integral ample divisor.
Show that
\[
c:= \min \left\{ H \cdot \ch_1^B(F) \, :\, \begin{array}{l} F\in\Coh^{\omega,B}(X)\\ H \cdot \ch_1^B(F)>0\end{array}\right\} > 0.
\]
exists.
Let $E\in\Coh^{\omega,B}(X)$ satisfy $H \cdot \ch_1^B(E)=c$ and $\Hom(A,E)=0$, for all $A\in\Coh^{\omega,B}(X)$ with $H \cdot \ch_1^B(A)=0$. Show that $E$ is $\sigma_{\omega,B}$-stable.
\end{exercise}

We can now prove Theorem \ref{thm:BogomolovBridgelandStability} (and so Theorem \ref{thm:BridgelandSurface}) in the rational case.

\begin{proof}[Proof of Theorem \ref{thm:BogomolovBridgelandStability}, rational case]
Since $B$ is rational and $\omega = \alpha H$, with $H$ being an integral divisor, the imaginary part is a discrete function. We proceed by induction on $H \cdot \ch_1^B$.

Let $E\in\Coh^{H,B}(X)$ be $\sigma_{\alpha_0 H,B}$-semistable for some $\alpha_0>0$ such that $H \cdot \ch_1^B>0$ is minimal.
Then by Exercise \ref{exer:SmallerImaginaryPart}, $E$ is stable for all $\alpha\gg0$.
Therefore, by Lemma \ref{lem:large_volume_limit_tilt}, both $\Delta^C_{\omega,B}(E)\geq 0$ and $\overline{\Delta}^B_\omega(E) \geq0$ hold, since they are true for semistable sheaves.

The induction step is subtle: we would like to use the wall and chamber structure.
If an object $E$ is stable for all $\alpha\gg0$, then again by Lemma \ref{lem:large_volume_limit_tilt} we are done.
If not, it is destabilized at a certain point. By Lemma \ref{lem:convex_cone}, we can then proceed by induction on $H \cdot \ch_1^B$. The issue is that we do not know the support property yet and therefore, that walls are indeed well behaved.
Fortunately, as $\alpha$ increases, all possible destabilizing subobjects and quotients have strictly smaller $H \cdot \ch_1^B$, which satisfy the desired inequality by our induction assumption.
This is enough to ensure that $E$ satisfies well-behaved wall-crossing. By following the same argument as in Proposition \ref{prop:locally_finite} (and Remark \ref{rmk:WallChamberGeneralSemistability}), it is enough to know a support property type statement for all potentially destabilizing classes.
\end{proof}

\subsubsection*{The general case}

We only sketch the argument for the general case.
This is explained in detail in the case of K3 surfaces in \cite{Bri08:stability_k3}, and can be directly deduced from \cite[Sections 6 \& 7]{Bri07:stability_conditions}.

The idea is to use the support property we just proved in the rational case to deform $B$ and $\omega$ in all possible directions, by using Bridgeland's Deformation Theorem \ref{thm:BridgelandMain} (and Proposition \ref{prop:ExplicitBridgelandMain}).
The only thing we have to make sure is that once we have the correct central charge, the category is indeed $\Coh^{\omega,B}(X)$.
The intuition behind this is in the following result.

\begin{lem}
Let $\sigma=(\AA, Z_{\omega, B})$ be a stability condition satisfying the support property such that all skyscraper sheaves are stable of phase one. Then $\AA = \Coh^{\omega, B}(X)$ holds.
\end{lem}

\begin{proof}
We start by showing that all objects $E \in \AA$ have $H^i(E) = 0$ whenever $i \neq 0, -1$ and that $H^{-1}(E)$ is torsion free. Notice that $E$ is an iterated extension of stable objects, so we can assume that $E$ is stable. This is clearly true for skyscraper sheaves, so we can assume that $E$ is not a skyscraper sheaf. Then Serre duality implies that for $i \neq 0, 1$ and any $x \in X$ we have
\[
\Ext^i(E, \C(x)) = \Ext^{2-i}(\C(x), E) = 0.
\]
By Proposition \ref{prop:locally_free_complex} we get that $E$ is isomorphic to a two term complex of locally free sheaves and the statement follows. This also implies that $H^{-1}(E)$ is torsion free.

Therefore, the inclusion $\AA \subset \langle \Coh(X), \Coh(X)[1] \rangle$ holds. Set $\TT = \Coh(X) \cap \AA$ and $\FF = \Coh(X) \cap \AA[-1]$. By Exercise \ref{exercise:heartIsATilt} we get $\AA = \langle \TT, \FF[1] \rangle$ and $\Coh(X) = \langle \TT, \FF \rangle$. We need to show $\TT = \TT_{\omega, B}$ and $\FF = \FF_{\omega, B}$. In fact, it will be enough to show $\TT_{\omega, B} \subset \TT$ and $\FF_{\omega, B} \subset \FF$.

Let $E \in \Coh(X)$ be slope semistable. There is an exact sequence $0 \to T \to E \to F \to 0$ with $T \in \TT$ and $F \in \FF$. We already showed that $F = H^{-1}(F)$ is torsion free. If $E$ is torsion this implies $F = 0$ and $E = T \in \TT$ as claimed.

Assume that $E$ is torsion free. Since $F[1] \in \AA$ and $T \in \AA$, we get $\mu_{\omega, B}(F) \leq 0$ and $\mu_{\omega, B}(T) \geq 0$. This is a contradiction to $E$ being stable unless either $F = 0$ or $T = 0$. Therefore, we showed $E \in \FF$ or $E \in \TT$.

If $\omega \cdot \ch_1^B(E) > 0$, then $E \in \TT$ and if $\omega \cdot \ch_1^B(E) < 0$, then $E \in \FF$. Assume that $\omega \cdot \ch_1^B(E) = 0$, but $E \in \TT$. Then by definition of a stability condition we have $Z_{\omega, B}(E) \in \R_{<0}$ and $E$ is $\sigma$-semistable. Since $E$ is a sheaf, there is a skyscraper sheaf $\C(x)$ together with a surjective morphism of coherent sheaves $E \onto \C(x)$. Since $\C(x)$ is stable of slope $\infty$ this morphism is also a surjection in $\AA$. Let $F \in \AA \cap \Coh(X) = \TT$ be the kernel of this map. Then $Z(F) = Z(E) + 1$. Iterating this procedure will lead to an object $F$ with $Z(F) \in \R_{\geq 0}$, a contradiction.
\end{proof}


\subsection{The wall and chamber structure in the $(\alpha,\beta)$-plane}\label{subsec:WallChamber}

If we consider a certain slice of the space of Bridgeland stability conditions, the structure of the walls turns out rather simple.

\begin{defn}\label{def:alphabetaPlane}
Let $H\in\NS(X)$ be an ample integral divisor class and let $B_0\in \NS_\Q$.
We define the $(\alpha,\beta)$-plane as the set of stability conditions of the form
$\sigma_{\alpha H, B_0 + \beta H}$, for $\alpha,\beta\in\R$, $\alpha>0$.
\end{defn}

When it is clear from context which $(\alpha,\beta)$ plane we choose (for example, if the Picard number is one), we will abuse notation and drop $H$ and $B_0$ from the notation; for example, we denote stability conditions by $\sigma_{\alpha,\beta}$, the twisted Chern character by $\ch^\beta$, etc.

The following proposition describes all walls in the $(\alpha,\beta)$-plane.
It was originally observed  by Bertram and completely proved in \cite{Mac14:nested_wall_theorem}.

\begin{prop}[Structure Theorem for Walls on Surfaces]\label{prop:StructureThmWallsSurfaces}
Fix a class $v \in K_{\num}(X)$.
\begin{enumerate}
\item All numerical walls are either semicircles with center on the $\beta$-axis or vertical rays.
\item \label{item:walls_dont_intersect}
Two different numerical walls for $v$ cannot intersect.
\item For a given class $v \in K_{\num}(X)$ the hyperbola $\Re Z_{\alpha, \beta}(v) = 0$ intersects all numerical semicircular walls at their top points.
\item If $\ch_0^{B_0}(v) \neq 0$, then there is a unique numerical vertical wall defined by the equation
\[
\beta = \frac{H \ch_1^{B_0}(v)}{H^2 \ch_0^{B_0}(v)}.
\]
\item If $\ch_0^{B_0}(v) \neq 0$, then all semicircular walls to either side of the unique numerical vertical wall are strictly nested semicircles.
\item If $\ch_0^{B_0}(v) = 0$, then there are only semicircular walls that are strictly nested.
\item \label{item:walls_dont_stop}
If a wall is an actual wall at a single point, it is an actual wall everywhere along the numerical wall. 
\end{enumerate}
\end{prop}

\begin{exercise}
Prove Proposition \ref{prop:StructureThmWallsSurfaces}. \emph{(Hint: For (\ref{item:walls_dont_intersect}) ignore the slope and rephrase everything with just $Z$ using linear algebra. For (\ref{item:walls_dont_stop}) show that a destabilizing subobject or quotient would have to destabilize itself at some point of the wall.)}
\end{exercise}

As application, it is easy to show that a ``largest wall'' exists and we will also be able to prove that walls are locally finite in the surface case without using Proposition \ref{prop:locally_finite}. Both will directly follow from the following statement that is taken from \cite[Lemma 5.6]{Sch15:stability_threefolds}.

\begin{lem}
Let $v \in K_{\num}(X)$ be a non-zero class such that $\overline{\Delta}_H^{B_0}(v) \geq 0$. For fixed $\beta_0 \in \Q$ there are only finitely many walls intersecting the vertical line $\beta = \beta_0$.
\end{lem}

\begin{proof}
Any wall has to come from an exact sequence $0 \to F \to E \to G \to 0$ in $\Coh^{\beta}(X)$. Let $(H^2 \cdot  \ch_0^{\beta}(E), H \cdot \ch_1^{\beta}(E), \ch_2^{\beta}(E)) = (R, C, D)$ and $(H^2 \cdot  \ch_0^{\beta}(F), H \cdot \ch_1^{\beta}(F), \ch_2^{\beta}(F)) = (r, c, d)$. Notice that due to the fact that $\beta \in \Q$ the possible values of $r$, $c$ and $d$ are discrete in $\R$. Therefore, it will be enough to bound those values to obtain finiteness.

By definition of $\Coh^{\beta}(X)$ one has $0 \leq c \leq C$. If $C = 0$, then $c = 0$ and we are dealing with the unique vertical wall. Therefore, we may assume $C \neq 0$. Let $\Delta := C^2 - 2RD$. The Bogomolov inequality together with Lemma \ref{lem:convex_cone} implies $0 \leq c^2 - 2rd \leq \Delta$. Therefore, we get
\[
\frac{c^2}{2} \geq rd \geq \frac{c^2 - \Delta}{2}.
\]
Since the possible values of $r$ and $d$ are discrete in $\R$, there are finitely many possible values unless $r = 0$ or $d = 0$.

Assume $R = r = 0$. Then the equality $\nu_{\alpha, \beta}(F) = \nu_{\alpha,
\beta}(E)$ holds if and only if $Cd-Dc = 0$. In particular, it is independent of
$(\alpha, \beta)$. Therefore, the sequence does not define a wall.

If $r = 0$, $R \neq 0$, and $D - d \neq 0$, then using the same type of inequality for $G$ instead of $E$ will finish the proof. If $r = 0$ and $D - d = 0$, then $d = D$ and there are are only finitely many walls like this, because we already bounded $c$.

Assume $D = d = 0$. Then the equality $\nu_{\alpha, \beta}(F) = \nu_{\alpha,
\beta}(E)$ holds if and only if $Rc - Cr = 0$. Again this cannot define a wall.

If $d = 0$, $D \neq 0$, and $R - r \neq 0$, then using the same type of inequality for $G$ instead of $E$ will finish the proof. If $d = 0$ and $R - r = 0$, then $r = R$ and there are are only finitely many walls like this, because we already bounded $c$.
\end{proof}

\begin{cor}
\label{cor:LargestWallExists}
Let $v \in K_{\num}(X)$ be a non-zero class such that $\overline{\Delta}_H^{B_0}(v)\geq 0$. Then semicircular walls in the $(\alpha, \beta)$-plane with respect to $v$ are bounded from above.
\end{cor}

\begin{cor}
\label{cor:locally_finite_surfaces}
Let $v \in K_{\num}(X)$ be a non-zero class such that $\overline{\Delta}_H^{B_0}(v)\geq 0$. Walls in the $(\alpha, \beta)$-plane with respect to $v$ are locally finite, i.e., any compact subset of the upper half plane intersects only finitely many walls.
\end{cor}

We will compute the largest wall in examples: see Section \ref{subsec:LargestWallHilbertScheme}. An immediate consequence of Proposition \ref{cor:LargestWallExists} is the following precise version of Lemma \ref{lem:large_volume_limit_tilt}. This was proved in the case of K3 surfaces in \cite[Proposition 14.2]{Bri08:stability_k3}. The general proof is essentially the same if the statement is correctly adjusted.

\begin{exercise}
\label{exercise:BridgelandvsGieseker}
Let $v \in K_{\num}(X)$ be a non-zero class with positive rank and let $\beta \in \R$ such that $H \cdot \ch_1^{\beta}(v) > 0$. Then there exists $\alpha_0>0$ such that for any $\alpha > \alpha_0$ the set of $\sigma_{\alpha, \beta}$-semistable objects with class $v$ is the same as the set of twisted $(\omega, B_0 - \tfrac{1}{2} K_X)$-Gieseker semistable sheaves with class $v$. Moreover, $\sigma_{\alpha, \beta}$-stable objects of class $v$ are the same as twisted $(\omega, B_0 - \tfrac{1}{2} K_X)$-Gieseker stable sheaves with class $v$. \emph{Hint: Follow the proof of Lemma \ref{lem:large_volume_limit_tilt} and compare lower terms.}
\end{exercise}

\subsection{Examples of semistable objects}\label{subsec:ExampleSemistable}

We already saw a few easy examples of stable objects with respect to $\sigma_{\omega,B}$: skyscraper sheaves and objects with minimal $H \cdot \ch_1^B$ (or with $H \cdot \ch_1^B=0$).
See Exercise \ref{exercise:MinimalObjects} and Exercise \ref{exer:SmallerImaginaryPart}.

The key example of Bridgeland semistable objects are those with trivial discriminant.

\begin{lem}
\label{lem:parallel_line_bundle}
Let $E$ be a $\mu_{\omega,B}$-stable vector bundle.
Assume that either $\Delta_{\omega,B}^C(E)=0$, or $\overline{\Delta}_H^B(E)=0$.
Then $E$ is $\sigma_{\omega,B}$-stable.
\end{lem}

\begin{proof}
We can assume $\omega$ and $B$ to be rational.
Consider the $(\alpha,\beta)$-plane.
By Exercise \ref{exercise:BridgelandvsGieseker}, $E$ (or $E[1]$) is stable for $\alpha\gg0$.
The statement now follows directly from Corollary \ref{cor:Qzero}.
\end{proof}

In particular, all line bundles are stable everywhere only if the N\'eron-Severi group is of rank one, or if the constant $C_\omega$ of Exercise \ref{exer:ConstantCH} is zero (e.g., for abelian surfaces).

\begin{exercise}
Let $\sigma\colon X\to\P^2$ be the blow up of $\P^2$ at one point.
Let $h$ be the pull-back $\sigma^*\OO_{\P^2}(1)$ and let $f$ denote the fiber of the $\P^1$-bundle $X\to\P^1$.
Consider the anti-canonical line bundle $H:= - K_X = 2h+f$, and consider the $(\alpha,\beta)$-plane with respect to $\omega=\alpha H$, $B=\beta H$.
Show that $\OO_X(h)$ is $\sigma_{\alpha,\beta}$-stable for all $\alpha,\beta$, while there exists $(\alpha_0,\beta_0)$ for which $\OO(2h)$ is not $\sigma_{\alpha_0,\beta_0}$-semistable.
\end{exercise}

\subsection{Moduli spaces}\label{subsec:ModuliSurfaces}

We keep the notation as in the beginning of this section, namely $\omega,B\in N^1(X)$ with $\omega$ ample.
We fix $v_0\in K_{\num}(X)$ and $\phi\in\R$.
Consider the stack $\mathfrak{M}_{\omega,B}(v_0,\phi):=\mathfrak{M}_{\sigma_{\omega,B}}(v_0,\phi)$ (and $\mathfrak{M}_{\omega,B}^s(v_0,\phi)$) as in Section \ref{subsec:moduli}.

\begin{thm}[Toda]
$\mathfrak{M}_{\omega,B}(v_0,\phi)$ is a universally closed  Artin stack of finite type over $\C$.
Moreover, $\mathfrak{M}_{\omega,B}^s(v_0,\phi)$ is a $\G_m$-gerbe over an algebraic space $M_{\omega,B}(v_0,\phi)$. Finally, if $\mathfrak{M}_{\omega,B}(v_0,\phi)=\mathfrak{M}_{\omega,B}^s(v_0,\phi)$, then $M_{\omega,B}(v_0,\phi)$ is a proper algebraic space over $\C$.
\end{thm}

\begin{proof}[Ideas from the proof]
As we observed after Question \ref{question:OpenAndBounded}, we only need to show openness and boundedness.
The idea to show boundedness is to reduce to semistable sheaves and use boundedness from them.
For openness, the key technical result is a construction by Abramovich-Polishchuk \cite{AP06:constant_t_structures}, which we will recall in Theorem \ref{thm:constant_t_structure}. Then openness for Bridgeland semistable objects follows from the existence of relative Harder-Narasimhan filtrations for sheaves. 
We refer to \cite{Tod08:K3Moduli} for all the details.
\end{proof}

There are only few examples where we can be more precise on the existence of moduli spaces. We will present a few of them below.
To simplify notation, we will drop from now on the phase $\phi$ from the notation for a moduli space.

\subsubsection*{The projective plane}
Let $X=\P^2$.
We identify the lattice $K_0(\P^2)=K_{\num}(\P^2)$ with $\Z^{\oplus 2}\oplus \frac{1}{2}\Z$, and the Chern character with a triple $\ch=(r,c,s)$, where $r$ is the rank, $c$ is the first Chern character, and $s$ is the second Chern character.
We fix $v_0\in K_{\num}(\P^2)$.

\begin{thm}
\label{thm:P2}
For all $\alpha,\beta\in\R$, $\alpha>0$, there exists a coarse moduli space $M_{\alpha,\beta}(v_0)$ parameterizing S-equivalence classes of $\sigma_{\alpha,\beta}$-semistable objects.
It is a projective variety.
Moreover, if $v_0$ is primitive and $\sigma_{\alpha,\beta}$ is outside a wall for $v_0$, then $M_{\alpha,\beta}^s(v_0)=M_{\alpha,\beta}(v_0)$ is a smooth irreducible projective variety.
\end{thm}

The projectivity was first observed in \cite{ABCH13:hilbert_schemes_p2}, while generic smoothness is proved in this generality in \cite{LZ19:NewStabilityP2}. For the proof a GIT construction is used, but with a slightly different GIT problem.
First of all, an immediate consequence of Proposition \ref{prop:StructureThmWallsSurfaces} is the following (we leave the details to the reader).

\begin{exercise}\label{exercise:SmallAlphaP2}
Given $\alpha,\beta\in\R$, $\alpha>0$, there exist $\alpha_0,\beta_0\in\R$, such that $0<\alpha_0<\frac{1}{2}$ and $\widehat{M}_{\alpha,\beta}(v_0)=\widehat{M}_{\alpha_0,\beta_0}(v_0)$.
\end{exercise}

For $0<\alpha<\frac{1}{2}$, Bridgeland stability on the projective plane is related to finite dimensional algebras, where moduli spaces are easy to construct via GIT (see Exercise \ref{ex:Quiver}).
For $k\in\Z$, we consider the vector bundle $\EE:=\OO_{\P^2}(k-1)\oplus \Omega_{\P^2}(k+1)\oplus \OO_{\P^2}(k)$.
Let $A$ be the finite-dimensional associative algebra $\End(E)$.
Then, by the Beilinson Theorem \cite{Bei78:exceptional_collection_pn}, the functor
\[
\Phi_\EE: \Db(\P^2) \to \Db(\text{mod-}A), \, \, \Phi_\EE(F) := \mathbf{R}\Hom(\EE,F)
\]
is an equivalence of derived categories.
The simple objects in the category $\Phi_\EE^{-1}(\text{mod-}A)$ are $\OO_{\PP^2}(k-3)[2]$, $\OO_{\PP^2}(k-2)[1]$, $\OO_{\PP^2}(k-1)$.

\begin{exercise}\label{exercise:quiverP2}
For $\alpha,\beta\in\R$, $0<\alpha<\frac{1}{2}$, there exist $k\in\Z$ and an element $A\in\widetilde{\GL}^+(2,\R)$ such that $A\cdot \sigma_{\alpha,\beta}$ is a a stability condition as in Exercise \ref{ex:Quiver}.
\end{exercise}

It is not hard to prove that families of Bridgeland semistable objects coincide with families of modules over an algebra, as defined in \cite{Kin94:moduli_quiver_reps}.
Therefore, this proves the first part of Theorem \ref{thm:P2}. The proof of smoothness in the generic primitive case is more subtle. The complete argument is given in \cite{LZ19:NewStabilityP2}.

\subsubsection*{K3 surfaces}
Let $X$ be a K3 surface.
We consider the algebraic Mukai lattice
\[
H^*_{\mathrm{alg}}(X):=H^0(X,\Z)\oplus \NS(X) \oplus H^4(X,\Z)
\]
together with the Mukai pairing
\[
\left( (r,c,s),(r',c',s') \right) = c.c' - rs' - sr'.
\]
We also consider the \emph{Mukai vector}, a modified version of the Chern character by the square root of the Todd class:
\[
v := \ch \cdot \sqrt{\td_X} = \left(\ch_0, \ch_1, \ch_2 + \ch_0 \right).
\]
Our lattice $\Lambda$ is $H^*_{\mathrm{alg}}(X)$ and the map $v$ is nothing but the Mukai vector.

Bridgeland stability conditions for K3 surfaces can be described beyond the ones we just defined using $\omega,B$.
More precisely, the main result in \cite{Bri08:stability_k3} is that there exists a connected component $\Stab^\dagger(X)$, containing all stability conditions $\sigma_{\omega,B}$ such that the map
\[
\eta\colon \Stab^\dagger(X) \xrightarrow{\ZZ} \Hom(H^*_{\mathrm{alg}}(X),\C) \xrightarrow{(-,-)} H^*_{\mathrm{alg}}(X)_\C
\]
is a covering onto its image, which can be described as a certain period domain.

\begin{thm}\label{thm:K3ModuliProjective}
Let $v_0\in H^*_{\mathrm{alg}}(X)$.
Then for all generic stability conditions $\sigma\in\Stab^\dagger(X)$, there exists a coarse moduli space $M_{\sigma}(v_0)$ parameterizing S-equivalence classes of $\sigma$-semistable objects.
It is a projective variety.
Moreover, if $v_0$ is primitive, then $M_{\sigma}^s(v_0)=M_{\sigma}(v_0)$ is a smooth integral projective variety.
\end{thm}

We will not prove Theorem \ref{thm:K3ModuliProjective}; we refer to \cite{BM14:projectivity}.
The idea of the proof, based on \cite{MYY14:stability_k_trivial_surfaces}, is to reduce to the case of semistable sheaves by using a Fourier-Mukai transform.
The corresponding statement for non-generic stability conditions in $\Stab^\dagger(X)$ is still unknown.

When the vector is primitive, varieties appearing as moduli spaces of Bridgeland stable objects are so called \emph{Irreducible Holomorphic Symplectic}. They are all deformation equivalent to Hilbert schemes of points on a K3 surface.


\section{Applications and Examples}
\label{sec:applications}

In this section we will give examples and applications for studying stability on surfaces.

\subsection{The largest wall for Hilbert schemes}\label{subsec:LargestWallHilbertScheme}

Let $X$ be a smooth complex projective surface of Picard rank one. We will deal with computing the largest wall for ideal sheaves of zero dimensional schemes $Z \subset X$. The moduli space of these ideal sheaves turns out to be the Hilbert scheme of points on $X$. The motivation for this problem lies in understanding its nef cone. We will explain in the next section how, given a stability condition $\sigma$, one constructs nef divisors on moduli spaces of $\sigma$-semistable objects. It turns out that stability conditions on walls will often times induce boundary divisors of the nef cone. If the number of points is large, this was done in \cite{BHLRSW15:nef_cones}.

\begin{prop}
\label{prop:surface_largest_wall}
Let $X$ be a smooth complex projective surface with $\NS(X) = \Z \cdot H$, where $H$ is ample. Moreover, let $a > 0$ be the smallest integer such that $aH$ is effective and $n > a^2 H^2$. Then the biggest wall for the Chern character $(1,0,-n)$ to the left of the unique vertical wall is given by the equation $\nu_{\alpha, \beta}(\OO(-aH)) = \nu_{\alpha, \beta}(1,0,-n)$. Moreover, the ideal sheaves $\II_Z$ that destabilize at this wall are exactly those for which $Z \subset C$ for a curve $C \in |aH|$.
\end{prop}

We need the following version of a result from \cite{CHW17:effective_cones_p2}.

\begin{lem}
\label{lem:higherRankBound}
Let $0 \to F \to E \to G \to 0$ be an exact sequence in $\Coh^{\beta}(X)$ defining a non empty semicircular wall $W$. Assume further that $\ch_0(F) > \ch_0(E) \geq 0$. Then the radius $\rho_W$ satisfies the inequality
\[
\rho_W^2 \leq \frac{\overline{\Delta}_H(E)}{4 H^2 \cdot \ch_0(F) (H^2 \cdot \ch_0(F) - H^2 \cdot \ch_0(E))}.
\]
\begin{proof}
Let $v,w \in K_0(X)$ be two classes such that the wall $W$ given by $\nu_{\alpha, \beta}(v) = \nu_{\alpha, \beta}(w)$ is a non empty semicircle. Then a straightforward computation shows that the radius $\rho_W$ and center $s_W$ satisfy the equation
\begin{equation}
\label{eq:radius_center}
(H^2 \cdot \ch_0(v))^2 \rho_W^2 + \overline{\Delta}_H(v) = (H^2 \cdot \ch_0(v) s_W - H \cdot \ch_1(v))^2.
\end{equation}
For all $(\alpha, \beta) \in W$ we have the inequalities $H \cdot \ch_1^{\beta}(E) \geq H \cdot \ch_1^{\beta}(F) \geq 0$. This can be rewritten as
\[
H \cdot \ch_1(E) + \beta (H^2 \cdot \ch_0(F) - H^2 \cdot \ch_0(E)) \geq H \cdot \ch_1(F) \geq \beta H^2 \cdot \ch_0(F).
\]
Since $H \cdot \ch_1(F)$ is independent of $\beta$ we can maximize the right hand side and minimize the left hand side individually in the full range of $\beta$ between $s_W - \rho_W$ and $s_W + \rho_W$. By our assumptions this leads to
\[
H \cdot \ch_1(E) + (s_W - \rho_W) (H^2 \cdot \ch_0(F) - H^2 \cdot \ch_0(E)) \geq (s_W + \rho_W) H^2 \cdot \ch_0(F).
\]
By rearranging the terms and squaring we get
\[
(2H^2 \cdot \ch_0(F) - H^2 \cdot \ch_0(E))^2 \rho_W^2 \leq (H \cdot \ch_1(E) - H^2 \cdot \ch_0(E) s_W)^2 = (H^2 \cdot \ch_0(E))^2 \rho_W^2 + \overline{\Delta}_H(E).
\]
The claim follows by simply solving for $\rho_W^2$.
\end{proof}
\end{lem}

\begin{proof}[Proof of Proposition \ref{prop:surface_largest_wall}]
We give the proof in the case where $H$ is effective, i.e., $a = 1$. The general case is longer, but not substantially harder. The full argument can be found in \cite{BHLRSW15:nef_cones}.

The equation $\nu_{\alpha, \beta}(1,0,-n) = \nu_{\alpha, \beta}(\OO(-H))$ is equivalent to
\begin{equation}
\label{eq:largestwall}
\alpha^2 + \left(\beta + \frac{1}{2} + \frac{n}{H^2}\right)^2 = \left( \frac{n}{H^2} -\frac{1}{2} \right)^2.
\end{equation}
This shows that every larger semicircle intersects the line $\beta = -1$. Moreover, the object $\OO(-H)$ is in the category along the wall if and only if $n > \tfrac{H^2}{2}$.

We will first show that there is no bigger semicircular wall. Assume we have an exact sequence
\[
0 \to F \to \II_Z \to G \to 0
\]
where $Z \subset X$ has dimension $0$ and length $n$. Moreover, assume the equation $\nu_{\alpha, -1}(F) = \nu_{\alpha, -1}(G)$ has a solution $\alpha > 0$. We have $\ch^{-1}(E) = (1, H, \tfrac{H^2}{2} - n)$. By definition of $\Coh^{-1}(X)$ we have $H \cdot \ch^{-1}_1(F), H \cdot \ch^{-1}_1(G) \geq 0$ and both those numbers add up to $H \cdot \ch^{-1}_1(E) = H^2$. Since $H$ is the generator of the Picard group, this implies either $H \cdot \ch^{-1}_1(F) = 0$ or $H \cdot \ch^{-1}_1(G) = 0$. In particular, either $F$ or $G$ have slope infinity and it is impossible for $F$, $E$ and $G$ to have the same slope for $\beta = -1$ and $\alpha > 0$.

Next, assume that $0 \to F \to \II_Z \to G \to 0$ induces the wall $W$. By the long exact sequence in cohomology $F$ is a torsion free sheaf. By Lemma \ref{lem:higherRankBound} the inequality $\ch_0(F) \geq 2$ leads to
\[
\rho^2 \leq \frac{2 H^2 n }{8(H^2)^2} = \frac{n}{4H^2} < \left( \frac{n}{H^2} -\frac{1}{2} \right)^2.
\]
Therefore, any such sequence giving the wall $\nu_{\alpha, \beta}(1,0,-n) = \nu_{\alpha, \beta}(\OO(-H))$ must satisfy $\ch_0(F) = 1$. Moreover, we must also have $H \cdot \ch^{-1}_1(F) = H \cdot \ch_1(F) + H^2 \geq 0$ and $H \cdot \ch^{-1}_1(G) = - H \cdot \ch_1(F) \geq 0$. A simple calculation shows that $\ch_1(F) = 0$ does not give the correct wall and therefore, $\ch_1(F) = -H$. Another straightforward computation implies that only $\ch_2(F) = \tfrac{H^2}{2}$ defines the right wall numerically. Since $F$ is a torsion free sheaf, this means $F = \OO(-H)$ implying the claim.
\end{proof}

\begin{exercise}
Any subscheme $Z \subset \P^n$ of dimension $0$ and length $4$ is contained in a quadric. Said differently there is a morphism $\OO(-2) \into \II_Z$, i.e., no ideal sheaf is stable below the wall $\nu_{\alpha, \beta}(\OO(-2)) = \nu_{\alpha, \beta}(\II_Z)$. Note that $\ch(\II_Z) = (1,0,-4)$. The goal of this exercise, is to compute all bigger walls for this Chern character.
\begin{enumerate}
\item Compute the equation of the wall $\nu_{\alpha, \beta}(\OO(-2)) = \nu_{\alpha, \beta}(\II_Z)$. Why do all bigger walls intersect the line $\beta = -2$?
\item Show that there are two walls bigger than $\nu_{\alpha, \beta}(\OO(-2)) = \nu_{\alpha, \beta}(\II_Z)$. \emph{Hint: Let $0 \to F \to \II_Z \to G \to 0$ define a wall for some $\alpha > 0$ and $\beta = -2$. Then $F$ and $G$ are semistable, i.e. they satisfy the Bogomolov inequality. Additionally show that $0 < \ch^{\beta}_1(F) < \ch^{\beta}_1(E)$.}
\item Determine the Jordan-H\"older filtration of any ideal sheaf $\II_Z$ that destabilizes at any of these three walls. What do these filtrations imply about the geometry of $Z$? \emph{Hint: Use the fact that a semistable sheaf on $\P^2$ with Chern character $n \cdot \ch(\OO(-m))$ has to be $\OO(-m)^{\oplus n}$.}
\end{enumerate}
\end{exercise}

\subsection{Kodaira vanishing}
As another application, we will give a proof of Kodaira vanishing for surfaces using tilt stability. The argument was first pointed out in \cite{AB13:k_trivial}.
While it is a well-known argument by Mumford that Kodaira vanishing in the surface case is a consequence of Bogomolov's inequality, this proof follows a slightly different approach.

\begin{thm}[Kodaira Vanishing for Surfaces]
Let $X$ be a smooth projective complex surface and $H$ an ample divisor on $X$. If $K_X$ is the canonical divisor class of $X$, then the vanishing
\[
H^i(\OO(H + K_X)) = 0
\]
holds for all $i > 0$.
\begin{proof}
By Serre duality $H^2(\OO(H + K_X)) = H^0(\OO(-H))$. Since anti-ample divisors are never effective we get $H^0(\OO(-H)) = 0$.

The same way Serre duality implies $H^1(\OO(H + K_X)) = H^1(\OO(-H)) = \Hom(\OO, \OO(-H)[1])$. By Lemma \ref{lem:parallel_line_bundle} both $\OO$ and $\OO(-H)[1]$ are tilt semistable for all $\omega = \alpha H$, $B = \beta H$, where $\alpha > 0$ and $\beta \in (-1,0)$. A straightforward computation shows
\[
\nu_{\omega, B}(\OO) > \nu_{\omega, B}(\OO(-H)[1]) \Leftrightarrow \alpha^2 + \left(\beta - \frac{1}{2}\right)^2 > \frac{1}{4}.
\]
Therefore, there is a region in which $\nu_{\omega, B}(\OO) > \nu_{\omega, B}(\OO(-H)[1])$ and both these objects are tilt stable, i.e., $\Hom(\OO, \OO(-H)[1]) = 0$.
\end{proof}
\end{thm}

\subsection{Stability of special objects}

As in Section \ref{subsec:ExamplesCurves}, we can look at Bridgeland stability for Lazarsfeld-Mukai bundles on a surface.
The setting is the following.
Let $X$ be a smooth projective surface.
Let $L$ be an ample integral divisor on $X$.
We assume the following: $L^2>8$ and $L.C>3$, for all integral curves $C\subset X$.
By Reider's Theorem \cite{Rei88:vector_bundle_linear_systems}, the divisor $L+K_X$ is very ample.
We define the \emph{Lazarsfeld-Mukai} vector bundle as the kernel of the evaluation map:
\[
M_{L+K_X} := \Ker \left(\OO_X\otimes H^0(X,\OO_X(L+K_X)) \onto \OO_X(L+K_X) \right).
\]
We consider the ample divisor $H:=2L+K_X$ and $B:=\frac{K_X}{2}$, and consider the $(\alpha,\beta)$-plane with respect to $H$ and $B$.

The following question would have interesting geometric applications:

\begin{question}\label{question:ProjectiveNormality}
For which $(\alpha,\beta)$ is $M_{L+K_X}$ tilt stable?
\end{question}

\begin{exercise}\label{exercise:StabilityLazarsfeldMukai}
Assume that $M_{L+K_X}$ is $\nu_{\alpha,\beta}$-semistable for $(\alpha,\beta)$ inside the wall defined by $\nu_{\alpha,\beta}(M_{L+K_X})=\nu_{\alpha,\beta}(\OO_X(-L)[1])$.
Show that $H^1(X,M_{L+K_X}\otimes\OO_X(K_X+L))=0$.
\end{exercise}

If the assumption in Exercise \ref{exercise:StabilityLazarsfeldMukai} is satisfied, then we would prove that the multiplication map $H^0(X,\OO_X(L+K_X)) \otimes H^0(X,\OO_X(L+K_X)) \to H^0(X,\OO_X(2L+2K_X))$ is surjective, the first step toward projective normality for the embedding of $X$ given by $L+K_X$.

\begin{ex}
\label{ex:LazarsfeldMukaiK3}
The assumption in Exercise \ref{exercise:StabilityLazarsfeldMukai} is satisfied when $X$ is a K3 surface.
Indeed, this is an explicit computation by using the fact that in that case, the vector bundle $M_L$ is a \emph{spherical} object (namely, $\End(M_L)=\C$ and $\Ext^1(M_L,M_L)=0$).
\end{ex}


\section{Nef Divisors on Moduli Spaces of Bridgeland Stable Objects}
\label{sec:nef}

So far we neglected another important aspect of moduli spaces of Bridgeland stable objects, namely the construction of divisors on them. Let $X$ be an arbitrary smooth projective variety. Given a stability condition $\sigma$ and a fixed set of invariants $v \in \Lambda$ we will demonstrate how to construct a nef divisor on moduli spaces of $\sigma$-semistable objects with class $v$. This was originally described in \cite{BM14:projectivity}.

Assume that $\sigma = (Z,\AA)$ is a stability condition on $\Db(X)$, $S$ is a proper algebraic space of finite type over $\C$ and $v \in \Lambda$ a fixed set of invariants. Moreover, let $\EE \in \Db(X \times S)$ be a flat family of $\sigma$-semistable objects of class $v$, i.e., $\EE$ is $S$-perfect (see Subsection \ref{subsec:moduli} for a definition) and for every $\C$-point $P \in S$, the derived restriction $\EE_{|X \times \{P\}}$ is $\sigma$-semistable of class $v$. The purpose of making the complex $S$-perfect is to make the derived restriction well defined. A divisor class $D_{\sigma, \EE}$ can defined on $S$ by its intersection number with any projective integral curve $C \subset S$:
\[
D_{\sigma, \EE} \cdot C = \Im \left( -\frac{Z((p_X)_* \EE_{|X \times C})}{Z(v)} \right).
\]
We skipped over a detail in the definition that is handled in Section 4 of \cite{BM14:projectivity}. It is necessary to show that the number only depends on the numerical class of the curve $C$.

The motivation for this definition is the following Theorem due to \cite{BM14:projectivity}.

\begin{thm}[Positivity Lemma]
\label{thm:positivity_lemma}
\begin{enumerate}
\item The divisor $D_{\sigma, \EE}$ is nef on $S$.
\item A projective integral curve $C \subset X$ satisfies $D_{\sigma, \EE} \cdot C = 0$ if and only if for two general elements $c,c' \in C$ the restrictions $\EE_{|X \times \{c\}}$ and $\EE_{|X \times \{c'\}}$ are $S$-equivalent, i.e., their Jordan-H\"older filtrations have the same stable factors up to order.
\end{enumerate}
\end{thm}

We will give a proof that the divisor is nef and refer to \cite{BM14:projectivity} for the whole statement. Further background material is necessary. We will mostly follow the proof in \cite{BM14:projectivity}. Without loss of generality we can assume that $Z(v) = -1$ by scaling and rotating the stability condition. Indeed, assume that $Z(v) = r_0 e^{\sqrt{-1} \phi_0 \pi}$ and let $\PP$ be the slicing of our stability condition. We can then define a new stability condition by setting $\PP'(\phi) := \PP(\phi + 1 - \phi_0)$ and $Z' = Z \cdot \tfrac{1}{r_0} e^{\sqrt{-1} (1-\phi_0) \pi}$. Clearly, $Z'(v) = -1$. The definition of $D_{\sigma, \EE}$ simplifies to
\[
D_{\sigma, \EE} \cdot C = \Im \left( Z((p_X)_* \EE_{|X \times C}) \right)
\]
for any projective integral curve $C \subset S$. From this formula the motivation for the definition becomes more clear. If $(p_X)_* \EE_{|X \times C} \in \AA$ holds, the fact that the divisor is nef follows directly from the definition of a stability function. As always, things are more complicated.
We can use Bridgeland's Deformation Theorem and assume further, without changing semistable objects (and the previous assumption), that the heart $\AA$ is noetherian.

One of the key properties in the proof is an extension of a result by Abramovich and Polishchuk from \cite{AP06:constant_t_structures} to the following statement by Polishchuk. We will not give a proof.

\begin{thm}{{\cite[Theorem 3.3.6]{Pol07:constant_t_structures}}}
\label{thm:constant_t_structure}
Let $\AA$ be the heart of a noetherian bounded t-structure on $\Db(X)$. The category $\AA^{qc} \subset D_{qc}(X)$ is the closure of $\AA$ under infinite coproducts in the (unbounded) derived category of quasi-coherent sheaves on $X$. There is a noetherian heart $\AA_S$ on $\Db(X \times S)$ for any finite type scheme $S$ over the complex numbers satisfying the following three properties.
\begin{enumerate}
\item The heart $\AA_S$ is characterized by the property
\[
\EE \in \AA_S \Leftrightarrow (p_X)_* \EE_{|X \times U} \in \AA^{qc} \text{ for every open affine } U \subset S.
\]
\item If $S = \bigcup_i U_i$ is an open covering of $S$, then
\[
\EE \in \AA_S \Leftrightarrow \EE_{|X \times U_i} \in \AA_{U_i} \text{ for all } i.
\]
\item If $S$ is projective and $\OO_S(1)$ is ample, then
\[
\EE \in \AA_S \Leftrightarrow (p_X)_* (\EE \otimes p_S^* \OO_S(n)) \in \AA \text{ for all } n \gg 0.
\]
\end{enumerate}
\end{thm}

In order to apply this theorem to our problem we need the family $\EE$ to be in $\AA_S$. The proof of the following statement is somewhat technical and we refer to \cite[Lemma 3.5]{BM14:projectivity}.

\begin{lem}
Let $\EE \in \Db(X \times S)$ be a flat family of $\sigma$-semistable objects of class $v$. Then $\EE \in \AA_S$.
\end{lem}

\begin{proof}[Proof of Theorem \ref{thm:positivity_lemma} (1)]
Let $\OO_C(1)$ be an ample line bundle on $C$. If we can show 
\[
D_{\sigma, \EE} \cdot C = \Im \left( Z((p_X)_* (\EE \otimes p_S^* \OO_C(n))) \right)
\]
for $n \gg 0$, then we are done. Indeed, Theorem \ref{thm:constant_t_structure} part (3) implies $(p_X)_* (\EE \otimes p_S^* \OO_C(n)) \in \AA$ for $n \gg 0$ and the proof is concluded by the positivity properties of a stability function. 

Choose $n \gg 0$ large enough such that $H^0(\OO_C(n)) \neq 0$. Then there is a torsion sheaf $T$ together with a short exact sequence
\[
0 \to \OO_C \to \OO_C(n) \to T \to 0.
\]
Since $T$ has zero dimensional support and $\EE$ is a family of objects with class $v$, we can show $Z((p_X)_* (\EE \otimes p_S^* T)) \in \R$ by induction on the length of the support of $T$. But that shows
\[
\Im \left( Z((p_X)_* (\EE \otimes p_S^* \OO_C(n))) \right) = \Im \left( Z((p_X)_* (\EE \otimes p_S^* \OO_C))) \right). \qedhere
\]
\end{proof}

\subsection{The Donaldson morphism}

The definition of the divisor is made such that the proof of the Positivity Lemma is as clear as possible. However, it is hard to explicitly compute it in examples directly via the definition. Computation is often times done via the Donaldson morphism. This is originally explained in Section 4 of \cite{BM14:projectivity}.
Recall that, for a proper scheme $S$, the Euler characteristic gives a well-defined pairing
\[
\chi\colon K_0(\Db(S)) \times K_0(\mathrm{D}_{\mathrm{perf}}(S)) \to \Z
\]
between the Grothendieck groups of the bounded derived categories of coherent sheaves $\Db(S)$ and of perfect complexes $\mathrm{D}_{\mathrm{perf}}(S)$.
Taking the quotient with respect to the kernel of $\chi$ on each
side we obtain numerical Grothendieck groups $K_{\num}(S)$ and $K_{\num}^{\perf}(S)$, respectively, with an induced perfect pairing
\[
\chi: K_{\num}(S) \otimes K_{\num}^{\perf}(S) \to \Z.
\]

\begin{defn}
We define the additive \emph{Donaldson morphism} $\lambda_{\EE}: v^{\#} \to N^1(S)$ by
\[
w \mapsto \det((p_S)_*(p_X^* w \cdot [\EE])).
\]
Here,
\[
v^{\#} = \{ w \in K_{\num}(S)_{\R} : \chi(v \cdot w) = 0 \}.
\]
\end{defn}

Let $w_{\sigma} \in v^{\#}$ be the unique vector such that
\[
\chi(w_{\sigma} \cdot w') = \Im \left(-\frac{Z(w')}{Z(v)} \right)
\]
for all $w' \in K_{\num}(X)_{\R}$.

\begin{prop}[{\cite[Theorem 4.4]{BM14:projectivity}}]
\label{prop:donaldson_computation}
We have $\lambda_{\EE}(w_{\sigma}) = D_{\sigma, \EE}$.
\begin{proof}
Let $\LL_{\sigma} = (p_S)_*(p_X^* w_{\sigma} \cdot [\EE])$. For any $s \in S$ we can compute the rank
\[
r(\LL_{\sigma}) = \chi(S, [\OO_s] \cdot \LL_{\sigma}) = \chi(S \times X, [\OO_{\{s\}\times X}] \cdot [\EE] \cdot p_X^* w_{\sigma}) = \chi(w_{\sigma} \cdot v) = 0.
\]
Therefore, $\LL_{\sigma}$ has rank zero. This implies
\[
\lambda_{\EE}(w_{\sigma}) \cdot C = \chi (\LL_{\sigma | C})
\]
for any projective integral curve $C \subset S$. Let $i_C: C \into S$ be the embedding of $C$ into $S$. Cohomology and base change implies
\begin{align*}
\LL_{\sigma | C} &= i_C^* (p_S)_*(p_X^* w_{\sigma} \cdot [\EE]) \\
&= (p_C)_* (i_C \times \id_X)^* (p_X^* w_{\sigma} \cdot [\EE]) \\
&= (p_C)_* (p_X^* w_{\sigma} \cdot [\EE_{|C \times X}]).
\end{align*}
The proof can be finished by using the projection formula as follows
\begin{align*}
\chi (C, \LL_{\sigma | C}) & = \chi (C, (p_C)_* (p_X^* w_{\sigma} \cdot [\EE_{|C \times X}])) \\
&= \chi(X, w_{\sigma} \cdot (p_X)_* [\EE_{|C \times X}]) \\
&=  \Im \left( -\frac{Z((p_X)_* \EE_{|X \times C})}{Z(v)} \right). \qedhere
\end{align*}
\end{proof}
\end{prop}

\subsection{Applications to Hilbert schemes of points}

Recall that for any positive integer $n \in \N$ the Hilbert scheme of $n$-points $X^{[n]}$ parameterizes subschemes $Z \subset X$ of dimension zero and length $n$. This scheme is closely connected to the symmetric product $X^{(n)}$ defined as the quotient $X^n/S_n$ where $S_n$ acts on $X^n$ via permutation of the factors. By work of Fogarty in \cite{Fog68:hilbert_schemeI} the natural map $X^{[n]} \to X^{(n)}$ is a birational morphism that resolves the singularities of $X^{(n)}$. 

We will recall the description of $\Pic(X^{[n]})$ in case the surface $X$ has irregularity zero, i.e., $H^1(\OO_X) = 0$. It is a further result in \cite{Fog73:hilbert_schemeII}. If $D$ is any divisor on $X$, then there is an $S_n$ invariant divisor $D^{\boxtimes n}$ on $X^n$. This induces a divisor $D^{(n)}$ on $X^{(n)}$, which we pull back to a divisor $D^{[n]}$ on $X^{[n]}$. If $D$ is a prime divisor, then $D^{[n]}$ parameterizes those $Z \subset X$ that intersect $D$ non trivially. Then
\[
\Pic(X^{[n]}) \cong \Pic(X) \oplus \Z \cdot \frac{E}{2},
\]
where $E$ parameterizes the locus of non reduces subschemes $Z$. Moreover, the restriction of this isomorphism to $\Pic(X)$ is the embedding given by $D \mapsto D^{[n]}$. A direct consequence of this result is the description of the N\'eron-Severi group as
\[
\NS(X^{[n]}) \cong \NS(X) \oplus \Z \cdot  \frac{E}{2}.
\]
The divisor $\frac{E}{2}$ is integral because it is given by $\det((p_{X^{[n]}})_* \UU_n)$, where $\UU_n \in \Coh(X \times X^{[n]})$ is the \emph{universal ideal sheaf} of $X^{[n]}$.

\begin{thm}
\label{thm:divisor_hilbert_scheme}
Let $X$ be a smooth complex projective surface with $\Pic(X) = \Z \cdot H$, where $H$ is ample. Moreover, let $a > 0$ be the smallest integer such that $aH$ is effective. If $n \geq a^2 H^2$, then the divisor
\[
D = \frac{1}{2} K_X^{[n]} + \left(\frac{a}{2} + \frac{n}{aH^2}\right) H^{[n]} - \frac{1}{2} E
\]
is nef. If $g$ is the arithmetic genus of a curve $C \in |aH|$ and $n \geq g+1$, then $D$ is extremal. In particular, the nef cone of $X^{[n]}$ is spanned by this divisor and $H^{[n]}$.
\end{thm}

For the proof of this statement, we need to describe the image of the Donaldson morphism more precisely. In this case, the vector $v$ is given by $(1,0,-n)$ and $\EE = \UU_n \in \Coh(X \times X^{[n]})$ is the universal ideal sheaf for $X^{[n]}$.

\begin{prop}
\label{prop:donaldson_morphism_hilb}
Choose $m$ such that $(1, 0, m) \in v^{\#}$ and for any divisor $D$ on $X$ choose $m_D$ such that $(0, D, m_D) \in v^{\#}$. Then
\begin{align*}
\lambda_{\UU_n}(1, 0, m) &= \frac{E}{2}, \\
\lambda_{\UU_n}(0, D, m_D) &= -D^{[n]}.
\end{align*}
\begin{proof}[Sketch of the proof]
Let $x \in X$ be an arbitrary point and $\C(x)$ the corresponding skyscraper sheaf. The Grothendieck-Riemann-Roch Theorem implies
\[
\ch((p_{X^{[n]}})_*(p_X^*\C(x) \otimes \UU_n)) = (p_{X^{[n]}})_* (\ch(p_X^*\C(x) \otimes \UU_n) \cdot \td(T_{p_{X^{[n]}}})).
\]
As a consequence
\[
\lambda_{\UU_n}(0,0,1) = p_X^*\left(- [x] \cdot \frac{K_X}{2} \right) = 0
\]
holds. Therefore, the values of $m$ and $m_D$ are irrelevant for the remaining computation. Similarly, we can show that 
\[
-\lambda_{\UU_n}\left(0,D,-\frac{D^2}{2}\right) = (p_{X^{[n]}})_* (p_X^*(-D) \cdot \ch_2(\UU_n)).
\]
Intuitively, this divisor consists of those schemes $Z$ parameterized in $X^{[n]}$ that intersect the divisor $D$, i.e., it is $D^{[n]}$ (see \cite{CG90:d_very_ample}).

Finally, we have to compute $\lambda_{\UU_n}(1,0,0)$. By definition it is given as $\det((p_{X^{[n]}})_* \UU_n) = \tfrac{E}{2}$.
\end{proof}
\end{prop}

\begin{cor}
\label{cor:explicit_divisor}
Let $W$ be a numerical semicircular wall for $v = (1, 0, -n)$ with center $s_W$. Assume that all ideal sheaves $I_Z$ on $X$ with $\ch(I_Z) = (1, 0, -n)$ are semistable along $W$. Then
\[
D_{\alpha, \beta, \UU_n} \in \R_{> 0} \left( \frac{K_X^{[n]}}{2} - s_W H^{[n]} - \frac{E}{2} \right)
\]
for all $\sigma \in W$.
\begin{proof}
We fix $\sigma = (\alpha, \beta) \in W$. By Proposition \ref{prop:donaldson_computation} there is a class $w_{\sigma} \in v^{\#}$ such that the divisor $D_{\sigma, \UU_n}$ is given by $\lambda_{\UU_n}(w_{\sigma})$. This class is characterized by the property
\[
\chi(w_{\sigma} \cdot w') = \Im \left(-\frac{Z_{\alpha, \beta}(w')}{Z_{\alpha, \beta}(v)} \right)
\]
for all $w' \in K_{\num}(X)_{\R}$. We define $x,y \in \R$ by
\[
-\frac{1}{Z_{\alpha, \beta}(v)} = x + \sqrt{-1} y.
\]
The fact that $Z_{\alpha, \beta}$ is a stability function implies $y \geq 0$. We even have $y>0$ because $y=0$ holds if and only if $(\alpha, \beta)$ is on the unique numerical vertical wall contrary to assumption. The strategy of the proof is to determine $w_{\sigma}$ by pairing it with some easy to calculate classes and then using the previous proposition.

Let $w_{\sigma} = (r,C,d)$, where $r \in \Z$, $d \in \tfrac{1}{2} \Z$ and $C$ is a curve class. We have
\[
r = \chi(w_{\sigma} \cdot (0,0,1)) = \Im ((x+\sqrt{-1}y)Z_{\alpha, \beta}(0,0,1)) = -y.
\]
A similar computation together with Riemann-Roch shows
\begin{align*}
(x + \beta y)H^2 &= \Im ((x+iy) \cdot Z_{\alpha, \beta} (0, H, 0)) \\
&= \chi(w_{\sigma} \cdot (0, H, 0))\\
&= \chi(0, -yH, H \cdot C) \\
&= \int_X (0, -yH, H \cdot C) \cdot \left(1, -\frac{K_X}{2}, \chi(\OO_X) \right) \\
&= \frac{y}{2} H \cdot K_X + H \cdot C.
\end{align*}
Since $\Pic(X) = \Z \cdot H$, we obtain
\[
C = (x + \beta y) H - \frac{y}{2} K_X = ys_W H - \frac{y}{2} K_X.
\]
The last step used $x + \beta y = ys_W$ which is a straightforward computation. We get
\[
D_{\alpha, \beta, \UU_n} \in \R_{> 0} \left( \lambda_{\UU_n}(-1, s_W H - \frac{K_X}{2}, m) \right),
\]
where $m$ is uniquely determined by $w_{\sigma} \in v^{\#}$. The statement follows now from a direct application of Proposition \ref{prop:donaldson_morphism_hilb}.
\end{proof}
\end{cor}

\begin{proof}[Proof of Theorem \ref{thm:divisor_hilbert_scheme}]
By Proposition \ref{prop:surface_largest_wall} and the assumption $n \geq a^2H^2$ we know that the largest wall destabilizes those ideal sheaves $\II_Z$ that fit into an exact sequence
\[
0 \to \OO(-aH) \to \II_Z \to \II_{Z/C} \to 0,
\]
where $C \in |aH|$. This wall has center
\[
s = -\frac{a}{2} - \frac{n}{aH^2}. 
\]
By Corollary \ref{cor:explicit_divisor} we get that
\[
D = \frac{1}{2} K_X^{[n]} + \left(\frac{a}{2} + \frac{n}{aH^2}\right) H^{[n]} - \frac{1}{2} E
\]
is nef. We are left to show extremality of the divisor in case $n \geq g+1$. By part (2) of the Positivity Lemma, we have to construct a one dimensional family of $S$-equivalent objects. It exists if and only if $\ext^1(\II_{Z/C}, \OO(-aH)) \geq 2$. A Riemann-Roch calculation shows
\begin{align*}
1 - g = \chi(\OO_C) &= \int_X \left(0, aH, -\frac{a^2H^2}{2} \right) \cdot \left(1, -\frac{K_X}{2}, \chi(\OO_X) \right)  \\
&= - \frac{a}{2} H \cdot K_X -\frac{a^2H^2}{2}.
\end{align*}
Another application of Riemann-Roch for surfaces shows
\begin{align*}
\ext^1(\II_{Z/C}, \OO(-aH)) &\geq -\chi(\II_{Z/C}, \OO(-aH)) \\
&= -\int_X \left(1, -aH, \frac{a^2H^2}{2} \right) \left(0, -aH, -\frac{a^2H^2}{2} - n \right) \left(1, -\frac{K_X}{2}, \chi(\OO_X) \right) \\
&= n - \frac{a}{2} H \cdot K_X - \frac{a^2H^2}{2} \\
&= n + 1 - g.
\end{align*}
Therefore, $n \geq g+1$ implies that $D$ is extremal.
\end{proof}

\begin{rmk}
In particular cases, Theorem \ref{thm:divisor_hilbert_scheme} can be be made more precise and general.
In the case of the projective plane \cite{CHW17:effective_cones_p2, CH18:nef_cones, LZ19:NewStabilityP2} and of K3 surfaces \cite{BM14:projectivity, BM14:stability_k3}, given any primitive vector, varying stability conditions corresponds to a directed Minimal Model Program for the corresponding moduli space. This allows to completely describe the nef cone, the movable cone, and the pseudo-effective cone for them. Also, all corresponding birational models appear as moduli spaces of Bridgeland stable objects. This has also deep geometrical applications; for example, see \cite{Bay18:BrillNoether}.
\end{rmk}


\section{Stability Conditions on Threefolds}
\label{sec:P3}


In this section we will give a short sketchy introduction to the higher-dimensional case.
The main result is the construction of stability condition on $X = \P^3$.
By abuse of notation, we will identify $\ch^{\beta}_i(E)\in\Q$, for any $E \in \Db(\P^3)$.

\subsection{Tilt stability and the second tilt}

The construction of $\sigma_{\alpha, \beta} = (\Coh^{\beta}(X), Z_{\alpha, \beta})$ for $\alpha > 0$ and $\beta \in \R$ can carried out as before. However, this will not be a stability condition, because $Z_{\alpha, \beta}$ maps skyscraper sheaves to the origin. In \cite{BMT14:stability_threefolds}, this (weak) stability condition is called \emph{tilt stability}. The idea is that, by repeating the previous process of tilting another time with $\sigma_{\alpha, \beta}$ instead of slope stability, this might allow to construct a Bridgeland stability condition on $\Db(X)$. Let 
\begin{align*}
\TT'_{\alpha, \beta} &= \{E \in \Coh^{\beta}(\P^3) : \text{any quotient $E
\onto G$ satisfies $\nu_{\alpha, \beta}(G) > 0$} \}, \\
\FF'_{\alpha, \beta} &=  \{E \in \Coh^{\beta}(\P^3) : \text{any subobject $F
\into E$ satisfies $\nu_{\alpha, \beta}(F) \leq 0$} \}
\end{align*}
and set $\AA^{\alpha, \beta}(\P^3) = \langle \FF'_{\alpha, \beta}[1],
\TT'_{\alpha, \beta} \rangle $. For any $s>0$, we define
\[
Z_{\alpha,\beta,s} = -\ch^{\beta}_3 + (s+\tfrac{1}{6})\alpha^2 \ch^{\beta}_1 + \sqrt{-1} (\ch^{\beta}_2 - \frac{\alpha^2}{2} \ch^{\beta}_0).
\]
and the corresponding slope
\[
\lambda_{\alpha,\beta,s} = \frac{\ch^{\beta}_3 - (s+\tfrac{1}{6})\alpha^2 \ch^{\beta}_1}{\ch^{\beta}_2 - \frac{\alpha^2}{2} \cdot \ch^{\beta}_0}.
\]
Recall that the key point in the construction stability conditions on surfaces was the classical Bogomolov inequality for slope stable sheaves. That is the idea behind the next theorem.

\begin{thm}[\cite{BMT14:stability_threefolds, Mac14:conjecture_p3, BMS16:abelian_threefolds}]
\label{thm:construction_threefold}
The pair $(\AA^{\alpha, \beta}(\P^3), \lambda_{\alpha,\beta,s})$ is Bridgeland stability condition for all $s > 0$ if and only if for all tilt stable objects $E \in \Coh^{\beta}(\P^3)$ we have
\[Q_{\alpha, \beta}(E) = \alpha^2 \Delta(E) + 4(\ch_2^{\beta}(E))^2 - 6\ch_1^{\beta}(E) \ch_3^{\beta}(E) \geq 0.\]
\end{thm}

It was more generally conjectured that similar inequalities are true for all smooth projective threefolds. This turned out to be true in various cases. The first proof was in the case of $\P^3$ in \cite{Mac14:conjecture_p3} and a very similar proof worked for the smooth quadric hypersurface in $\P^4$ in \cite{Sch14:conjecture_quadric}. These results were generalized with a fundamentally different proof to all Fano threefolds of Picard rank one in \cite{Li19:conjecture_fano_threefold}. The conjecture is also known to be true for all abelian threefolds with two independent proofs by \cite{MP16:conjecture_abelian_threefoldsII} and \cite{BMS16:abelian_threefolds}. Unfortunately, it turned out to be problematic in the case of the blow up of $\P^3$ in a point as shown in \cite{Sch17:counterexample}.

\begin{rmk}\label{rmk:ExtensionToTiltStability3folds}
Many statements about Bridgeland stability on surfaces work almost verbatim in tilt stability. Analogously to Theorem \ref{prop:locally_finite} walls in tilt stability are locally finite and stability is an open property. The structure of walls from Proposition \ref{prop:StructureThmWallsSurfaces} is the same in tilt stability. The Bogomolov inequality also works, i.e., any tilt semistable object $E$ satisfies $\Delta(E) \geq 0$. The bound for high degree walls in Lemma \ref{lem:higherRankBound} is fine as well.
Finally, slope stable sheaves $E$ are $\nu_{\alpha, \beta}$-stable for all $\alpha \gg 0$ and $\beta < \mu(E)$.
\end{rmk}

We will give the proof of Chunyi Li in the special case of $\P^3$ in these notes.
We first recall the Hirzebruch-Riemann-Roch Theorem for $\P^3$.

\begin{thm}
Let $E \in \Db(\P^3)$. Then
\[
\chi(\P^3, E) = \ch_3(E) + 2 \ch_2(E) + \frac{11}{6} \ch_1(E) + \ch_0(E).
\]
\end{thm}

In order to prove the inequality in Theorem \ref{thm:construction_threefold}, we want to reduce the problem to a simpler case.

\begin{defn}
For any object $E \in \Coh^{\beta}(\P^3)$, we define
\[
\overline{\beta}(E) = \begin{cases}
\frac{\ch_1(E) - \sqrt{\Delta_H(E)}}{\ch_0(E)} & \ch_0(E) \neq 0, \\
\frac{\ch_2(E)}{\ch_1(E)} & \ch_0(E) = 0.
\end{cases}
\]
The object $E$ is called \emph{$\overline{\beta}$-(semi)stable}, if $E$ (semi)stable in a neighborhood of $(0, \overline{\beta}(E))$.
\end{defn}

A straightforward computation shows that $\ch_2^{\overline{\beta}}(E) = 0$.

\begin{lem}[\cite{BMS16:abelian_threefolds}]
Proving the inequality in Theorem \ref{thm:construction_threefold} can be reduced to $\overline{\beta}$-stable objects. Moreover, in that case the inequality reduces to $\ch_3^{\overline{\beta}}(E) \leq 0$.
\begin{proof}
Let $E \in \Coh^{\beta_0}(\P^3)$ be a $\nu_{\alpha_0, \beta_0}$-stable object with $\ch(E) = v$. If $(\alpha_0, \beta_0)$ is on the unique numerical vertical wall for $v$, then $\ch_1^{\beta_0}(E) = 0$ and therefore,
\[
Q_{\alpha_0, \beta_0} = \alpha_0^2 \Delta(E) + 4(\ch_2^{\beta_0}(E))^2 \geq 0.
\]
Therefore, $(\alpha_0, \beta_0)$ lies on a unique numerical semicircular wall $W$ with respect to $v$. One computes that there are $x,y \in \R$ such that
\begin{align*}
Q_{\alpha, \beta}(E) \geq 0 &\Leftrightarrow \Delta(E) \alpha^2 + \Delta(E) \beta^2 + x \beta + y \alpha \geq 0, \\
Q_{\alpha, \beta}(E) = 0 &\Leftrightarrow \nu_{\alpha, \beta}(E) = \nu_{\alpha, \beta}\left(\ch_1(E), 2\ch_2(E), 3\ch_3(E), 0 \right).
\end{align*}
In particular, the equation $Q_{\alpha, \beta}(E) \geq 0$ defines the complement of a semi-disc with center on the $\beta$-axis or a quadrant to one side of a vertical line in case $\Delta(E) = 0$. Moreover, $Q_{\alpha, \beta}(E) = 0$ is a numerical wall with respect to $v$.

We will proceed by an induction on $\Delta(E)$ to show that it is enough to prove the inequality for $\overline{\beta}$-stable objects. Assume $\Delta(H) = 0$. Then $Q_{\alpha_0, \beta_0}(E) \geq 0$ is equivalent to $Q_{0, \overline{\beta}}(E) \geq 0$. If $E$ would not be $\overline{\beta}$-stable, then it must destabilize along a wall between $W$ and $(0, \overline{\beta}(E))$. By part (3) of Lemma \ref{lem:convex_cone} with $Q = \Delta$ all stable factors of $E$ along that wall have $\Delta = 0$. By part (4) of the same lemma this can only happen if at least one of the stable factors satisfies $\ch_{\leq 2}(E) = (0,0,0)$ which only happens at the numerical vertical wall.

Assume $\Delta(H) > 0$. If $E$ is $\overline{\beta}$-stable, then it is enough to show $Q_{0, \overline{\beta}(E)} \geq 0$. Assume $E$ is destabilized along a wall between $W$ and $(0, \overline{\beta}(E))$ and let $F_1, \ldots, F_n$ be the stable factors of $E$ along this wall. By Lemma \ref{lem:convex_cone} (3) we have $\Delta(F_i) < \Delta(E)$ for all $i = 1, \ldots n$. We can then use Lemma \ref{lem:convex_cone} (3) with $Q = Q_{\alpha, \beta}$ to finish the proof by induction.
\end{proof}
\end{lem}

The idea for the following proof is due to \cite{Li19:conjecture_fano_threefold}.

\begin{thm}
\label{thm:conj_p3}
For all tilt stable objects $E \in \Coh^{\beta}(\P^3)$ we have
\[Q_{\alpha, \beta}(E) = \alpha^2 \Delta(E) + 4(\ch_2^{\beta}(E))^2 - 6 \ch_1^{\beta}(E) \ch_3^{\beta}(E) \geq 0.\]
\begin{proof}
As observed in Remark \ref{rmk:ExtensionToTiltStability3folds}, tilt semistable objects satisfy $\Delta\geq0$.
Hence, line bundles are stable everywhere in tilt stability in the case $X = \P^3$.

Let $E \in \Coh^{\overline{\beta}}(X)$ be a $\overline{\beta}$-stable object. By tensoring with line bundles, we can assume $\overline{\beta} \in [-1,0)$.
By assumption we have $\nu_{0, \overline{\beta}}(E) = 0 < \nu_{0, \overline{\beta}}(\OO_{\P^3})$ which implies $\Hom(\OO_{\P^3}, E) = 0$. Moreover, the same argument together with Serre duality shows
\[
\Ext^2(\OO_{\P^3}, E) = \Ext^1(E, \OO_{\P^3}(-4)) = \Hom(E, \OO_{\P^3}(-4)[1]) = 0.
\]
Therefore,
\begin{align*}
0 \geq \chi(\OO, E) &= \ch_3(E) + 2 \ch_2(E) + \frac{11}{6} \ch_1(E) + \ch_0(E) \\
&= \ch_3^{\overline{\beta}}(E) + (\overline{\beta} + 2) \ch_2^{\overline{\beta}}(E) + \frac{1}{6}(3\overline{\beta}^2 + 12 \overline{\beta} + 11) \ch_1^{\overline{\beta}}(E) \\
& \ \ \ + \frac{1}{6}(\overline{\beta}^3 + 6\overline{\beta}^2 + 11\overline{\beta} + 6) \ch_0^{\overline{\beta}}(E) \\
& = \ch_3^{\overline{\beta}}(E) + \frac{1}{6}(3\overline{\beta}^2 + 12 \overline{\beta} + 11) \ch_1^{\overline{\beta}}(E) + \frac{1}{6}(\overline{\beta}^3 + 6\overline{\beta}^2 + 11\overline{\beta} + 6) \ch_0^{\overline{\beta}}(E).
\end{align*}
By construction of $\Coh^{\overline{\beta}}(X)$, we have $\ch_1^{\overline{\beta}}(E) \geq 0$. Due to $\overline{\beta} \in [-1,0)$ we also have $3\overline{\beta}^2 + 12 \overline{\beta} + 11 \geq 0$ and $\overline{\beta}^3 + 6\overline{\beta}^2 + 11\overline{\beta} + 6 \geq 0$. If $\ch_0^{\overline{\beta}}(E) \geq 0$ or $\overline{\beta} = -1$, this finishes the proof. If $\ch_0^{\overline{\beta}}(E) < 0$ and $\overline{\beta} \neq -1$, the same type of argument works with $\chi(\OO(3), E) \leq 0$.
\end{proof}
\end{thm}

\subsection{Castelnuovo's genus bound}

The next exercise outlines an application of this inequality to a proof of Castelnuovo's classical theorem on the genus of non-degenerate curves in $\P^3$.

\begin{thm}[\cite{Cas37:inequality}]
Let $C \subset \P^3$ be an integral curve of degree $d \geq 3$ and genus $g$. If $C$ is not contained in a plane, then
\[
g \leq \frac{1}{4}d^2 - d + 1.
\]
\end{thm}

\begin{exercise}
Assume there exists an integral curve $C \subset \P^3$ of degree $d \geq 3$ and arithmetic genus $g > \tfrac{1}{4}d^2 - d + 1$ which is not contained in a plane.
\begin{enumerate}
\item Compute that the ideal sheaf $\II_C$ satisfies
\[
\ch(\II_C) = (1,0,-d,2d + g - 1).
\]
\emph{Hint: Grothendieck-Riemann-Roch!}
\item Show that $\II_C$ has to be destabilized by an exact sequence
\[
0 \to F \to \II_C \to G \to 0,
\]
where $\ch_0(F) = 1$. \emph{Hint: Use Lemma \ref{lem:higherRankBound} to show that any wall of higher rank has to be contained inside $Q_{\alpha, \beta}(\II_C) < 0$.}
\item Show that $\II_C$ has to be destabilized by a sequence of the form
\[
0 \to \OO(-a) \to \II_C \to G \to 0,
\]
where $a > 0$ is a positive integer. \emph{Hint: Show that if the subobject is an actual ideal sheaf and not a line bundle, then integrality of $C$ implies that $\II_C$ is destabilized by a line bundle giving a bigger or equal wall.}
\item Show that the wall is inside $Q_{\alpha, \beta}(\II_C) < 0$ if $a > 2$. Moreover, if $d = 3$ or $d = 4$, then the same holds additionally for $a = 2$.
\item Derive a contradiction in the cases $d = 3$ and $d = 4$.
\item Let $E \in \Coh^{\beta}(\P^3)$ be a tilt semistable object for some $\alpha > 0$ and $\beta \in \R$ such that $\ch(E) = (0,2,d,e)$ for some $d$, $e$. Show that the inequality 
\[
e \leq \frac{d^2}{4} + \frac{1}{3}
\]
holds. \emph{Hint: Use Theorem \ref{thm:conj_p3} and Lemma \ref{lem:higherRankBound}.}
\item Obtain a contradiction via the quotient $G$ in $0 \to \OO(-2) \to \II_C \to G \to 0$.
\end{enumerate}
\end{exercise}


\appendix

\section{Background on Derived Categories}
\label{sec:derived_categories}

This section contains definition and important properties of the bounded derived category $\Db(\AA)$ for an abelian category $\AA$. Most of the time the category $\AA$ will be the category $\Coh(X)$ of coherent sheaves on a smooth projective variety $X$. To simplify notation $\Db(X)$ will be written for $\Db(\Coh(X))$. Derived categories were introduced by Verdier in his thesis under the supervision of Grothendieck. The interested reader can find a detailed account of the theory in \cite{GM03:homological_algebra}, the first two chapters of \cite{Huy06:fm_transforms} or the original source \cite{Ver67:verdier_thesis}. 

\begin{defn}
\begin{enumerate}
\item A complex
\[\ldots \to A^{i-1} \to A^i \to A^{i+1} \to \ldots \]
is called \emph{bounded} if $A^i = 0$ for both $i \gg 0$ and $ i \ll 0$.

\item The objects of the category $\Kom^b(\AA)$ are bounded complexes over $\AA$ and its morphisms are homomorphisms of complexes.

\item A morphism $f: A \to B$ in $\Kom^b(\AA)$ is called a \emph{quasi isomorphism} if the induced morphism of cohomology groups $H^i(A) \to H^i(B)$ is an isomorphism for all integers $i$.
\end{enumerate}
\end{defn}

The bounded derived category of $\AA$ is the localization of $\Kom^b(\AA)$ by quasi isomorphisms. The exact meaning of this is the next theorem.

\begin{thm}[{\cite[Theorem 2.10]{Huy06:fm_transforms}}]
There is a category $\Db(\AA)$ together with a functor $Q: \Kom^b(\AA) \to \Db(\AA)$ satisfying two properties.
\begin{enumerate}
\item The morphism $Q(f)$ is an isomorphism for any quasi-isomorphism $f$ in the category $\Kom^b(\AA)$.
\item Any functor $F: \Kom^b(\AA) \to \DD$ satisfying property (i) factors uniquely through $Q$, i.e., there is a unique (up to natural isomorphism) functor $G: \Db(\AA) \to \DD$ such that $F$ is naturally isomorphic to $G \circ Q$.
\end{enumerate}
\end{thm}

In particular, $Q$ identifies objects in $\Kom^b(\AA)$ and $\Db(\AA)$. By the definition of quasi isomorphisms we still have well defined cohomology groups $H^i(A)$ for any $A \in \Db(\AA)$. The category $\AA$ is equivalent to the full subcategory of $\Db(\AA)$ consisting of those objects $A \in \AA$ that satisfy $H^i(A) = 0$ for all $i \neq 0$. In the next section we will learn that this is the simplest example of what is known as the heart of a bounded t-structure.

Notice, there is the automorphism $[1]: \Db(\AA) \to \Db(\AA)$, where $E[1]$ is defined by $E[1]^i = E^{i+1}$. It simply changes the grading of a complex. Moreover, we define the shift functor $[n] = [1]^n$ for any integer $n$. The following lemma will be used to actually compute homomorphisms in the derived category.

\begin{lem}[{\cite[Proposition 2.56]{Huy06:fm_transforms}}]
Let $\AA$ be either an abelian category with enough injectives or $\Coh(X)$ for a smooth projective variety $X$. For any $A,B \in \AA$ and $i \in \Z$ we have the equality
\[\Hom_{\Db(\AA)}(A,B[i]) = \Ext^i(A, B).\]
\end{lem}

In contrast to $\Kom^b(\AA)$ the bounded derived category $\Db(\AA)$ is not abelian. This lead Verdier and Grothendieck to the notion of a triangulated category which will be explained in the next theorem.

\begin{defn}
For any morphism $f: A \to B$ in $\Kom^b(\AA)$ the cone $C(f)$ is defined by $C(f)^i = A^{i+1} \oplus B^i$. The differential is given by the matrix
\[\left(\begin{matrix}
-d_A & 0 \\
f & d_B
\end{matrix}\right).\]
\end{defn}

The inclusion $B^i \into A^{i+1} \oplus B^i$ leads to a morphism $B \to C(f)$ and the projection $B^i \oplus A^{i+1} \onto  A^{i+1}$ leads to a morphism $C(f) \to A[1]$.

\begin{defn}
A sequence of maps $F \to E \to G \to F[1]$ in $\Db(\AA)$ is called a \emph{distinguished triangle} if there is a morphism $f: A \to B$ in $\Kom^b(\AA)$ and a commutative diagram with vertical isomorphisms in $\Db(\AA)$ as follows
\[
\xymatrix{
F \ar[r] \ar[d] & E \ar[r] \ar[d] & G \ar[r] \ar[d] & F[1] \ar[d] \\
A \ar[r] & B \ar[r] & C(f) \ar[r] & A[1].
}
\]
\end{defn}

These distinguished triangles should be viewed as the analogue of exact sequences in an abelian category. If $0 \to A \to B \to C \to 0$ is an exact sequence in $\AA$, then $A \to B \to C \to A[1]$ is a distinguished triangle where the map $C \to A[1]$ is determined by the element in $\Hom(C, A[1]) = \Ext^1(C,A)$ that determines the extension $B$. The following properties of the derived category are essentially the defining properties of a \emph{triangulated category}.

\begin{thm}[{\cite[Chapter IV]{GM03:homological_algebra}}]
\begin{enumerate}
\item Any morphism $A \to B$ in $\Db(\AA)$ can be completed to a distinguished triangle $A \to B \to C \to A[1]$.
\item A triangle $A \to B \to C \to A[1]$ is distinguished if and only if the induced triangle $B \to C \to A[1] \to B[1]$ is distinguished.
\item Assume we have two distinguished triangles with morphisms $f$ and $g$ making the diagram below commutative.
\[
\xymatrix{
A \ar[r] \ar[d]^f & B \ar[r] \ar[d]^g & C \ar[r] \ar@{-->}[d]^{\exists h} & A[1] \ar[d]^{f[1]} \\
A' \ar[r] & B' \ar[r] & C' \ar[r] & A'[1].
}
\]
Then we can find $h: C \to C'$ making the whole diagram commutative.
\item Assume we have two morphisms $A \to B$ and $B \to C$. Then together with (1) and (3) we can get a commutative diagram as follows where all rows and columns are distinguished triangles.
\[
\xymatrix{
A \ar[r] \ar[d]^{\id} & B \ar[r] \ar[d] & D \ar[r] \ar@{-->}[d]^{\exists} & A[1] \ar[d]^{\id[1]} \\
A \ar[r] \ar[d] & C \ar[r] \ar[d] & E \ar[r] \ar@{-->}[d]^{\exists} & A[1] \ar[d] \\
0 \ar[r] \ar[d] & F \ar[r]^{\id} \ar[d] & F \ar[r] \ar@{-->}[d]^{\exists} & 0 \ar[d] \\
A[1] \ar[r] & B[1] \ar[r] & D[1] \ar[r] & A[2]
}
\]
\end{enumerate}
\end{thm}

The key in property (4) is that the triangle $D \to E \to F \to D[1]$ is actually distinguished. Be aware that contrary to most definitions in category theory the morphism in (3) is not necessarily unique.

\begin{exercise}
Let $A \to B \to C \to A[1]$ be a distinguished triangle and $E \in \Db(\AA)$ be an arbitrary object. Then there are long exact sequences
\[
\ldots \to \Hom(E, C[-1]) \to \Hom(E, A) \to \Hom(E, B) \to \Hom(E, C) \to  \Hom(E, A[1]) \to \ldots
\]
and
\[
\ldots \to \Hom(A[-1], E) \to \Hom(C, E) \to \Hom(B, E) \to \Hom(A, E) \to \Hom(C[-1], E) \to \ldots
\]
Show the existence of one of the two long exact sequences (their proofs are almost the same).
\end{exercise}

\begin{exercise}
Let $f: A \to B$ be a morphism in $\Db(\AA)$. Show that $f$ is an isomorphism if and only if $C = 0$.
\end{exercise}

\begin{exercise}
Prove the corresponding statement to the Five Lemma for derived categories: Assume there is a commutative diagram between distinguished triangles
\[
\xymatrix{
A \ar[r] \ar[d]^f & B \ar[r] \ar[d]^g & C \ar[r] \ar[d]^h & A[1] \ar[d]^{f[1]} \\
A' \ar[r] & B' \ar[r] & C' \ar[r] & A'[1].
}
\]
If two of the morphisms $f,g,h$ are isomorphisms, so is the third one.
\end{exercise}

We will need the following technical statement in the main text.

\begin{prop}[{\cite[Proposition 5.4]{BM02:fm_elliptic_fibrations}}]
\label{prop:locally_free_complex}
Let $X$ be a smooth projective variety and $E \in \Db(X)$. If $\Ext^i(E, \C(x)) = 0$ for all $x \in X$ and $i < 0$ or $i > s \in \Z$. Then $E$ is isomorphic to a complex $F^{\bullet}$ of locally free sheaves such that $F^i = 0$ for $i > 0$ and $i < -s$.
\end{prop}

All triangles coming up in this article are distinguished. Therefore, we will simply drop the word distinguished from the notation.


\newcommand{\etalchar}[1]{$^{#1}$}
\def\cprime{$'$} \def\cprime{$'$}

\end{document}